\documentclass[11pt, oneside]{article}   	
\usepackage{geometry}                		
\usepackage{graphicx}				
\usepackage{amssymb}
\usepackage{amsmath}
\usepackage{amsthm}
\usepackage{tikz}
\usepackage{enumerate}
\usetikzlibrary{matrix,arrows,decorations.pathmorphing}
\usepackage[mathscr]{euscript}
\usepackage{wasysym}
\usepackage{mathtools}
\usepackage{hyperref}
\usepackage{array}
\usepackage{comment}

\newcommand{\R}{\mathbb{R}}

\newtheorem{thm}{Theorem}[section]

\newtheorem{lemma}[thm]{Lemma}
\newtheorem{prop}[thm]{Proposition}

\theoremstyle{definition}
\newtheorem{defn}[thm]{Definition}
\newtheorem{ex}[thm]{Example}

\theoremstyle{remark}
\newtheorem*{Proof of Proposition 4.8}{Proof of Proposition 4.8}

\title{Plabic R-matrices}
\author{Sunita Chepuri}
\date{ }							

\begin{document}
\maketitle
\begin{abstract}
Postnikov's plabic graphs in a disk are used to parametrize totally positive Grassmannians. In recent years plabic graphs have found numerous applications in math and physics. One of the key features of the theory is the fact that if a plabic graph is reduced, the face weights can be uniquely recovered from boundary measurements. On surfaces more complicated than a disk this property is lost. In this paper we undertake a comprehensive study of a certain semi-local transformation of weights for plabic networks on a cylinder that preserve boundary measurements. We call this a plabic R-matrix. We show that plabic R-matrices have underlying cluster algebra structure, generalizing recent work of Inoue-Lam-Pylyavskyy. Special cases of transformations we consider include geometric R-matrices appearing in Berenstein-Kazhdan theory of geometric crystals, and also certain transformations appearing in a recent work of Goncharov-Shen.
\end{abstract}

\section{Introduction}\label{sec:intro}
The relationship between total positivity and networks has been studied extensively (see \cite{ALT}, \cite{CM}, \cite{FZ}).  In his groundbreaking paper~\cite{P}, Postnikov develops a theory of plabic networks for studying the connection between the totally nonnegative Grassmannian and planar directed networks in a disk.  Plabic graphs have since been found to have many additional applications.  They have been used by Kodama and Williams to study soliton solutions to the KP equation~\cite{KW1, KW2}, by Arkani-Hamed, et. al., to study scattering amplitudes for $\mathcal{N}=4$ supersymmetric Yang-Mills~\cite{ABCGPT, ABCPT, ABCTPbook}, and by Gekhtman, Shapiro, and Vainshtein to study Poisson geometry~\cite{GSVbook, GSV2}.

Postnikov defines a set of local moves and reductions so that the boundary measurement map gives a bijection between move-reduction equivalence classes for plabic networks in a disk and the totally nonnegative Grassmannian.  However, there are plabic networks on a cylinder that are not move-reduction equivalent and yet have the same boundary measurements.  In particular, we define a semi-local transformation on weights for plabic networks on a cylinder that preserves boundary measurements.  We call this a \emph{plabic R-matrix}.  Plabic R-matrices are different from Postnikov's moves and reductions in that they do not alter the underlying graph.  

In the simplest case (see Example~\ref{ex:LP-edge}), we recover the geometric R-matrix.  The geometric R-matrix arises in Berenstein and Kazhdan's theory of geometric crystals~\cite{BK}.  It is studied by Etingof in exploring set-theoretical solutions to the Yang-Baxter equation~\cite{E}, and by Kajiwara, Nouni, and Yamada in the context of representations of affine Weyl groups~\cite{KNY}.  Kashiwara, Nakashima, and Okado give a thorough survey constructing the geometric R-matrix in different types~\cite{KNO}.  The geometric R-matrix also appears in Lam and Pylyavskyy's investigation of total positivity in loop groups~\cite{LP2}.

Another special case of the plabic R-matrix (see Examples~\ref{ex:GS-face},~\ref{ex:GS-y}, and~\ref{ex:GS-coincidence}) is a transformation used by Goncharov and Shen to study Donaldson-Thomas invariants~\cite{GS}.  Given a surface $\mathbb{S}$ with punctures, there is a birational Weyl group action on the moduli space $\mathcal{X}_{\text{PGLm},\mathbb{S}}$ for each puncture in $\mathbb{S}$.  Goncharov and Shen show that this action is given by cluster Poisson transformations.  We can obtain the plabic R-matrix from a generalization of these transformations (see Sections~\ref{sec:quivers} and~\ref{sec:R-mat-mutation}), and in specific cases recover the Weyl group action.

{\bf Section~\ref{sec:bg1}.} This section provides background on planar directed networks.  In Section~\ref{sec:bg1-disk}, we review Postinikov's boundary measurement map (Section 4 of~\cite{P}).  Section~\ref{sec:bg1-cyl} follows Gekhtman, Shapiro, and Vainshtein's lifting of this construction to planar directed networks on a cylinder \cite{GSV1}.

{\bf Section~\ref{sec:plabic-networks}.} Here we introduce plabic networks on a cylinder.  Section~\ref{sec:face-trail} details how to obtain face and trail weights from edge weights and how to calculate the weight of a path in a face weighted network.  In Section~\ref{sec:orientation}, we expose a major difficulty of lifting the theory of plabic networks to the cylinder: Postinikov's results (Section 10 of~\cite{P}) regarding changing the orientation of edges of a plabic network do not hold.  We do however show in Theorem~\ref{thm:orientation-invol} that plabic networks on a cylinder contain enough information that we can define involutions on the level of plabic networks, rather than planar directed networks.

{\bf Section~\ref{sec:R-mat}.} We begin by reviewing Postnikov's moves and reductions for plabic networks (Section 12 of~\cite{P}) and Postnikov diagrams, also known as alternating strand diagrams (Section 14 of~\cite{P}).  The definition of Postnikov diagrams can be generalized so that it lifts to a cylinder.  We prove in Theorem~\ref{thm:alt-str-reduced-gen} that Postnikov diagrams on a cylinder are in bijection with leafless reduced plabic graphs on a cylinder with no unicolored edges, an analogue of Corollary 14.2 of~\cite{P}.  Using alternating stand diagrams, we introduce a family of plabic networks on a cylinder called cylindric $k$-loop plabic networks.  We define an transformation called a plabic R-matrix on weights for such networks, in both an edge-weighted and face-weighted setting.  In Theorems~\ref{thm:Te} and~\ref{thm:Tf}, we show that plabic R-matrices preserve boundary measurements, are an involutions, give the only choices of weights that preserve the boundary measurements, and satisfy the braid relation.

{\bf Section~\ref{sec:bg2}.} This section gives a brief background on cluster algebras from quivers, including both $x$- and $y$-dynamics.

{\bf Section~\ref{sec:quivers}.} In this section, we define spider web quivers and a mutation sequence $\tau$ for these quivers.  Proposition~\ref{prop:mutation-seq} shows $\tau$ preserves the original quiver.  We give formulas for both the $x$ and $y$ variables after applying $\tau$ to a spider web quiver.  In Theorems~\ref{thm:x-involution},~\ref{thm:x-braid},~\ref{thm:y-involution}, and~\ref{thm:y-braid}, we show that $\tau$ is an involution and satisfies the braid relation for both $x$ and $y$ variables.

{\bf Section~\ref{sec:R-mat-mutation}.}  We show in Theorem~\ref{thm:coincidence} that for a cylindric $k$-loop plabic network the face weighted plabic R-matrix is realized by the $y$-dynamics of the dual quiver, under the involution $\tau$ from Section~\ref{sec:quivers}.

The rest of the sections provide proofs of the main theorems.  Section~\ref{sec:AltStr} proves some useful facts about plabic graphs on a cylinder using Postnikov diagrams, Section~\ref{sec:R-matThm} gives proofs about plabic R-matrices, and Section~\ref{sec:ClAlgPfs} contains proofs of the theorems from Section~\ref{sec:quivers}.

{\bf Acknowledgements.}  I would like to thank my adviser Pavlo Pylyavskyy for introducing me to this topic and for the many helpful conversations and suggestions.  I am also thankful for the support of RTG NSF grant DMS-1148634. 

\section{Planar directed networks}\label{sec:bg1}

\begin{defn}\label{defn:planar-directed-network}
We will assume a \emph{planar directed graph} on a surface with boundary, considered up to homotopy, has $n$ vertices on the boundary, $b_1,...,b_n$.  We will call these \emph{boundary vertices}, and all other vertices \emph{internal vertices}.  Additionally, we will assume that all boundary vertices are sources or sinks.  A \emph{planar directed network} is a planar directed graph with a weight $x_e\in\R_{>0}$ assigned to each edge.
\end{defn}

\begin{defn}\label{defn:source-sink}
The \emph{source set} of a planar directed graph or network is the set $I=\{i\in [n]\ |\ b_i\text{ is a source}\}$.  The \emph{sink set} is $\overline{I}=[n]\setminus I=\{j\in [n]\ |\ b_j\text{ is a sink}\}$.  
\end{defn}

\begin{defn}\label{defn:xP}
For any path $P$ in a planar directed network, the weight of $P$ is $$x_P=\prod_{e\in P}x_e.$$
\end{defn}

\subsection{Planar directed networks in a disk}\label{sec:bg1-disk}

The material in this section can be found in Section 4 of \cite{P}.

For a planar directed graph or network in a disk, we will label the boundary vertices $b_1,...,b_n$ in clockwise order.

\begin{defn}\label{defn:winding-index-disk}
For a path $P$ in a planar directed graph or network in a disk from $b_i$ to $b_j$, we define its \emph{winding index}, $wind(P)$.  First, we smooth any corners in $P$ and make the tangent vector of $P$ at $b_j$ have the same direction as the tangent vector at $b_i$.  Then $wind(P)$ is the full number of counterclockwise $360^\circ$ turns the tangent vector makes from $b_i$ to $b_j$, counting clockwise turns as negative.  For $C$ a closed curve, we can define $wind(C)$ similarly.  See Lemma 4.2 in~\cite{P} for a recursive formula for $wind(P)$ and $wind(C)$.
\end{defn}

\begin{figure}[htp]
\centering
\begin{tikzpicture}[scale=0.8]
\draw [line width=0.25mm] (0,0) circle (2cm);
\path[-,font=\large, >=angle 90, line width=0.5mm]
(-2,0) edge (0.2,-1)
(0.2,-1) edge (-0.2,0.5)
(-0.2,0.5) edge (-0.6,-1.5)
(-0.6,-1.5) edge (-0.6,1)
(-0.6,1) edge (1.2, 0)
(1.2, 0) edge (0.7, -0.1);
\path[->,font=\large, >=angle 90, line width=0.5mm]
(0.7, -0.1) edge (1.42,1.42);
\end{tikzpicture}
\caption{A path $P$ with $wind(P)=-1$.}
\label{fig:winding-disk}
\end{figure}

\begin{defn}\label{defn:formal-boundary-meas-disk}
Let $b_i$ be a source and $b_j$ be a sink in a planar directed network in a disk with graph $G$.  Let the edge weights be the formal variables $x_e$.  Then the \emph{formal boundary measurement} $M_{ij}^{\text{form}}$ is the formal power series $$M_{ij}^{\text{form}}:=\sum_{\substack{\text{paths }P\text{ from}\\b_i\text{ to }b_j}}(-1)^{wind(P)}x_P.$$
\end{defn}

\begin{lemma}\label{lemma:formal-boundary-meas-disk}
The formal power series $M_{ij}^{\text{form}}$ sum to subtraction-free rational expressions in the variables $x_e$.  Thus, $M_{ij}^{\text{form}}$ is well-defined function on $\R_{\geq0}^{|E(G)|}$, where $E(G)$ is the set of edges in the graph $G$.
\end{lemma}

\begin{defn}\label{defn:boundary-meas-disk}
The \emph{boundary measurements} $M_{ij}$ for a planar directed network in a disk are nonnegative real numbers obtained by writing the formal boundary measurements $M_{ij}^{\text{form}}$ as subtraction-free rational expressions, and then specializing them by assigning $x_e$ the real weight of the edge $e$.
\end{defn}

\begin{ex}\label{ex:boundary-meas-disk}
Suppose we have the following network:

\begin{tikzpicture}
\node(A) at (-3,0) {$b_1$};
\node(B) at (3,0) {$b_2$};
\path[->,font=\large, >=angle 90, line width=0.4mm]
(A) edge node[above] {$x_1=1$} (-1,0)
(-1,0) edge[bend left=50] node[above] {$x_2=2$} (1,0)
(1,0) edge[bend left=50] node[below] {$x_3=1$} (-1,0)
(1, 0) edge node[above] {$x_4=1$} (B);
\end{tikzpicture}
\begin{align*}
M_{12}^{\text{form}}&=x_1 x_2 x_4-x_1 x_2 x_3 x_2 x_4 +x_1 x_2 x_3 x_2 x_3 x_2 x_4 -...\\
&=x_1 x_2 x_4 \sum_{i=0}^\infty (-x_2 x_3)^i\\
&=\frac{x_1 x_2 x_4 }{1+x_2 x_3 }
\end{align*}

Substituting our values for the $x_e$'s, we find $M_{12}=\frac{2}{3}$.
\end{ex}

\begin{defn}\label{defn:Grassmannian}
For $0\leq k\leq n$, the \emph{Grassmannian} $Gr_{kn}$ is the manifold of $k$-dimensional subspaces of $\R^n$.
\end{defn}

We can associate any full-rank real $k\times n$ matrix to a point in $Gr_{kn}$ by taking the span of its rows.  Let $Mat^*_{kn}$ be the space of full-rank real $k\times n$ matrices.  Since left-multiplying a matrix by an element of the general linear group is equivalent to performing row operations and row operations do not change the row-span of a matrix, we can think of $Gr_{kn}$ as the quotient $GL_k\backslash Mat^*_{kn}$.

\begin{defn}\label{defn:minor}
For a $k\times n$ matrix $M$, a \emph{maximal minor} is $\Delta(M)_{S}$ where $S\subseteq[n]$ and $|S|=k$.  $\Delta(M)_S$ is the determinant of the submatrix of $M$ obtained by taking only the columns indexed by $S$.
\end{defn}

\begin{defn}\label{defn:boundary-meas-map-disk}
If $Net_{kn}$ is the set of planar directed networks in a disk with $k$ boundary sources and $n-k$ boundary sinks,  then we can define the \emph{boundary measurement map} $Meas: Net_{kn}\to Gr_{kn}$.  $Meas(N)$ is the point in $Gr_{kn}$ represented by the matrix $A(N)$, which is defined as follows:
\begin{enumerate}[(1)]
\item $A(N)_I$, the submatrix of $A(N)$ containing only the columns in the source set, is the identity $Id_k$.

\item For $I=\{i_1<...<i_k\},\ r\in[k]$, and $b_j\in\overline{I}$, we define $a_{rj}=(-1)^sM_{i_r,j}$, where $s$ is the number of elements of $I$ strictly between $i_r$ and $j$.
\end{enumerate}

Note that the map $Meas$ is constructed so that $M_{ij}=\Delta(A(N))_{(I\setminus\{i\})\cup\{j\}}$. 
\end{defn}

\begin{ex}
Consider the network from the pervious example:

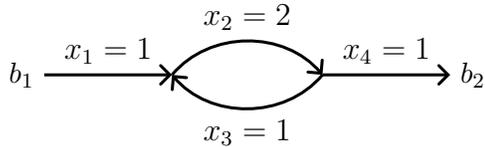
\begin{tikzpicture}
\node(A) at (-3,0) {$b_1$};
\node(B) at (3,0) {$b_2$};
\path[->,font=\large, >=angle 90, line width=0.4mm]
(A) edge node[above] {$x_1=1$} (-1,0)
(-1,0) edge[bend left=50] node[above] {$x_2=2$} (1,0)
(1,0) edge[bend left=50] node[below] {$x_3=1$} (-1,0)
(1, 0) edge node[above] {$x_4=1$} (B);
\end{tikzpicture}

In this case $I=\{1\}$, so we put $Id_1$ in the first column of $A(N)$.

We compute $a_{12}=(-1)^0M_{12}=M_{12}=\frac{2}{3}$.

So, we have $A(N)=\begin{bmatrix} 1 & \frac{2}{3} \end{bmatrix}$.
\end{ex}

\subsection{Planar directed networks on a cylinder}\label{sec:bg1-cyl}

Throughout this paper, we will draw a cylinder as a fundamental domain of its universal cover such that it is a rectangle with boundary components on the left and right (see Figure~\ref{fig:cylinder}).

\begin{figure}[htb]
\begin{center}
\begin{tikzpicture}[scale=2]
\fill[gray!40!white] (0,0) rectangle (1,1);
\path[-,font=\large, >=angle 90]
(0,0) edge (0,1)
(1,0) edge (1,1);
\path[dashed,font=\large, >=angle 90]
(0,0) edge (1,0)
(0,1) edge (1,1);
\end{tikzpicture}
\end{center}
\caption{A cylinder, as we will represent them in this paper.}
\label{fig:cylinder}
\end{figure}
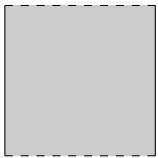

The constructions in this section may be found in~\cite{GSV1}.  They are based on Postnikov's theory of planar directed networks in a disk~\cite{P}, as seen in~\ref{sec:bg1-disk}.

For a planar directed graph or network on a cylinder, we will label the boundary vertices $b_1,...,b_n$ from the top of the left boundary component to the bottom and then from the bottom of the right boundary component to the top.

\begin{defn}\label{defn:cut}
A \emph{cut} $\gamma$ is an oriented non-self-intersecting curve from one boundary component to another, considered up to homotopy.  The endpoints of the cut are \emph{base points}.  We will always assume the cut is disjoint from the set of vertices of the graph and that it corresponds to the top and bottom of our rectangle when we draw a cylinder.  The cut is denoted by a directed dashed line.
\end{defn}

\begin{defn}\label{defn:intersection-number}
For a path $P$, the \emph{intersection number}, $int(P)$, is the number of times $P$ crosses $\gamma$ from the right minus the number where $P$ crosses $\gamma$ from the left.
\end{defn}

\begin{defn}\label{defn:CP}
If $P$ is a path from $b$ to $b'$ where $b,b'$ are on the same boundary component, then $C_P$ is the closed loop created from following the path $P$ and then going down along the boundary from $b'$ to $b$.  If $P$ is a path from $b$ to $b'$ where $b,b'$ are on the different boundary components, then $C_P$ is the closed loop created from following the path $P$ going down on the boundary from $b'$ to the base point of the cut, following the cut (or its reverse), and then down on the boundary from base point of the cut to $b$.
\end{defn}

\begin{defn}\label{defn:winding-index-cyl}
We can glue together the top and bottom of our rectangle, which represents a cylinder, in the plane to form an annulus.  Do this such that going up along the boundary of the rectangle corresponds to going clockwise around the boundary of the annulus (see Figure~\ref{fig:annulus}).  Then for a path $P$, the \emph{winding index} of $P$ is defined to be $wind(C_P)$, when $C_P$ is drawn on this annulus and $wind(C_P)$ is calculated as in Definition~\ref{defn:winding-index-disk}.
\end{defn}

\begin{figure}[htb]
\begin{center}
\begin{tikzpicture}
\fill[gray!40!white] (0,0) rectangle (1.5,1.5);
\path[-,font=\large, >=angle 90]
(0,0) edge (0,1.5)
(1.5,0) edge (1.5,1.5);
\path[dashed,font=\large, >=angle 90]
(0,0) edge (1.5,0)
(0,1.5) edge (1.5,1.5);
\path[->,font=\large, >=angle 90]
(2.5,1) edge (3,1);
\path [fill=gray!40!white] (4,0) to [out=90,in=180] (6,2) -- (6,1) to [out=180,in=90] (5,0);
\draw (4,0) arc [radius=2, start angle=180, end angle= 90];
\draw (5,0) arc [radius=1, start angle=180, end angle= 90];
\path[dashed,font=\large, >=angle 90]
(4,0) edge (5,0)
(6,2) edge (6,1);
\path[->,font=\large, >=angle 90]
(7,1) edge (7.5,1);
\fill[gray!40!white] (10,1.5) circle [radius=1.5];
\fill[white] (10,1.5) circle [radius=0.5];
\draw (10,1.5) circle [radius=1.5];
\draw (10,1.5) circle [radius=0.5];
\path[dashed,font=\large, >=angle 90]
(8.5,1.5) edge (9.5,1.5);
\end{tikzpicture}
\end{center}
\caption{Turning a cylinder into an annulus.}
\label{fig:annulus}
\end{figure}
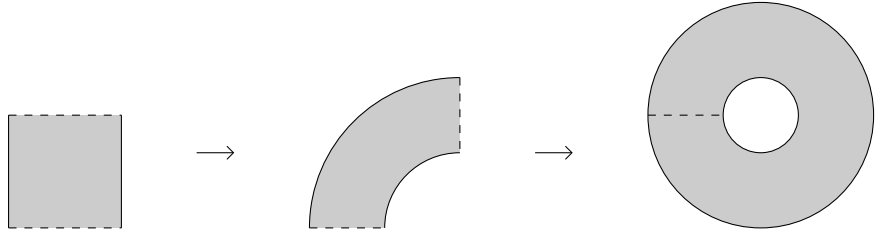

\vspace{2in}
\begin{ex}\label{ex:intersection-number-winding-index-cyl}
$ $
\begin{center}
\begin{tikzpicture}[scale=0.8]
\draw [line width=0.25mm] (0,0) circle (3cm);
\draw [line width=0.25mm] (0,0) circle (1cm);
\path[red, line width=0.5mm]
(-3,0) edge (-1,0);
\path[->, dashed, >=angle 90, line width=0.5mm]
(-3,0) edge (-1,0);
\path[-,font=\large, >=angle 90, line width=0.5mm, blue]
(0,3) edge (-1.5,1)
(-1.5,1) edge (0.5,2)
(0.5,2) edge (-2,1.5)
(-2,1.5) edge (-0.5,-2)
(-0.5,-2) edge (1.5,-1.5)
(1.5,-1.5) edge (1,-1)
(1,-1) edge (0.5,-2.5);
\path[->,font=\large, >=angle 90, line width=0.5mm, blue]
(0.5,-2.5) edge (0,-1);
\draw[line width=0.5mm, red] ([shift=(-90:1cm)]0,0) arc (-90:180:1cm);
\draw[line width=0.5mm, red] ([shift=(-180:3cm)]0,0) arc (-180:90:3cm);
\end{tikzpicture}
\end{center}

Here we have the cylinder depicted as an annulus.  The black dashed line is the cut.  A path $P$ is shown in blue.  $P$ crosses the cut once from left to right, so $int(P)=-1$.  The extension of $P$ to $C_P$ is shown in red.  We can see $wind(C_P)=-3$.
\end{ex}

\begin{defn}\label{defn:formal-boundary-meas-cyl}
Let $b_i$ be a source and $b_j$ be a sink in a planar directed network on a cylinder with graph $G$.  Let the edge weights be the formal variables $x_e$.  Then the \emph{formal boundary measurement} $M_{ij}^{\text{form}}$ is the formal power series $$M_{ij}^{\text{form}}:=\sum_{\substack{\text{paths }P\text{ from}\\b_i\text{ to }b_j}}(-1)^{wind(C_P)-1}\zeta^{int(P)}x_P.$$
\end{defn}

\begin{lemma}[Corollary 2.3 of~\cite{GSV1}]\label{lemma:formal-boundary-meas-cyl}
If $N$ is a planar network on a cylinder, then the formal power series $M_{ij}^{\text{form}}$ sum to rational expressions in the variables $x_e$ and $\zeta$.
\end{lemma}

\begin{defn}\label{defn:boundary-meas-cyl}
The \emph{boundary measurements} $M_{ij}$ for a planar directed network on a cylinder are rational functions in $\zeta$ obtained by writing the formal boundary measurements $M_{ij}^{\text{form}}$ as rational expressions, and then specializing them by assigning $x_e$ the real weight of the edge $e$.
\end{defn}

\begin{ex}\label{ex:boundary-meas-cyl}
Suppose we have the following network:

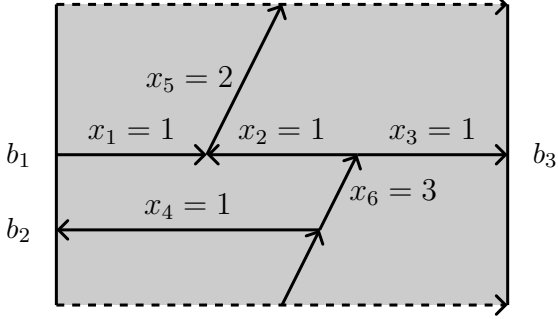
\begin{tikzpicture}
\fill[gray!40!white] (0,0) rectangle (6,4);
\node at (-0.5,2) {$b_1$};
\node at (-0.5,1) {$b_2$};
\node at (6.5,2) {$b_3$};
\path[->,font=\large, >=angle 90, line width=0.4mm]
(0,2) edge node[above] {$x_1=1$} (2,2)
(4,2) edge node[above] {$x_2=1$} (2,2)
(4,2) edge node[above] {$x_3=1$} (6,2)
(3.5,1) edge node[above] {$x_4=1$} (0,1)
(2,2) edge node[left] {$x_5=2$} (3,4)
(3,0) edge (3.5,1)
(3.5,1) edge node[right] {$x_6=3$} (4,2);
\path[-,font=\large, >=angle 90, line width=0.4mm]
(0,0) edge (0,4)
(6,0) edge (6,4);
\path[->,dashed,font=\large, >=angle 90, line width=0.4mm]
(0,0) edge (6,0)
(0,4) edge (6,4);
\end{tikzpicture}
\begin{align*}
M_{12}^{\text{form}}&=x_1x_5x_4\zeta-x_1x_5x_6x_2x_5x_4\zeta^2+x_1x_5x_6x_2x_5x_6x_2x_5x_4\zeta^3...\\
&=x_1x_5x_6\zeta\sum_{i=0}^\infty(-x_5x_6x_2\zeta)^i\\
&=\frac{x_1x_5x_6\zeta}{1+x_5x_6x_2\zeta}\\
M_{13}^{\text{form}}&=-x_1x_5x_6x_3\zeta+x_1x_5x_6x_2x_5x_6x_3\zeta^2-x_1x_5x_6x_2x_5x_6x_2x_5x_6x_3\zeta^3...\\
&=-x_1x_5x_6x_3\zeta\sum_{i=0}^\infty(-x_5x_6x_2\zeta)^i\\
&=\frac{-x_1x_5x_6x_3\zeta}{1+x_5x_6x_2\zeta}
\end{align*}

Substituting our values for the $x_e$'s, we find $M_{12}=\frac{2\zeta}{1+6\zeta}$ and $M_{13}=\frac{-6\zeta}{1+6\zeta}$.

\end{ex}

\begin{defn}\label{defn:Gr-loops}
The \emph{space of Grassmannian loops}, $LGr_{kn}$ is the space of rational functions $X:\R\to Gr_{kn}$.  Elements in $LGr_{kn}$ can be represented as a full-rank $k\times n$ matrix where the entries are functions of a parameter $\zeta$.
\end{defn}

\begin{defn}\label{defn:boundary-mean-map-cyl}
If $Net_{kn}^C$ is the set of planar directed networks on a cylinder with $k$ boundary sources and $n-k$ boundary sinks, we can define the \emph{boundary measurement map} as $Meas^C:Net_{kn}^C\to LG_{kn}$ where $Meas^C(N)$ is represented by the matrix $A(N)$ such that:
\begin{enumerate}[(1)]
\item $A(N)_I$ is the identity $Id_k$.

\item For $I=\{i_1<...<i_k\}, r\in[k]$, and $b_j\in\overline{I}$, we define $a_{rj}=(-1)^sM_{i_r,j}$, where $s$ is the number of elements of $I$ strictly between $i_r$ and $j$.
\end{enumerate}

Note that the map $Meas^C$ is constructed so that $M_{ij}=\Delta(A(N))_{(I\setminus\{i\})\cup\{j\}}$.
\end{defn}

\begin{ex}
Consider the network from the previous example:

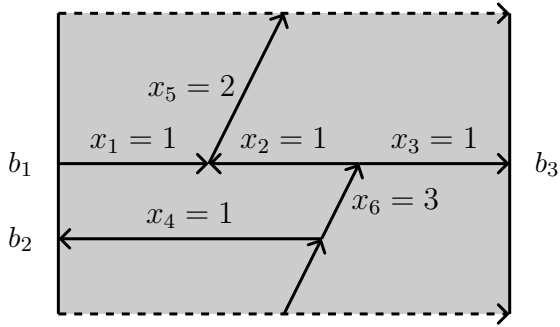
\begin{tikzpicture}
\fill[gray!40!white] (0,0) rectangle (6,4);
\node at (-0.5,2) {$b_1$};
\node at (-0.5,1) {$b_2$};
\node at (6.5,2) {$b_3$};
\path[->,font=\large, >=angle 90, line width=0.4mm]
(0,2) edge node[above] {$x_1=1$} (2,2)
(4,2) edge node[above] {$x_2=1$} (2,2)
(4,2) edge node[above] {$x_3=1$} (6,2)
(3.5,1) edge node[above] {$x_4=1$} (0,1)
(2,2) edge node[left] {$x_5=2$} (3,4)
(3,0) edge (3.5,1)
(3.5,1) edge node[right] {$x_6=3$} (4,2);
\path[-,font=\large, >=angle 90, line width=0.4mm]
(0,0) edge (0,4)
(6,0) edge (6,4);
\path[->,dashed,font=\large, >=angle 90, line width=0.4mm]
(0,0) edge (6,0)
(0,4) edge (6,4);
\end{tikzpicture}

In this case $I=\{1\}$, so we put $Id_1$ in the first column of $A(N)$.

We compute $a_{12}=(-1)^0M_{12}=M_{12}=\frac{2\zeta}{1+6\zeta}$ and $a_{13}=(-1)^0M_{13}=M_{13}=\frac{-6\zeta}{1+6\zeta}$.

So, we have $A(N)=\begin{bmatrix} 1 & \frac{2\zeta}{1+6\zeta} & \frac{-6\zeta}{1+6\zeta} \end{bmatrix}$.
\end{ex}

\section{Plabic networks}\label{sec:plabic-networks}

\begin{defn}\label{defn:plabic-graph}
A \emph{planar bicolored graph}, or \emph{plabic graph}, on a surface with boundary is a planar undirected graph such that each boundary vertex has degree 1 and each internal vertex is colored black or white.
\end{defn}

\begin{defn}[Definition 11.5 of~\cite{P}]\label{defn:plabic-network-disk}
A \emph{plabic network} in a disk is a plabic graph with a weight $y_f\in\R_{>0}$ assigned to each face.
\end{defn}

Postnikov~\cite{P} solves the inverse boundary problem for planar directed networks in a disk by turning them into plabic networks.  We will approach the problem for planar directed networks on a cylinder in the same way.

\subsection{Face and trail weights}\label{sec:face-trail}

\begin{defn}[Section 4 of~\cite{P}]\label{defn:gauge-trans}
A \emph{gauge transformation} is a rescaling of edge weights in a planar directed network so that all incoming edges of a particular vertex are multiplied by a positive real number $c$ and all outgoing edges of that vertex are divided by $c$.
\end{defn}

It is clear that gauge transformations preserve the boundary measurements, as they preserve the weight of each path.  This means that we can only ever hope to solve the inverse boundary problem up to gauge transformations.  To this end, we introduce the space of face and trail weights, which eliminates gauge transformations.  This space was introduced by Gekhtman, Shapiro, and Vainshtein in~\cite{GSV1}.  Here we present their results and also state explicitly how to obtain boundary measurements from a face weighted planar directed network on a cylinder.

\begin{defn}[Section 11 of~\cite{P}]\label{defn:face-weight-cyl}
For a face in a planar directed network, define the \emph{face weight} to be $$y_f:= \left(\prod_{e_i\in I^+_f} x_{e_i}\right)\left(\prod_{e_j\in I^-_f} x_{e_j}\right)$$ where $I^+_f$ is the set of edges on the outer boundary of $f$ oriented clockwise and edges on the inner boundary (if $f$ is not simply connected) oriented counterclockwise, and $I^-_f$ is the set of edges on the outer boundary of $f$ oriented counterclockwise and edges on the inner boundary (if $f$ is not simply connected) oriented clockwise.
\end{defn}

\begin{figure}
\begin{center}
\begin{tikzpicture}
\path[->,font=\large, >=angle 90, line width=0.4mm, every loop/.style={min distance=12mm,in=0,out=60,looseness=10}]
(-2,0) edge node[left] {$x_1$} (0,2)
(2,0) edge node[right] {$x_2$} (0,2)
(0,-2) edge node[right] {$x_3$} (2,0)
(-2,0) edge node[left] {$x_4$} (0,-2)
(-1,0) edge node[above] {$x_5$} (0,0)
(0,0) edge [loop right] node[above] {$x_6$} (0,0)
(-3,0) edge (-2,0)
(3,0) edge (2,0)
(0,2) edge (0,3)
(0,-2) edge (0,-3);
\node at (0.3,-0.7) {$f$};
\end{tikzpicture}
\caption{A face with $y_f=x_1 x_2^{-1} x_3^{-1} x_4^{-1} x_5 x_6^{-1} x_5^{-1}=\frac{x_1}{x_2 x_3 x_4 x_6}$.}
\label{fig:face-weight}
\end{center}
\end{figure}
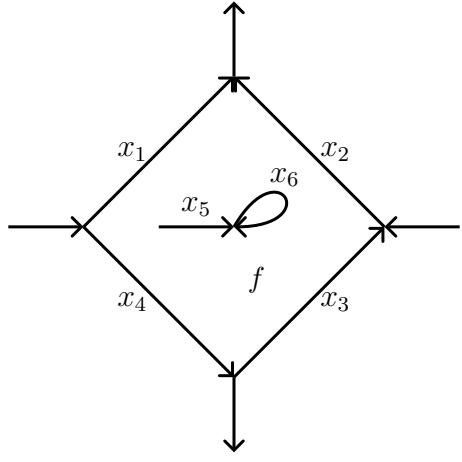

We can see that the product of weights of all the faces is 1, as each edge is counted once going clockwise and once going counterclockwise.

\begin{defn}[Section 2.3 of~\cite{GSV1}]\label{defn:trail}
A \emph{trail} in a planar directed network on a cylinder $N$ is a sequence of vertices $v_1,...,v_{m+1}$ where $v_1, v_{m+1}$ are boundary vertices on different boundary components and for each $i$, either $(v_i,v_{i+1})$ or $(v_{i+1},v_i)$ is an edge in $N$.  The \emph{weight} of a trail is $$t=\left(\prod_{(v_i,v_{i+1})\text{ an edge in }N} x_{(v_i,v_{i+1})}\right)\left(\prod_{(v_{i+1},v_i)\text{ an edge in }N} x^{-1}_{(v_{i+1},v_i)}\right).$$
\end{defn}

Notice that the face and trail weights are invariant under gauge transformations.

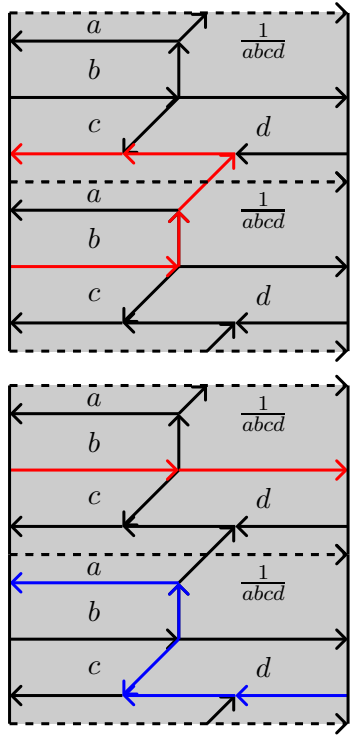
\begin{thm}[Section 2.3 of~\cite{GSV1}]\label{thm:mod-gauge-cyl}
For a planar directed graph $G$ on a cylinder with edge set $E$ and face set $F$, $$\R_{>0}^E/\{\text{gauge transformations}\}=\begin{cases} \R_{>0}^{F-1} & \text{if there is no trail,}\\ \R_{>0}^{F-1}\oplus\R_{>0} & \text{otherwise.}\end{cases}$$ where $\R_{>0}^{F-1}$ is generated by the face weights under the relation that their product be equal to 1 and $\R_{>0}$ in the second case is generated by the weight of a trail.
\end{thm}

Having a network that falls under the first case is equivalent to having a network in a disk.  In this case, we recover Postnikov's face weight construction for planar directed networks in a disk. 

\begin{defn}\label{defn:face-weight-path}
We will define $wt(P,y,t)$ in three cases:
\begin{enumerate}[(1)]
\item For a path $P$ that begins and ends on the same boundary component, draw enough copies of the fundamental domain that we can draw $P$ as a connected curve.  If $P$ is path from $b_i$ to $b_j$, then $P$ along with a segment of the boundary between $b_i$ and $b_j$ form a closed shape $P_b$ on the universal cover.  When $P_b$ is to the right of $P$, $wt(P,y,t)$ is the product of the weights of the faces in the interior of $P_b$.  When $P_b$ is to the left of $P$, $wt(P,y,t)$ is the inverse of this product.

\item For a path $P$ that begins on the same boundary component as the trail and ends on the other boundary component, draw enough copies of the fundamental domain we that we can draw $P$ as a connected curve and that there is at least one copy of the trail that lies completely to the right of $P$.  Then $wt(P,y,t)$ is the product of the weights of the faces that lie to the right of $P$ and to the left of a copy of the trail that is completely to the right of $P$ times the weight of the trail.

\item For a path $P$ that begins on the boundary component where the trail ends and ends on the other boundary component, draw enough copies of the fundamental domain we that we can draw $P$ as a connected curve and that there is at least one copy of the trail that lies completely to the right of $P$.  Then $wt(P,y,t)$ is the product of the weights of the faces that lie to the right of $P$ and to the right of a copy of the trail that is completely to the right of $P$ times the inverse of the weight of the trail.
\end{enumerate}

In Cases 2 and 3, $wt(P,y,t)$ is well-defined because if we choose two trails that lie completely to the right of $P$, the product of the weights of faces between the trails is 1.
\end{defn}

\begin{ex}\label{ex:face-weight-cyl}
Suppose we have the following network with face and trail weights:

\begin{tabular}[t]{ c m{0.3cm} m{6cm} }
\raisebox{\dimexpr-\height + 7ex\relax}{\begin{tikzpicture}[scale=0.75]
\fill[gray!40!white] (0,0) rectangle (6,3);
\node at (1.5,2.75) {$a$};
\node at (1.5,2) {$b$};
\node at (1.5,1) {$c$};
\node at (4.5,2.5) {$\frac{1}{abcd}$};
\node at (4.5,1) {$d$};
\node at (6.5,0.5) {\textcolor{blue}{$t$}};
\node at (-0.05,0.5) {\textcolor{blue}{ }};
\path[->,font=\large, >=angle 90, line width=0.4mm]
(2,0.5) edge (0,0.5)
(0,1.5) edge (3,1.5)
(3,1.5) edge (6,1.5)
(3,2.5) edge (3.5,3)
(3.5,0) edge (4,0.5);
\path[->,font=\large, >=angle 90, line width=0.4mm,blue]
(6,0.5) edge (4,0.5)
(4, 0.5) edge (2,0.5)
(3,1.5) edge (2,0.5)
(3,1.5) edge (3,2.5)
(3,2.5) edge (0,2.5);
\path[-,font=\large, >=angle 90, line width=0.4mm]
(0,0) edge (0,3)
(6,0) edge (6,3);
\path[dashed,->,font=\large, >=angle 90, line width=0.4mm]
(0,0) edge (6,0)
(0,3) edge (6,3);
\end{tikzpicture}}
&
&
The trail and trail weight appear in blue, and the trail is oriented from right to left.
\end{tabular}

\vspace{0.1in}
\begin{tabular}[t]{ c m{0.3cm} m{6cm} }
\raisebox{\dimexpr-\height + 15ex\relax}{
\begin{tikzpicture}[scale=0.75]
\fill[gray!40!white] (0,3) rectangle (6,6);
\node at (1.5,5.75) {$a$};
\node at (1.5,5) {$b$};
\node at (1.5,4) {$c$};
\node at (4.5,5.5) {$\frac{1}{abcd}$};
\node at (4.5,4) {$d$};
\path[->,font=\large, >=angle 90, line width=0.4mm]
(0,4.5) edge (3,4.5)
(3,4.5) edge (6,4.5)
(3,5.5) edge (3.5,6)
(6,3.5) edge (4,3.5)
(3,4.5) edge (2,3.5)
(3,4.5) edge (3,5.5)
(3,5.5) edge (0,5.5);
\path[->,font=\large, >=angle 90, line width=0.4mm,red]
(4, 3.5) edge (2,3.5)
(2,3.5) edge (0,3.5);
\path[-,font=\large, >=angle 90, line width=0.4mm]
(0,3) edge (0,6)
(6,3) edge (6,6);
\path[dashed,->,font=\large, >=angle 90, line width=0.4mm]
(0,3) edge (6,3)
(0,6) edge (6,6);
\fill[gray!40!white] (0,0) rectangle (6,3);
\node at (1.5,2.75) {$a$};
\node at (1.5,2) {$b$};
\node at (1.5,1) {$c$};
\node at (4.5,2.5) {$\frac{1}{abcd}$};
\node at (4.5,1) {$d$};
\path[->,font=\large, >=angle 90, line width=0.4mm]
(2,0.5) edge (0,0.5)
(3,1.5) edge (6,1.5)
(3.5,0) edge (4,0.5)
(6,0.5) edge (4,0.5)
(4, 0.5) edge (2,0.5)
(3,1.5) edge (2,0.5)
(3,1.5) edge (3,2.5)
(3,2.5) edge (0,2.5);
\path[->,font=\large, >=angle 90, line width=0.4mm,red]
(0,1.5) edge (3,1.5)
(3,1.5) edge (3,2.5)
(3,2.5) edge (4,3.5);
\path[-,font=\large, >=angle 90, line width=0.4mm]
(0,0) edge (0,3)
(6,0) edge (6,3);
\path[dashed,->,font=\large, >=angle 90, line width=0.4mm]
(0,0) edge (6,0)
(0,3) edge (6,3);
\end{tikzpicture}}
&
&
For the path $P$ shown in red, the interior of $P_b$ is to the left of $P$.  So, $wt(P,y,t)=\frac{1}{ab}$.
\end{tabular}

\vspace{0.1in}
\begin{tabular}[t]{ c m{0.3cm} m{6cm} }
\raisebox{\dimexpr-\height + 15ex\relax}{
\begin{tikzpicture}[scale=0.75]
\fill[gray!40!white] (0,3) rectangle (6,6);
\node at (1.5,5.75) {$a$};
\node at (1.5,5) {$b$};
\node at (1.5,4) {$c$};
\node at (4.5,5.5) {$\frac{1}{abcd}$};
\node at (4.5,4) {$d$};
\path[->,font=\large, >=angle 90, line width=0.4mm]
(3,5.5) edge (3.5,6)
(6,3.5) edge (4,3.5)
(3,4.5) edge (2,3.5)
(3,4.5) edge (3,5.5)
(3,5.5) edge (0,5.5)
(4, 3.5) edge (2,3.5)
(2,3.5) edge (0,3.5);
\path[->,font=\large, >=angle 90, line width=0.4mm,red]
(0,4.5) edge (3,4.5)
(3,4.5) edge (6,4.5);
\path[-,font=\large, >=angle 90, line width=0.4mm]
(0,3) edge (0,6)
(6,3) edge (6,6);
\path[dashed,->,font=\large, >=angle 90, line width=0.4mm]
(0,3) edge (6,3)
(0,6) edge (6,6);
\fill[gray!40!white] (0,0) rectangle (6,3);
\node at (1.5,2.75) {$a$};
\node at (1.5,2) {$b$};
\node at (1.5,1) {$c$};
\node at (4.5,2.5) {$\frac{1}{abcd}$};
\node at (4.5,1) {$d$};
\path[->,font=\large, >=angle 90, line width=0.4mm]
(2,0.5) edge (0,0.5)
(3,1.5) edge (6,1.5)
(3.5,0) edge (4,0.5)
(3,1.5) edge (3,2.5)
(0,1.5) edge (3,1.5)
(3,2.5) edge (4,3.5);
\path[->,font=\large, >=angle 90, line width=0.4mm,blue]
(6,0.5) edge (4,0.5)
(4, 0.5) edge (2,0.5)
(3,1.5) edge (2,0.5)
(3,1.5) edge (3,2.5)
(3,2.5) edge (0,2.5);
\path[-,font=\large, >=angle 90, line width=0.4mm]
(0,0) edge (0,3)
(6,0) edge (6,3);
\path[dashed,->,font=\large, >=angle 90, line width=0.4mm]
(0,0) edge (6,0)
(0,3) edge (6,3);
\end{tikzpicture}}
&
&
For the path $P$ shown in red, $P$ is going in the opposite direction of the trail.  So, $wt(P,y,t)=\frac{d}{tb}$.
\end{tabular}
\end{ex}

\begin{thm}\label{thm:face-weight-equal-cyl}
For a path $P$ in a planar directed network on a cylinder, $$wt(P,y,t)=x_P.$$
\end{thm}
\begin{proof}
We can see this by counting how many times the weight of an edge and its inverse are in the product $wt(P,y,t)$ when the edge is in $P$, when it's is between $P$ and the boundary component $P$ that makes up part of $P_b$ (in Case 1), and when it's between $P$ and the trail (in Case 2 and 3).
\end{proof}

\subsection{Changing orientation}\label{sec:orientation}

\begin{defn}[Definition 9.2 of~\cite{P}]\label{defn:perfect}
A \emph{perfect network} is a planar directed network in which each boundary vertex has degree 1, and each internal vertex either has exactly one edge incoming (and all others outgoing) or exactly one edge outgoing (and all others incoming).
\end{defn}

\begin{prop}[Proposition 9.3 of~\cite{P}]\label{prop:perfect-disk}
Any planar directed network in a disk can be transformed into a perfect network without changing the boundary measurements.
\end{prop}

\begin{figure}[htb]
\begin{center}
\begin{tikzpicture}[scale=1.3]
\path[->,font=\large, >=angle 60, line width=0.4mm]
(0,1) edge node[below] {$x_1$} (1,0)
(0,-1) edge node[below] {$x_2$} (1,0)
(1,0) edge node[above] {$x_3$} (2,0)
(1,0) edge node[left] {$x_4$} (1.5,1)
(1.5,-1) edge node[left] {$x_5$} (1,0)
(2,0) edge node[right] {$x_6$} (1.5,1)
(2,0) edge node[right] {$x_7$} (1.5,-1)
(3,0) edge node[above] {$x_8$} (2,0)
(1.5,1) edge node[right] {$x_9$} (1.5,2)
(1.5,-1) edge node[right] {$x_{10}$} (1.5,-2);
\node at (3.5,0) {$\rightsquigarrow$};
\path[->,font=\large, >=angle 60, line width=0.4mm]
(4,1) edge node[left] {$x_1$} (4.5,-0.5)
(4,-1) edge node[left] {$x_2$} (4.5,-0.5)
(5,0) edge node[above] {$x_3$} (6,0)
(5,0) edge node[left] {$x_4$} (5.5,1)
(5.5,-1) edge node[left] {$x_5$} (4.5,-0.5)
(6,0) edge node[right] {$x_6$} (5.5,1)
(6,0) edge node[right] {$x_7$} (5.5,-1)
(7,0) edge node[above] {$x_8$} (6,0)
(5.5,1) edge node[right] {$x_9$} (5.5,2)
(5.5,-1) edge node[right] {$x_{10}$} (5.5,-2)
(4.5,-0.5) edge node[above] {1} (5,0);
\node at (7.5,0) {$\rightsquigarrow$};
\path[->,font=\large, >=angle 60, line width=0.4mm]
(8,1) edge node[left] {$x_1$} (8.5,-0.5)
(8,-1) edge node[left] {$x_2$} (8.5,-0.5)
(9,0) edge node[above] {$2x_3$} (10,0)
(9,0) edge node[left] {$x_4$} (9.5,1)
(9.5,-1) edge node[left] {$x_5$} (8.5,-0.5)
(10.5,0.5) edge node[right] {$x_6$} (9.5,1)
(10.5,-0.5) edge node[below] {$x_7$} (9.5,-1)
(11.5,0) edge node[above] {$2x_8$} (11,0)
(9.5,1) edge node[right] {$x_9$} (9.5,2)
(9.5,-1) edge node[right] {$x_{10}$} (9.5,-2)
(8.5,-0.5) edge node[above] {1} (9,0)
(10,0) edge node[left] {1} (10.5, 0.5)
(10.5, 0.5) edge node[above] {1} (11,0)
(11,0) edge node[right] {1} (10.5, -0.5)
(10.5, -0.5) edge node[left] {1} (10, 0);
\end{tikzpicture}
\end{center}
\caption{A planar directed network in a disk transformed into a perfect network.}
\label{fig:perfect-network}
\end{figure}
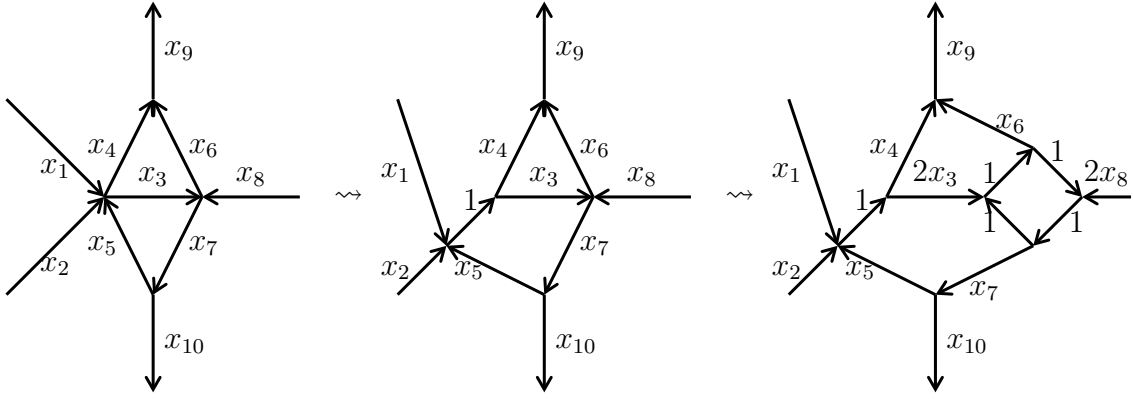

\begin{prop}\label{prop:perfect-cyl}
Any planar directed network on a cylinder can be transformed into a perfect network without changing the boundary measurements.
\end{prop}
\begin{proof}
The proof for Proposition~\ref{prop:perfect-disk} holds for planar directed networks on a cylinder.
\end{proof}

\begin{defn}[Section 9 of~\cite{P}]\label{defn:color}
For an internal vertex $v$ in a perfect network, define the \emph{color} of $v$, $col(v)$, to be black if $v$ has exactly one outgoing edge and white if $v$ has exactly one incoming edge.
\end{defn}

\begin{thm}\label{thm:change-orientation-disk}
Let $N,N'$ be two perfect networks in a disk such that:
\begin{enumerate}[(1)]
\item The underlying graphs $G$ and $G'$ are isomorphic as undirected graphs.

\item Each internal vertex of degree $\neq 2$ has $col_N(v)=col_{N'}(v)$.

\item If the undirected edge $e$ is directed in the same way in $N$ and $N'$, then $x_e=x'_e$.  If the edge $e$ has opposite direction in $N$ and $N'$, the $x_e=(x'_e)^{-1}$.
\end{enumerate}

Then $Meas(N)=Meas(N')\in Gr_{kn}$.
\end{thm}

\begin{ex}\label{ex:2.17}
Let $N,N'$ be as below:

\begin{tikzpicture}
\node at (-4,0) {$N=$};
\node(A) at (-3,0) {$b_1$};
\node(B) at (3,0) {$b_2$};
\node at (6.62,0) {$A(N)=\begin{bmatrix} 1 & x_1 x_4 (x_2 +x_3 )\end{bmatrix}$};
\path[->,font=\large, >=angle 90, line width=0.4mm]
(A) edge node[above] {$x_1$} (-1,0)
(-1,0) edge[bend left=50] node[above] {$x_2$} (1,0)
(-1,0) edge[bend right=50] node[below] {$x_3$} (1,0)
(1, 0) edge node[above] {$x_4$} (B);
\draw [line width=0.25mm, fill=white] (-1,0) circle (0.75mm);
\draw [line width=0.25mm, fill=black] (1,0) circle (0.75mm);
\end{tikzpicture}

\begin{tikzpicture}
\node at (-4,0) {$N'=$};
\node(A) at (-3,0) {$b_1$};
\node(B) at (3,0) {$b_2$};
\node at (7.1,0) {$A(N')=\begin{bmatrix} \frac{(x_1 x_3 x_4)^{-1}}{1+x_2x_3^{-1}} & 1\end{bmatrix}$};
\node at (6.8,-1) {$=\begin{bmatrix} \frac{1}{x_1 x_4 (x_2+x_3)}& 1\end{bmatrix}$};
\path[->,font=\large, >=angle 90, line width=0.4mm]
(-1,0) edge node[above] {$x_1^{-1}$} (A)
(-1,0) edge[bend left=50] node[above] {$x_2$} (1,0)
(1,0) edge[bend left=50] node[below] {$x_3^{-1}$} (-1,0)
(B) edge node[above] {$x_4^{-1}$} (1, 0);
\draw [line width=0.25mm, fill=white] (-1,0) circle (0.75mm);
\draw [line width=0.25mm, fill=black] (1,0) circle (0.75mm);
\end{tikzpicture}

Since left multiplication of $A(N')$ by $[x_1 x_4 (x_2 +x_3)]\in GL_1$ gives $A(N)$, these two matrices represent the same point in the Grassmannian.

\end{ex}

Notice that in our example, $N'$ could be obtained from $N$ by reversing a path from one boundary vertex to another.  In fact, for any two networks $N$ and $N'$ satisfying the conditions of Theorem~\ref{thm:change-orientation-disk}, $N'$ can be obtained from $N$ by reversing a set of paths between boundary vertices and a set of cycles.  Thus, the theorem can be proven by showing that reversing paths between boundary vertices and reversing cycles preserve the boundary measurement map.

As the edge reversal described in Theorem~\ref{thm:change-orientation-disk} does not affect face weights, we may use this theorem in conjunction with our definition of face weights to obtain a plabic network from any planar directed network on a disk.  Thus, plabic networks identify certain directed planar networks that have the same image under the boundary measurement map.

Unfortunately, the statement analogous to Theorem~\ref{thm:change-orientation-disk} for planar directed graphs on a cylinder does not hold.

\begin{thm}[Theorem 4.1 of~\cite{GSV1}]\label{thm:path-reversal-cyl}
Let $P$ be a path with no self-intersections from $b_i$ to $b_j$ in a planar directed network on a cylinder $N$ such that $M_{ij}\not\equiv0$ and $P$ does not intersect the cut.  Create $N'$ from $N$ by reversing the direction of all the edges in $P$ and inverting their weights.  Then $$(Meas^C(N'))(\zeta)=\begin{cases} (Meas^C(N))(\zeta) & b_i,b_j\text{ are on the same boundary component},\\ (Meas^C(N))(-\zeta) & \text{otherwise}.\end{cases}$$
\end{thm}

Since we cannot necessarily reverse paths that begin and end on different boundary components without changing the image of the network under the boundary measurement map, we cannot turn planar directed networks on a cylinder into plabic networks.  In particular, we have to keep track of the orientation of the edges.  However, as path reversal changes the boundary measurements in a predictable way, plabic networks will still prove useful to us (see Theorem~\ref{thm:orientation-invol}).

\begin{prop}[Proposition 2.1 of~\cite{GSV1}]\label{prop:cut-change}
Let $N,N'$ be two networks with the same graph and weights, where $N$ has cut $\gamma$ and $N'$ has cut $\gamma'$ obtained by interchanging one of the base points with $b$, the first boundary vertex below the base point.  Then $$(-1)^{wind(C_P')-1}\zeta^{int(P')}x_{P'}=((-1)^{\alpha(P)}\zeta)^{\beta(b,P)}(-1)^{wind(C_P)-1}\zeta^{int(P)}x_P,$$ where $x_P$ is as in Definition~\ref{defn:xP} and
\begin{align*}
\alpha(P)&=\begin{cases} 0 &\text{if the endpoints of $P$ are on the same boundary component,}\\ 1 &\text{otherwise,}\end{cases}\\
\beta(b,P)&=\begin{cases} 1 & \text{if }b\text{ is the sink of }P,\\ -1 & \text{if }b\text{ is the source of }P,\\ 0 & \text{otherwise.}\end{cases}
\end{align*}
\end{prop}

\begin{thm}\label{thm:orientation-invol}
Let $N,N'$ be two perfect networks on a cylinder such that:
\begin{enumerate}[(1)]
\item The underlying graphs $G$ and $G'$ are isomorphic as undirected graphs.

\item Each internal vertex of degree $\neq 2$ has $col_N(v)=col_{N'}(v)$.

\item If the undirected edge $e$ is directed in the same way in $N$ and $N'$, then $x_e=x'_e$.  If the edge $e$ has opposite direction in $N$ and $N'$, the $x_e=(x'_e)^{-1}$.
\end{enumerate}

Given an involution on the edge weights of $N$ that preserves the boundary measurement map, then there is a canonical way to define an involution of the edge weights of $N'$.
\end{thm}
\begin{proof}
As with planar directed networks on a disk, we can always obtain $N'$ from $N$ by reversing a set of paths and cycles.  Therefore, we only need to show the conclusion shows for $N'$ equal to $N$ with a cycle (with no self-intersections) reversed or $N$ with a path (with no self-intersections) reversed.

First consider a cycle $C$ with no self-intersections.  If $C$ is a contractible loop, then the proof that reversing cycles on a disk does not change the boundary measurements still holds (Lemma 10.5 of \cite{P}).

If $C$ is not a contractible loop, a similar proof holds, except that the winding number is more complicated.  We consider what happens for paths that have edges in $C$.  First, for paths that begin and end on the same boundary component, the winding numbers behave the same as for networks on a disk.  So, the boundary measurements for pairs of vertices on the same boundary component remain the same.  Now consider paths that begin and end on different components.  For vertices $v_i$ and $v_j$ in the cycle, going from $v_i$ to $v_j$ in one direction around the cycle, with as many loops as desired, crosses the cut the same number of times from each side and going from $v_i$ to $v_j$ in the other direction crosses the cut one more time from one side than from the other.  Crossing the cut an additional time adds or removes a loop in $C_P$ as we trace along the cut.  So, we get an extra factor of $-1$ in the boundary measurements for pairs of vertices on different boundary components.

For any cycle in $N'$ with no self-intersections, we can reverse the cycle, apply our involution, and reverse the cycle again.  Any boundary measurements that change when we reverse the cycle change only by a factor of $-1$, and they change again by the same factor when we reverse the cycle a second time.  So, for any cycle $C$, we have the following commutative diagram, where $f$ is defined to be the map that makes this diagram commute:

\begin{center}
\begin{tikzpicture}[scale=3]
\node(A) at (0,0.5) {$N$};
\node(B) at (0,0) {$\widetilde{N}$};
\node(C) at (1,0.5) {$N'$};
\node(D) at (1,0) {$\widetilde{N'}$};
\path[<->,font=\large, >=angle 90]
(A) edge node[above] {cycle reversal} (C)
(A) edge node[left] {involution} (B)
(B) edge node[below] {cycle reversal} (D);
\path[->,font=\large, >=angle 90]
(C) edge node[right] {$f$} (D);
\end{tikzpicture}
\end{center}

$f$ is an involution that preserves the boundary measurements.

Now consider a path $P$ with no self-intersections.  If $P$ does not intersect the cut, Theorem~\ref{thm:path-reversal-cyl} says we can reverse $P$, possibly at the expense of replacing $\zeta$ with $-\zeta$.  If $P$ does intersect the cut, we can move the cut so that $P$ no longer intersects it.  Moving the cut changes the weight of each path by a power of $\zeta$ and a power of $-1$.  These powers depend only on the source and sink of the path, so it changes the weight of each boundary measurement by a power of $\zeta$ and a power of $-1$.  Then reversing the path, since it no longer intersects the cut, either keeps the boundary measurement map the same, or replaces $\zeta$ with $-\zeta$.  Finally, we can move the cut back, which will cause each boundary measurement to again pick up a power of $-1$ and a power of $\zeta$.  This process gives us the following commutative diagram, where $f$ is defined to be the map that makes the diagram commute and $g$ is the composition of functions:

\begin{center}
\begin{tikzpicture}[scale=3]
\node(A) at (0,0.5) {$N$};
\node(B) at (0,0) {$\widetilde{N}$};
\node(C) at (1,0.5) {$M$};
\node(D) at (1,0) {$\widetilde{M}$};
\node(E) at (2,0.5) {$M'$};
\node(F) at (2,0) {$\widetilde{M'}$};
\node(G) at (3,0.5) {$N'$};
\node(H) at (3,0) {$\widetilde{N'}$};
\path[<->,font=\small, >=angle 90]
(A) edge node[above] {cut change} (C)
(A) edge node[left] {involution} (B)
(B) edge node[below] {cut change} (D)
(C) edge node[above] {path reversal} (E)
(D) edge node[below] {path reversal} (F)
(E) edge node[above] {cut change} (G)
(F) edge node[below] {cut change} (H);
\path[->,font=\large, >=angle 90]
(A) edge[bend left=35] node[below] {$g$} (G)
(B) edge[bend right=35] node[above] {$g$} (H)
(G) edge node[right] {$f$} (H);
\end{tikzpicture}
\end{center}

\parindent=0.25in It is clear $f$ is an involution, so we just need to check that $Meas^C(N')=Meas^C(\widetilde{N'})$.  Given a matrix representing the image of a network under the boundary measurement map, $g$ is equivalent to possibly switching $\zeta$ for $-\zeta$ and multiplying each entry by a power of $-1$ and a power of $\zeta$.  Since $Meas^C(N)=Meas^C(\widetilde{N})$, we can pick the same matrix representative for them, and we can see $Meas^C(g(N))=Meas^C(g(\widetilde{N}))$
\end{proof}

Since the plabic network structure is useful to us, but also we can't eliminate orientation, we will be working with directed plabic networks.

\begin{defn}\label{defn:directed-plabic-cyl}
A \emph{directed plabic graph} on a cylinder is a planar directed graph on a cylinder such that each boundary vertex has degree 1 and each internal vertex is colored black of white.  A \emph{directed plabic network} on a cylinder is a directed plabic graph with a weight $y_f\in\R_{>0}$ assigned to each face and a specified trail with weight $t$.
\end{defn}

\section{The Plabic R-Matrix}\label{sec:R-mat}

Postnikov solves the inverse boundary problem for plabic networks up to a set of local transformations which do not alter the boundary measurements (Theorem 12.1 of~\cite{P}).  These are as follows:

\begin{enumerate}
\item[(M1)] Square move.

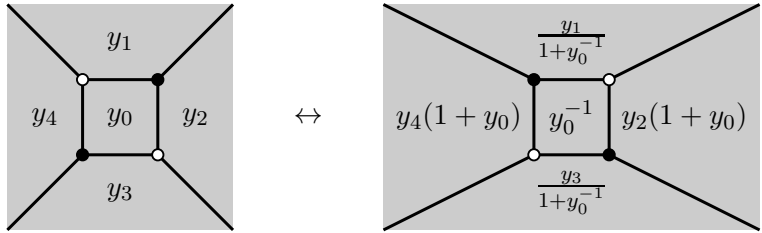
\begin{tikzpicture}[scale=0.5]
\fill[gray!40!white] (-3,-3) rectangle (3,3);
\node at (0,0) {$y_0$};
\node at (0,2) {$y_1$};
\node at (2,0) {$y_2$};
\node at (0,-2) {$y_3$};
\node at (-2,0) {$y_4$};
\path[-,font=\large, >=angle 90, line width=0.4mm]
(-1,-1) edge (1,-1)
(-1,-1) edge (-1,1)
(1,1) edge (1,-1)
(1,1) edge (-1,1)
(1,1) edge (3,3)
(-1,1) edge (-3,3)
(1,-1) edge (3,-3)
(-1,-1) edge (-3,-3);
\draw [line width=0.25mm, fill=white] (-1,1) circle (1.5mm);
\draw [line width=0.25mm, fill=black] (1,1) circle (1.5mm);
\draw [line width=0.25mm, fill=white] (1,-1) circle (1.5mm);
\draw [line width=0.25mm, fill=black] (-1,-1) circle (1.5mm);
\node at (5,0) {$\leftrightarrow$};
\fill[gray!40!white] (7,-3) rectangle (17,3);
\node at (12,0) {$y_0^{-1}$};
\node at (12,2) {$\frac{y_1}{1+y_0^{-1}}$};
\node at (15,0) {$y_2(1+y_0)$};
\node at (12,-2) {$\frac{y_3}{1+y_0^{-1}}$};
\node at (9,0) {$y_4(1+y_0)$};
\path[-,font=\large, >=angle 90, line width=0.4mm]
(11,-1) edge (13,-1)
(11,-1) edge (11,1)
(13,1) edge (13,-1)
(13,1) edge (11,1)
(13,1) edge (17,3)
(11,1) edge (7,3)
(13,-1) edge (17,-3)
(11,-1) edge (7,-3);
\draw [line width=0.25mm, fill=black] (11,1) circle (1.5mm);
\draw [line width=0.25mm, fill=white] (13,1) circle (1.5mm);
\draw [line width=0.25mm, fill=black] (13,-1) circle (1.5mm);
\draw [line width=0.25mm, fill=white] (11,-1) circle (1.5mm);
\end{tikzpicture}

\item[(M2)] Unicolored edge contraction/uncontraction.

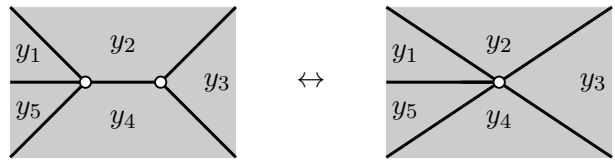
\begin{tikzpicture}[scale=0.5]
\fill[gray!40!white] (-3,-2) rectangle (3,2);
\node at (-2.5,0.75) {$y_1$};
\node at (0,1) {$y_2$};
\node at (2.5,0) {$y_3$};
\node at (0,-1) {$y_4$};
\node at (-2.5,-0.75) {$y_5$};
\path[-,font=\large, >=angle 90, line width=0.4mm]
(-1,0) edge (1,0)
(-3,0) edge (-1,0)
(-3,2) edge (-1,0)
(-3,-2) edge (-1,0)
(3,2) edge (1,0)
(3,-2) edge (1,0);
\draw [line width=0.25mm, fill=white] (-1,0) circle (1.5mm);
\draw [line width=0.25mm, fill=white] (1,0) circle (1.5mm);
\node at (5,0) {$\leftrightarrow$};
\fill[gray!40!white] (7,-2) rectangle (13,2);
\node at (7.5,0.75) {$y_1$};
\node at (10,1) {$y_2$};
\node at (12.5,0) {$y_3$};
\node at (10,-1) {$y_4$};
\node at (7.5,-0.75) {$y_5$};
\path[-,font=\large, >=angle 90, line width=0.4mm]
(9,0) edge (10,0)
(7,0) edge (10,0)
(7,2) edge (10,0)
(7,-2) edge (10,0)
(13,2) edge (10,0)
(13,-2) edge (10,0);
\draw [line width=0.25mm, fill=white] (10,0) circle (1.5mm);
\end{tikzpicture}

The unicolored edge may be white (as pictured) or black and there may be any number of edges on each of the vertices.  All of the face weights remain unchanged.

\item[(M3)] Middle vertex insertion/removal

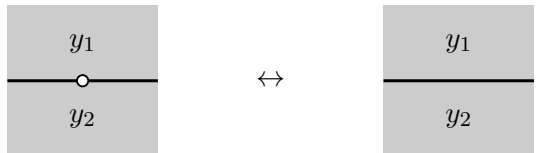
\begin{tikzpicture}[scale=0.5]
\fill[gray!40!white] (-2,-2) rectangle (2,2);
\node at (0,1) {$y_1$};
\node at (0,-1) {$y_2$};
\path[-,font=\large, >=angle 90, line width=0.4mm]
(-2,0) edge (2,0);
\draw [line width=0.25mm, fill=white] (0,0) circle (1.5mm);
\node at (5,0) {$\leftrightarrow$};
\fill[gray!40!white] (8,-2) rectangle (12,2);
\node at (10,1) {$y_1$};
\node at (10,-1) {$y_2$};
\path[-,font=\large, >=angle 90, line width=0.4mm]
(8,0) edge (12,0);
\end{tikzpicture}

Vertex insertion/removal may be done with a vertex of either color.

\item[(R1)] Parallel edge reduction

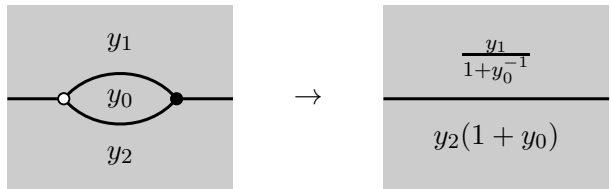
\begin{tikzpicture}[scale=0.5]
\fill[gray!40!white] (-3,-2.5) rectangle (3,2.5);
\node at (0,0) {$y_0$};
\node at (0,1.5) {$y_1$};
\node at (0,-1.5) {$y_2$};
\path[-,font=\large, >=angle 90, line width=0.4mm]
(-3,0) edge (-1.5,0)
(-1.5,0) edge[bend left=50] (1.5,0)
(-1.5,0) edge[bend right=50] (1.5,0)
(1.5,0) edge (3,0);
\draw [line width=0.25mm, fill=white] (-1.5,0) circle (1.5mm);
\draw [line width=0.25mm, fill=black] (1.5,0) circle (1.5mm);
\node at (5,0) {$\rightarrow$};
\fill[gray!40!white] (7,-2.5) rectangle (13,2.5);
\node at (10,1) {$\frac{y_1}{1+y_0^{-1}}$};
\node at (10,-1) {$y_2(1+y_0)$};
\path[-,font=\large, >=angle 90, line width=0.4mm]
(7,0) edge (13,0);
\end{tikzpicture}

\item[(R2)] Leaf reduction

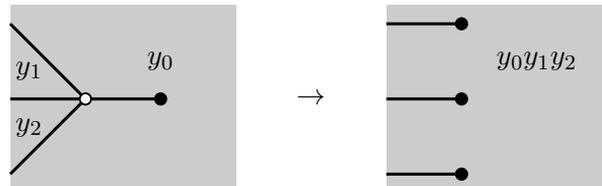
\begin{tikzpicture}[scale=0.5]
\fill[gray!40!white] (-3,-2.5) rectangle (3,2.5);
\node at (1,1) {$y_0$};
\node at (-2.5,0.75) {$y_1$};
\node at (-2.5,-0.75) {$y_2$};
\path[-,font=\large, >=angle 90, line width=0.4mm]
(-3,0) edge (-1,0)
(-3,2) edge (-1,0)
(-3,-2) edge (-1,0)
(-1,0) edge (1,0);
\draw [line width=0.25mm, fill=white] (-1,0) circle (1.5mm);
\draw [line width=0.25mm, fill=black] (1,0) circle (1.5mm);
\node at (5,0) {$\rightarrow$};
\fill[gray!40!white] (7,-2.5) rectangle (13,2.5);
\node at (11,1) {$y_0y_1y_2$};
\path[-,font=\large, >=angle 90, line width=0.4mm]
(7,2) edge (9,2)
(7,-2) edge (9,-2)
(7,0) edge (9,0);
\draw [line width=0.25mm, fill=black] (9,0) circle (1.5mm);
\draw [line width=0.25mm, fill=black] (9,2) circle (1.5mm);
\draw [line width=0.25mm, fill=black] (9,-2) circle (1.5mm);
\end{tikzpicture}

The vertex with degree 1, which we call a leaf, may be any color, and the vertex connected to the leaf (which is of the opposite color) may have any degree $\geq 2$.

\item[(R3)] Dipole reduction

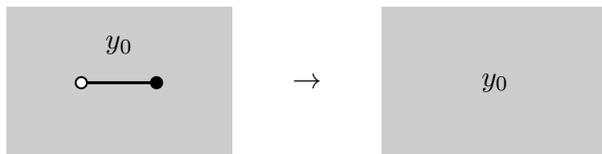
\begin{tikzpicture}[scale=0.5]
\fill[gray!40!white] (-3,-2) rectangle (3,2);
\node at (0,1) {$y_0$};
\path[-,font=\large, >=angle 90, line width=0.4mm]
(-1,0) edge (1,0);
\draw [line width=0.25mm, fill=white] (-1,0) circle (1.5mm);
\draw [line width=0.25mm, fill=black] (1,0) circle (1.5mm);
\node at (5,0) {$\rightarrow$};
\fill[gray!40!white] (7,-2) rectangle (13,2);
\node at (10,0) {$y_0$};
\end{tikzpicture}

\end {enumerate}

All of these transformations specialize to directed edge weighted versions.  From here, we will be using these transformations freely, and considering (directed) plabic networks that differ by them as equivalent.  Our goal in this section is to define and explore a semi-local transformation for planar directed networks on a cylinder.

\begin{defn}\label{defn:moves-reductions}
The transformations (M1) - (M3) are called \emph{moves} and the transformations (R1) - (R3) are called \emph{reductions}.  Two plabic graphs or networks are \emph{move-equivalent} if they can be transformed into the same graph or network by moves.
\end{defn}

\begin{defn}\label{defn:reduced-leafless}
A plabic graph or network is \emph{reduced} if it has no isolated connected components and it is not move-equivalent to any graph or network to which we can apply (R1) or (R2).  A plabic graph or network is \emph{leafless} if it has no non-boundary leaves.
\end{defn}

\begin{defn}\label{defn:alt-strand-diag}
A \emph{Postnikov diagram}, also known as an alternating strand diagram, on a surface with boundary is a set of directed curves, called \emph{strands}, such that when we draw the strands on the universal cover of the surface we have the following:
\begin{enumerate}[(1)]
\item Each strand begins and ends at a boundary vertex or is infinite.

\item There is exactly one strand that enters and one strand that leaves each boundary vertex.

\item No three strands intersect at the same point.

\item All intersections are transverse (the tangent vectors are independent).

\item There is a finite number of intersections in each fundamental domain.

\item Along any strand, the strands that cross it alternate crossing from the left and crossing from the right.

\item Strands do not have self-intersections, except in the case where a strand is a loop attached to a boundary vertex.  Notice that this excludes the possibility of a closed cycle.

\item If two strands intersect at $u$ and $v$, then one strand is oriented from $u$ to $v$ and one strand is oriented from $v$ to $u$.
\end{enumerate}
\end{defn}

Postnikov diagrams are considered up to homotopy.  We can obtain a plabic graph from a Postnikov diagram as follows:
\begin{enumerate}[(1)]
\item Place a black vertex in every face oriented counterclockwise and a white vertex in every face oriented clockwise.

\item If two oriented faces share a corner, connect the vertices in these two faces.
\end{enumerate}

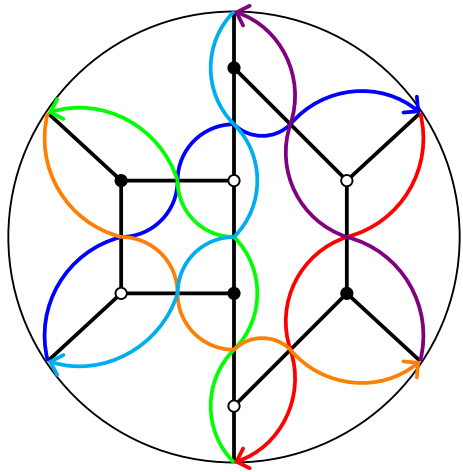
\begin{figure}[htb]
\begin{center}
\begin{tikzpicture}[scale=1.5]
\draw [line width=0.25mm] (0,0) circle (2cm);
\path[-,font=\large, >=angle 90, line width=0.5mm]
(-1,-0.5) edge (0,-0.5)
(-1,0.5) edge (0,0.5)
(0,-1.5) edge (0,-0.5)
(0,1.5) edge (0,0.5)
(0,1.5) edge (1,0.5)
(0,-1.5) edge (1,-0.5)
(1,0.5) edge (1,-0.5)
(0,0.5) edge (0,-0.5)
(-1,0.5) edge (-1,-0.5)
(0,-1.5) edge (0,-2)
(0,1.5) edge (0,2)
(1,0.5) edge (1.65,1.1)
(1,-0.5) edge (1.65,-1.1)
(-1,0.5) edge (-1.65,1.1)
(-1,-0.5) edge (-1.65,-1.1);
\draw [line width=0.25mm, fill=white] (-1,-0.5) circle (0.5mm);
\draw [line width=0.25mm, fill=white] (0,0.5) circle (0.5mm);
\draw [line width=0.25mm, fill=black] (-1,0.5) circle (0.5mm);
\draw [line width=0.25mm, fill=black] (0,-0.5) circle (0.5mm);
\draw [line width=0.25mm, fill=black] (0,1.5) circle (0.5mm);
\draw [line width=0.25mm, fill=white] (0,-1.5) circle (0.5mm);
\draw [line width=0.25mm, fill=black] (1,-0.5) circle (0.5mm);
\draw [line width=0.25mm, fill=white] (1,0.5) circle (0.5mm);
\path[-,font=\large, >=angle 90, line width=0.5mm, blue]
(-1.65,-1.1) edge[bend left=45] (-1,0)
(-0.5,0.5) edge[bend left=45] (-1,0)
(-0.5,0.5) edge[bend left=45] (0,1)
(0.5, 1) edge[bend left=45] (0,1);
\path[->,font=\large, >=angle 90, line width=0.5mm, blue]
(0.5, 1) edge[bend left=45] (1.65,1.1);
\path[-,font=\large, >=angle 90, line width=0.5mm, red]
(1.65,1.1) edge[bend left=45] (1,0)
(0.5,-1) edge[bend left=45] (1,0);
\path[->,font=\large, >=angle 90, line width=0.5mm, red]
(0.5,-1) edge[bend left=45] (0,-2);
\path[-,font=\large, >=angle 90, line width=0.5mm, green]
(0,-2) edge[bend left=45] (0,-1)
(0,0) edge[bend left=45] (0,-1)
(0,0) edge[bend left=45] (-0.5,0.5);
\path[<-,font=\large, >=angle 90, line width=0.5mm, green]
(-1.65,1.1) edge[bend left=45] (-0.5,0.5);
\path[-,font=\large, >=angle 90, line width=0.5mm, orange]
(-1,0) edge[bend left=45] (-1.65,1.1)
(-1,0) edge[bend left=45] (-0.5,-0.5)
(0,-1) edge[bend left=45] (-0.5,-0.5)
(0,-1) edge[bend left=45] (0.5,-1);
\path[<-,font=\large, >=angle 90, line width=0.5mm, orange]
(1.65,-1.1) edge[bend left=45] (0.5,-1);
\path[-,font=\large, >=angle 90, line width=0.5mm, violet]
(1.65,-1.1) edge[bend right=45] (1,0)
(0.5,1) edge[bend right=45] (1,0);
\path[->,font=\large, >=angle 90, line width=0.5mm, violet]
(0.5,1) edge[bend right=45] (0,2);
\path[-,font=\large, >=angle 90, line width=0.5mm, cyan]
(0,2) edge[bend right=45] (0,1)
(0,0) edge[bend right=45] (0,1)
(0,0) edge[bend right=45] (-0.5,-0.5);
\path[<-,font=\large, >=angle 90, line width=0.5mm, cyan]
(-1.65,-1.1) edge[bend right=45] (-0.5,-0.5);
\end{tikzpicture}
\caption{A plabic graph and its Postnikov diagram.}
\label{fig:alt-str}
\end{center}
\end{figure}

Notice that if our surface is a disk, we recover Definition 14.1 of~\cite{P} for a Postnikov diagram.

\begin{thm}[Corollary 14.2 of~\cite{P}]\label{thm:alt-str-reduced-disk}
Postnikov diagrams in a disk are in bijection with leafless reduced plabic graphs in a disk with no unicolored edges.
\end{thm}

\begin{thm}\label{thm:alt-str-reduced-gen}
Postnikov diagrams on a cylinder are in bijection with leafless reduced plabic graphs on a cylinder with no unicolored edges.
\end{thm}

See Section~\ref{sec:AltStr-reduced} for proof.

\begin{defn}\label{defn:cyl-k-loop}
A cylindric $k$-loop plabic graph is a plabic graph on a cylinder that arises from a Postnikov diagram where exactly $k$ of the strands are loops around the cylinder with the same orientation.
\end{defn}

Cylindric $k$-loop plabic graphs have $k$ strings of vertices around the cylinder.  Those strings alternate black and white vertices, and the black vertices only have additional edges on the left of the strand while the white vertices only have additional edges to the right of the strand (see Figure~\ref{fig:cyl-2-loop}).

\begin{figure}[htp]
\begin{center}
\begin{tikzpicture}[scale=0.75]
\fill[gray!40!white] (0,0) rectangle (5,7);
\path[-,font=\large, >=angle 90, line width=0.4mm]
(0,4) edge (1,4)
(0,6) edge (1,6)
(4,1) edge (5,1)
(4,3) edge (5,3)
(4,5) edge (5,5)
(1,0) edge (1,7)
(4,0) edge (4,7)
(1,3) edge (2,4)
(1,5) edge (2,4)
(2,4) edge (3,3)
(3,3) edge (4,2)
(3,3) edge (4,4)
(1,5) edge (4,6);
\path[dashed,->,font=\large, >=angle 90, line width=0.4mm]
(0,0) edge (5,0)
(0,7) edge (5,7);
\draw [line width=0.25mm, fill=white] (1,3) circle (1mm);
\draw [line width=0.25mm, fill=white] (1,5) circle (1mm);
\draw [line width=0.25mm, fill=white] (3,3) circle (1mm);
\draw [line width=0.25mm, fill=white] (4,1) circle (1mm);
\draw [line width=0.25mm, fill=white] (4,3) circle (1mm);
\draw [line width=0.25mm, fill=white] (4,5) circle (1mm);
\draw [line width=0.25mm, fill=black] (1,4) circle (1mm);
\draw [line width=0.25mm, fill=black] (1,6) circle (1mm);
\draw [line width=0.25mm, fill=black] (2,4) circle (1mm);
\draw [line width=0.25mm, fill=black] (4,2) circle (1mm);
\draw [line width=0.25mm, fill=black] (4,4) circle (1mm);
\draw [line width=0.25mm, fill=black] (4,6) circle (1mm);
\node at (2.5,-1) {plabic graph};
\fill[gray!40!white] (8,0) rectangle (14,7);
\node at (11,-1) {Postnikov diagram};
\draw [line width=0.4mm,blue] plot [smooth, tension=1] coordinates { (14,6) (10.5,5.25) (9.5,2) (8,1) };
\path[->,font=\large, >=angle 90, line width=0.4mm, blue]
(8.1,1) edge (8,1);
\draw [line width=0.4mm,blue] plot [smooth, tension=1] coordinates { (14,4) (12.5,3.4) (12.5,1.6) (14,1) };
\path[->,font=\large, >=angle 90, line width=0.4mm, blue]
(13.9,1) edge (14,1);
\draw [line width=0.4mm,blue] plot [smooth, tension=1] coordinates { (14,2) (11.25,2.4) (10.75,3.6) (8,4) };
\path[->,font=\large, >=angle 90, line width=0.4mm, blue]
(8.1,4) edge (8,4);
\path[->,font=\large, >=angle 90, line width=0.4mm, blue]
(8,3) edge (14,3)
(8,5) edge (14,5)
(9,0) edge (9,7)
(13,0) edge (13,7);
\path[dashed,->,font=\large, >=angle 90, line width=0.4mm]
(8,0) edge (14,0)
(8,7) edge (14,7);
\end{tikzpicture}
\end{center}
\caption{A cylindric 2-loop plabic graph and its Postnikov diagram.}
\label{fig:cyl-2-loop}
\end{figure}
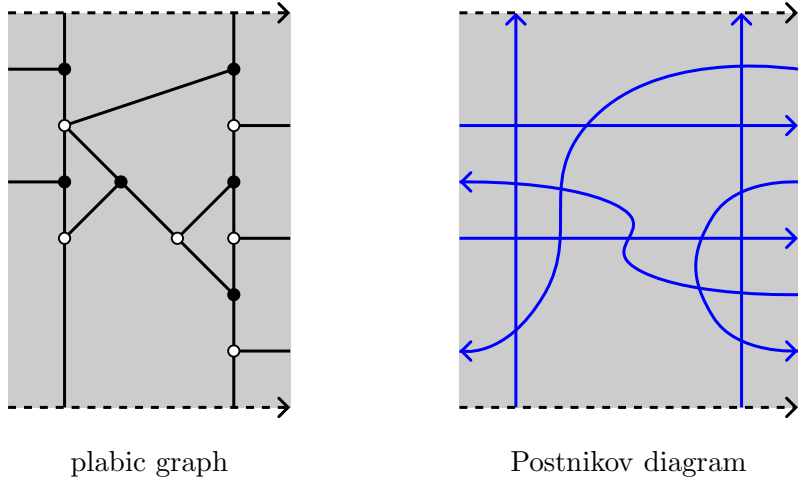

\begin{defn}\label{defn:interior-vertices}
For a cylindric $k$-loop plabic graph, any vertices that are not on one of the strings of vertices defined by the $k$ loops and lie between two of these strings are called \emph{interior vertices}.
\end{defn}

\begin{thm}\label{thm:int-vert-k-loop}
Any cylindric $k$-loop plabic graph can be transformed by moves to one that has no interior vertices.
\end{thm}

See Section~\ref{sec:IntVert} for proof.

Consider a cylindric $k$-loop plabic graph.  By Theorem \ref{thm:no-int-vert}, we can choose two adjacent strings and assume there are no interior vertices between them.  We will be describing an involution on the edge weights of these two adjacent strings and the edges between them, so we'll ignore the rest of the graph.  That is, we'll assume we have a cylindric 2-loop plabic graph with no vertices other than those on the 2 strings.

\begin{defn}\label{defn:cannonical-orientation}
The \emph{canonical orientation} of a cylindric 2-loop plabic graph is the orientation where the edges on the strings are oriented up and the other edges are oriented from left to right.
\end{defn}

Let us choose the edges from white vertices to black vertices to be variables, and set all the other edges to have weight 1.  We do this to have a canonical way to kill the gauge transformations on our network (notice the number of variables is the number of faces, or the dimension of $\R^E$/\{gauge transformations\}).  In our diagrams, any edges that are not labeled are assumed to have weight 1.  We now have a directed plabic network.

\begin{defn}\label{defn:expanded-directed-plabic-network}
We can expand the directed plabic network by splitting each vertex that has multiple edges to the other string into that many vertices, and inserting vertices of the opposite color between them.  Let any new edges created have weight 1.  Thus, we have a new network that is equivalent to the old one, but all the interior faces are hexagons where the colors of the vertices alternate and where there are 2 white vertices on the left string and 2 black vertices on the right string.  We'll call this the \emph{expanded directed plabic network}.
\end{defn}

Choose a white vertex on the left string of a cylindric 2-loop expanded directed plabic network.  Call the weight of the edge from this vertex to the black vertex above it on the string $x_1$.  For the white vertex on the left string that is part of the same face as these two vertices, call the weight of the edge from this vertex to the black vertex above it on the string $y_1$.  Call the weight of the edge that makes up the upper boundary of the face containing these vertices $z_1$.  Moving up the left string, give the next white to black edge the weight $x_2$, and so on.  Do the same on the right.  Moving in the same direction, label all the edges between the strings with weights $z_2, z_3$, etc.  For a particular network, some of these values might be set to 1, because we created the edges when we expanded the network.  Let $x=\{x_1,x_2,...,x_n\},y=\{y_1,y_2,...,y_n\},z=\{z_1,z_2,...,z_n\}$.  We will consider all of these indices to be modular.

Consider the network on the universal cover of the cylinder.  Choose a fundamental domain.  Label all the edge weights in the fundamental domain with a superscript (1), so the weights are $x_1^{(1)}, y_1^{(1)}, z_1^{(1)}, x_2^{(1)}$, etc.  Label all the edge weights in the copy of the fundamental domain that lies above with a superscript (2), and so on.  Define $\lambda_i(x,y,z)$ = sum of the path weights from vertex $a$ to vertex $b$, where vertex $a$ is a the highest vertex on the left string of the interior face that has an edge labeled $x_i^{(1)}$ and $b$ is the lowest vertex on the right string of the interior face that has an edge labeled $x_i^{(2)}$.  Define $A_i=\{i-j\ |\ j\geq 1,\ x_{i-1}=...=x_{i-j}=1\}$ and $\alpha_i=|A_i|$.  Define, $B_i=\{i+j\ |\ j\geq 1,\ y_{i+1}=...=y_{i+j}=1\}$ and $\beta_i=|B_i|$.

\begin{figure}[tbp]
\begin{center}
\begin{tikzpicture}[scale=0.75]
\fill[gray!40!white] (0,0) rectangle (6,13);
\path[->,font=\large, >=angle 90, line width=0.4mm]
(0,3) edge (2,3)
(0,6) edge (2,6)
(0,9) edge (2,9)
(0,12) edge (2,12)
(2,0) edge (2,1.5)
(2,1.5) edge node[left] {$a$} (2,3)
(2,3) edge (2,4.5)
(2,4.5) edge node[left] {$b$} (2,6)
(2,6) edge (2,7.5)
(2,7.5) edge node[left] {$c$} (2,9)
(2,9) edge (2,10.5)
(2,10.5) edge node[left] {$d$} (2,12)
(2,12) edge (2,13)
(4,0) edge (4,1.5)
(4,1.5) edge node[right] {$e$} (4,3)
(4,3) edge (4,4.5)
(4,4.5) edge node[right] {$f$} (4,6)
(4,6) edge (4,7.5)
(4,7.5) edge node[right] {$g$} (4,9)
(4,9) edge (4,10.5)
(4,10.5) edge node[right] {$h$} (4,12)
(4,12) edge (4,13)
(4,1.5) edge (6,1.5)
(4,4.5) edge (6,4.5)
(4,7.5) edge (6,7.5)
(4,10.5) edge (6,10.5)
(2,1.5) edge node[above] {$i$} (4,3)
(2,4.5) edge node[above] {$j$} (4,3)
(2,7.5) edge node[above] {$k$} (4,6)
(2,10.5) edge node[right] {$\ell$} (4,6)
(2,10.5) edge node[above] {$m$} (4,9)
(2,10.5) edge node[above] {$n$} (4,12);
\path[dashed,->,font=\large, >=angle 90, line width=0.4mm]
(0,0) edge (6,0)
(0,13) edge (6,13);
\draw [line width=0.25mm, fill=white] (2,1.5) circle (1mm);
\draw [line width=0.25mm, fill=white] (2,4.5) circle (1mm);
\draw [line width=0.25mm, fill=white] (2,7.5) circle (1mm);
\draw [line width=0.25mm, fill=white] (2,10.5) circle (1mm);
\draw [line width=0.25mm, fill=white] (4,1.5) circle (1mm);
\draw [line width=0.25mm, fill=white] (4,4.5) circle (1mm);
\draw [line width=0.25mm, fill=white] (4,7.5) circle (1mm);
\draw [line width=0.25mm, fill=white] (4,10.5) circle (1mm);
\draw [line width=0.25mm, fill=black] (2,3) circle (1mm);
\draw [line width=0.25mm, fill=black] (2,6) circle (1mm);
\draw [line width=0.25mm, fill=black] (2,9) circle (1mm);
\draw [line width=0.25mm, fill=black] (2,12) circle (1mm);
\draw [line width=0.25mm, fill=black] (4,3) circle (1mm);
\draw [line width=0.25mm, fill=black] (4,6) circle (1mm);
\draw [line width=0.25mm, fill=black] (4,9) circle (1mm);
\draw [line width=0.25mm, fill=black] (4,12) circle (1mm);
\node at (3,-1) {directed plabic network};
\path[<->,font=\large, >=angle 90]
(7,6.5) edge (8,6.5);
\fill[gray!40!white] (9,0) rectangle (15,13);
\path[->,font=\large, >=angle 90, line width=0.4mm]
(9,2) edge (11,2)
(9,4) edge (11,4)
(9,6) edge (11,6)
(9,12) edge (11,12)
(13,1) edge (15,1)
(13,5) edge (15,5)
(13,9) edge (15,9)
(13,11) edge (15,11)
(11,0) edge (11,1)
(11,1) edge node[left] {$a$}(11,2)
(11,2) edge (11,3)
(11,3) edge node[left] {$b$} (11,4)
(11,4) edge (11,5)
(11,5) edge node[left] {$c$} (11,6)
(11,6) edge (11,7)
(11,7) edge (11,8)
(11,8) edge (11,9)
(11,9) edge (11,10)
(11,10) edge (11,11)
(11,11) edge node[left] {$d$} (11,12)
(11,12) edge (11,13)
(13,0) edge (13,1)
(13,1) edge node[right] {$e$} (13,2)
(13,2) edge (13,3)
(13,3) edge (13,4)
(13,4) edge (13,5)
(13,5) edge node[right] {$f$} (13,6)
(13,6) edge (13,7)
(13,7) edge (13,8)
(13,8) edge (13,9)
(13,9) edge node[right] {$g$} (13,10)
(13,10) edge (13,11)
(13,11) edge node[right] {$h$} (13,12)
(13,12) edge (13,13)
(11,1) edge node[above] {$i$} (13,2)
(11,3) edge node[above] {$j$} (13,4)
(11,5) edge node[above] {$k$} (13,6)
(11,7) edge node[above] {$\ell$} (13,8)
(11,9) edge node[above] {$m$} (13,10)
(11,11) edge node[above] {$n$} (13,12);
\path[dashed,->,font=\large, >=angle 90, line width=0.4mm]
(9,0) edge (15,0)
(9,13) edge (15,13);
\draw [line width=0.25mm, fill=white] (11,1) circle (1mm);
\draw [line width=0.25mm, fill=white] (11,3) circle (1mm);
\draw [line width=0.25mm, fill=white] (11,5) circle (1mm);
\draw [line width=0.25mm, fill=white] (11,7) circle (1mm);
\draw [line width=0.25mm, fill=white] (11,9) circle (1mm);
\draw [line width=0.25mm, fill=white] (11,11) circle (1mm);
\draw [line width=0.25mm, fill=white] (13,1) circle (1mm);
\draw [line width=0.25mm, fill=white] (13,3) circle (1mm);
\draw [line width=0.25mm, fill=white] (13,5) circle (1mm);
\draw [line width=0.25mm, fill=white] (13,7) circle (1mm);
\draw [line width=0.25mm, fill=white] (13,9) circle (1mm);
\draw [line width=0.25mm, fill=white] (13,11) circle (1mm);
\draw [line width=0.25mm, fill=black] (11,2) circle (1mm);
\draw [line width=0.25mm, fill=black] (11,4) circle (1mm);
\draw [line width=0.25mm, fill=black] (11,6) circle (1mm);
\draw [line width=0.25mm, fill=black] (11,8) circle (1mm);
\draw [line width=0.25mm, fill=black] (11,10) circle (1mm);
\draw [line width=0.25mm, fill=black] (11,12) circle (1mm);
\draw [line width=0.25mm, fill=black] (13,2) circle (1mm);
\draw [line width=0.25mm, fill=black] (13,4) circle (1mm);
\draw [line width=0.25mm, fill=black] (13,6) circle (1mm);
\draw [line width=0.25mm, fill=black] (13,8) circle (1mm);
\draw [line width=0.25mm, fill=black] (13,10) circle (1mm);
\draw [line width=0.25mm, fill=black] (13,12) circle (1mm);
\node at (12,-1) {expanded directed plabic network};
\end{tikzpicture}
\caption{A directed plabic network and the corresponding expanded directed plabic network.}
\label{fig:expanded-directed-plabic-network}
\end{center}
\end{figure}
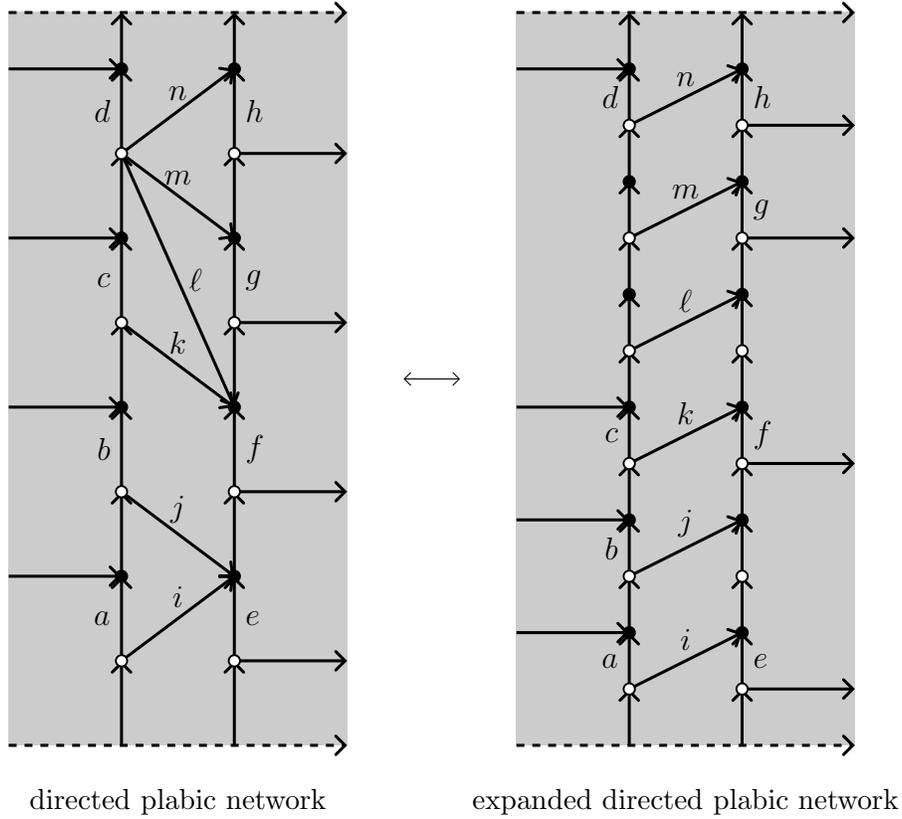

\newpage
\begin{ex}\label{ex:lambda}
Consider the network below:

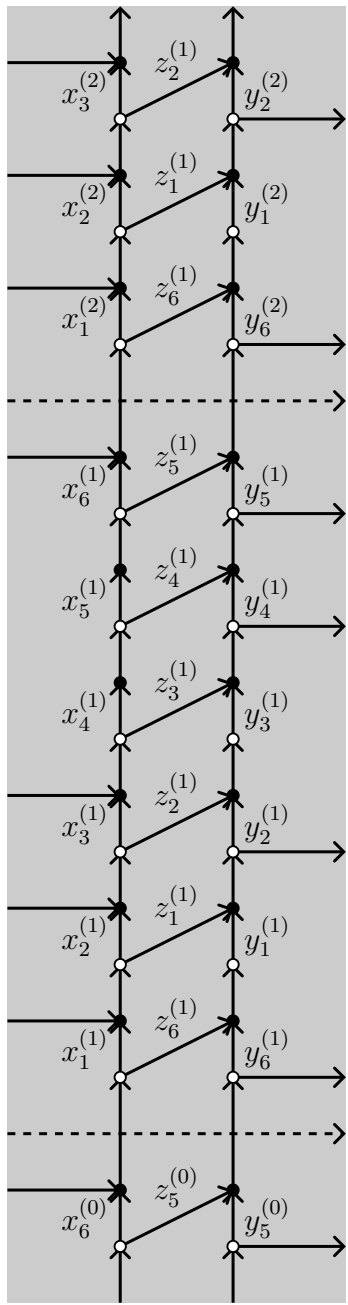
\begin{tabular}[t]{ c m{0.3cm} m{6cm} }
\raisebox{\dimexpr-\height + 53ex\relax}{
\begin{tikzpicture}[scale=0.75]
\fill[gray!40!white] (9,-3) rectangle (15,20);
\path[->,font=\large, >=angle 90, line width=0.4mm]
(9,2) edge (11,2)
(9,15) edge (11,15)
(9,4) edge (11,4)
(9,17) edge (11,17)
(9,6) edge (11,6)
(9,19) edge (11,19)
(9,12) edge (11,12)
(9,-1) edge (11,-1)
(13,1) edge (15,1)
(13,14) edge (15,14)
(13,5) edge (15,5)
(13,18) edge (15,18)
(13,9) edge (15,9)
(13,11) edge (15,11)
(13,-2) edge (15,-2)
(11,-1) edge (11,1)
(11,1) edge node[left] {$x_1^{(1)}$}(11,2)
(11,14) edge node[left] {$x_1^{(2)}$}(11,15)
(11,2) edge (11,3)
(11,15) edge (11,16)
(11,3) edge node[left] {$x_2^{(1)}$} (11,4)
(11,16) edge node[left] {$x_2^{(2)}$} (11,17)
(11,4) edge (11,5)
(11,17) edge (11,18)
(11,5) edge node[left] {$x_3^{(1)}$} (11,6)
(11,18) edge node[left] {$x_3^{(2)}$} (11,19)
(11,6) edge (11,7)
(11,19) edge (11,20)
(11,7) edge node[left] {$x_4^{(1)}$}(11,8)
(11,8) edge (11,9)
(11,9) edge node[left] {$x_5^{(1)}$}(11,10)
(11,10) edge (11,11)
(11,-3) edge (11,-2)
(11,11) edge node[left] {$x_6^{(1)}$} (11,12)
(11,-2) edge node[left] {$x_6^{(0)}$} (11,-1)
(11,12) edge (11,14)
(13,-1) edge (13,1)
(13,1) edge node[right] {$y_6^{(1)}$} (13,2)
(13,14) edge node[right] {$y_6^{(2)}$} (13,15)
(13,2) edge (13,3)
(13,15) edge (13,16)
(13,3) edge node[right] {$y_1^{(1)}$}(13,4)
(13,16) edge node[right] {$y_1^{(2)}$}(13,17)
(13,4) edge (13,5)
(13,17) edge (13,18)
(13,5) edge node[right] {$y_2^{(1)}$} (13,6)
(13,18) edge node[right] {$y_2^{(2)}$} (13,19)
(13,6) edge (13,7)
(13,19) edge (13,20)
(13,7) edge node[right] {$y_3^{(1)}$}(13,8)
(13,8) edge (13,9)
(13,9) edge node[right] {$y_4^{(1)}$} (13,10)
(13,10) edge (13,11)
(13,-3) edge (13,-2)
(13,11) edge node[right] {$y_5^{(1)}$} (13,12)
(13,-2) edge node[right] {$y_5^{(0)}$} (13,-1)
(13,12) edge (13,14)
(11,1) edge node[above] {$z_6^{(1)}$} (13,2)
(11,14) edge node[above] {$z_6^{(1)}$} (13,15)
(11,3) edge node[above] {$z_1^{(1)}$} (13,4)
(11,16) edge node[above] {$z_1^{(1)}$} (13,17)
(11,5) edge node[above] {$z_2^{(1)}$} (13,6)
(11,18) edge node[above] {$z_2^{(1)}$} (13,19)
(11,7) edge node[above] {$z_3^{(1)}$} (13,8)
(11,9) edge node[above] {$z_4^{(1)}$} (13,10)
(11,11) edge node[above] {$z_5^{(1)}$} (13,12)
(11,-2) edge node[above] {$z_5^{(0)}$} (13,-1);
\path[dashed,->,font=\large, >=angle 90, line width=0.4mm]
(9,0) edge (15,0)
(9,13) edge (15,13);
\draw [line width=0.25mm, fill=white] (11,1) circle (1mm);
\draw [line width=0.25mm, fill=white] (11,3) circle (1mm);
\draw [line width=0.25mm, fill=white] (11,5) circle (1mm);
\draw [line width=0.25mm, fill=white] (11,7) circle (1mm);
\draw [line width=0.25mm, fill=white] (11,9) circle (1mm);
\draw [line width=0.25mm, fill=white] (11,11) circle (1mm);
\draw [line width=0.25mm, fill=white] (13,1) circle (1mm);
\draw [line width=0.25mm, fill=white] (13,3) circle (1mm);
\draw [line width=0.25mm, fill=white] (13,5) circle (1mm);
\draw [line width=0.25mm, fill=white] (13,7) circle (1mm);
\draw [line width=0.25mm, fill=white] (13,9) circle (1mm);
\draw [line width=0.25mm, fill=white] (13,11) circle (1mm);
\draw [line width=0.25mm, fill=white] (11,14) circle (1mm);
\draw [line width=0.25mm, fill=white] (11,16) circle (1mm);
\draw [line width=0.25mm, fill=white] (11,18) circle (1mm);
\draw [line width=0.25mm, fill=white] (13,14) circle (1mm);
\draw [line width=0.25mm, fill=white] (13,16) circle (1mm);
\draw [line width=0.25mm, fill=white] (13,18) circle (1mm);
\draw [line width=0.25mm, fill=white] (11,-2) circle (1mm);
\draw [line width=0.25mm, fill=white] (13,-2) circle (1mm);
\draw [line width=0.25mm, fill=black] (11,2) circle (1mm);
\draw [line width=0.25mm, fill=black] (11,4) circle (1mm);
\draw [line width=0.25mm, fill=black] (11,6) circle (1mm);
\draw [line width=0.25mm, fill=black] (11,8) circle (1mm);
\draw [line width=0.25mm, fill=black] (11,10) circle (1mm);
\draw [line width=0.25mm, fill=black] (11,12) circle (1mm);
\draw [line width=0.25mm, fill=black] (13,2) circle (1mm);
\draw [line width=0.25mm, fill=black] (13,4) circle (1mm);
\draw [line width=0.25mm, fill=black] (13,6) circle (1mm);
\draw [line width=0.25mm, fill=black] (13,8) circle (1mm);
\draw [line width=0.25mm, fill=black] (13,10) circle (1mm);
\draw [line width=0.25mm, fill=black] (13,12) circle (1mm);
\draw [line width=0.25mm, fill=black] (13,-1) circle (1mm);
\draw [line width=0.25mm, fill=black] (13,15) circle (1mm);
\draw [line width=0.25mm, fill=black] (13,17) circle (1mm);
\draw [line width=0.25mm, fill=black] (13,19) circle (1mm);
\draw [line width=0.25mm, fill=black] (11,-1) circle (1mm);
\draw [line width=0.25mm, fill=black] (11,15) circle (1mm);
\draw [line width=0.25mm, fill=black] (11,17) circle (1mm);
\draw [line width=0.25mm, fill=black] (11,19) circle (1mm);
\end{tikzpicture}}
& & $\lambda_1(x,y,z)=$ paths from the vertex at the bottom of $x_2^{(1)}$ to the vertex at the top of $y_6^{(2)}$, so $\lambda_1(x,y,z)=z_1y_2y_3y_4y_5y_6+x_2z_2y_3y_4y_5y_6+x_2x_3z_3y_4y_5y_6+x_2x_3x_4z_4y_5y_6+x_2x_3x_4x_5z_5y_6+x_2x_3x_4x_5x_5z_6$.

$A_1=\emptyset$ because there are no weights directly below $x_1$ set to 1 from expanding the network.

$B_1=\emptyset$ because there are no weights directly above $y_1$ set to 1 from expanding the network.

$\alpha_1=\beta_1=0$.

\vspace{0.25in}
$\lambda_2(x,y,z)=z_2y_3y_4y_5y_6y_1+x_3z_3y_4y_5y_6y_1+x_3x_4z_4y_5y_6y_1+x_3x_4x_5z_5y_6y_1+x_3x_4x_5x_5z_6y_1+x_3x_4x_5x_5x_6z_1$.

$A_2=\emptyset$ because there are no weights directly below $x_2$ set to 1 from expanding the network.

$B_2=\{3\}$ because $y_3$, which is directly above $y_2$ is set to 1 from expanding the network, but $y_4$ is not set to 1.

$\alpha_2=0,\beta_2=1$.
\end{tabular}
\end{ex}

\vspace{2in}
\begin{defn}\label{defn:Te}
Define $T_e$ to be the transformation on edge weights from $(x,y,z)$ to $(x',y',z')$ where
\begin{align*}
x_i'&=\begin{cases} 1 & x_i\text{ set to }1, \\ \frac{\left(\prod_{k=0}^{\alpha_i} y_{i-k-1}\right)\lambda_{i-\alpha_i-1}(x,y,z)}{\lambda_i(x,y,z)} & \text{otherwise,} \end{cases}\\
y_i'&=\begin{cases} 1 & y_i\text{ set to }1, \\ \frac{\left(\prod_{k=0}^{\beta_i} x_{i+k+1}\right)\lambda_{i+\beta_i+1}(x,y,z)}{\lambda_i(x,y,z)} & \text{otherwise,} \end{cases}\\
z_i'&=\frac{z_i\left(\prod_{k=1}^{\alpha_i} y_{i-k-1}\right)\lambda_{i-\alpha_i-1}(x,y,z)\left(\prod_{k=1}^{\beta_i} x_{i+k+1}\right)\lambda_{i+\beta_i+1}(x,y,z)}{\lambda_{i-1}(x,y,z)\lambda_{i+1}(x,y,z)}.
\end{align*}
We call $T_e$ the edge weighted \emph{plabic R-matrix}.
\end{defn}

\begin{thm}\label{thm:Te}
$T_e$ has the following properties:
\begin{enumerate}
\item It preserves the boundary measurements.

\item It is an involution.

\item $(x,y,z)$ and $(x',y',z')$ are the only choices of weights on a fixed cylindric 2-loop plabic graph that preserve the boundary measurements.

\item It satisfies the braid relation.
\end{enumerate}
\end{thm}

See Section~\ref{sec:R-matThm} for proof.

\begin{ex}\label{ex:LP-edge}
Consider the network below:

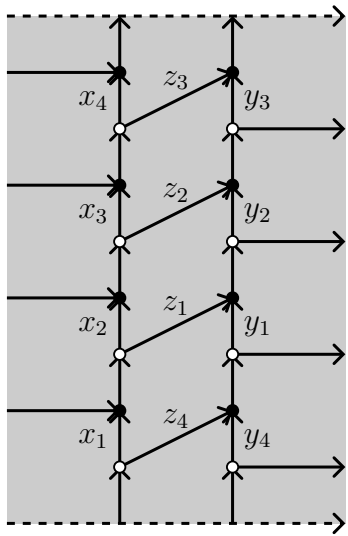
\begin{center}
\begin{tikzpicture}[scale=0.75]
\fill[gray!40!white] (0,0) rectangle (6,9);
\path[->,font=\large, >=angle 90, line width=0.4mm]
(2,0) edge (2,1)
(2,1) edge node[left] {$x_1$} (2,2)
(2,2) edge (2,3)
(2,3) edge node[left] {$x_2$} (2,4)
(2,4) edge (2,5)
(2,5) edge node[left] {$x_3$} (2,6)
(2,6) edge (2,7)
(2,7) edge node[left] {$x_4$} (2,8)
(2,8) edge (2,9)
(4,0) edge (4,1)
(4,1) edge node[right] {$y_4$} (4,2)
(4,2) edge (4,3)
(4,3) edge node[right] {$y_1$} (4,4)
(4,4) edge (4,5)
(4,5) edge node[right] {$y_2$} (4,6)
(4,6) edge (4,7)
(4,7) edge node[right] {$y_3$} (4,8)
(4,8) edge (4,9)
(2,1) edge node[above] {$z_4$} (4,2)
(2,3) edge node[above] {$z_1$} (4,4)
(2,5) edge node[above] {$z_2$} (4,6)
(2,7) edge node[above] {$z_3$} (4,8)
(0,2) edge (2,2)
(0,4) edge (2,4)
(0,6) edge (2,6)
(0,8) edge (2,8)
(4,1) edge (6,1)
(4,3) edge (6,3)
(4,5) edge (6,5)
(4,7) edge (6,7);
\path[dashed,->,font=\large, >=angle 90, line width=0.4mm]
(0,0) edge (6,0)
(0,9) edge (6,9);
\draw [line width=0.25mm, fill=white] (2,1) circle (1mm);
\draw [line width=0.25mm, fill=white] (2,3) circle (1mm);
\draw [line width=0.25mm, fill=white] (2,5) circle (1mm);
\draw [line width=0.25mm, fill=white] (2,7) circle (1mm);
\draw [line width=0.25mm, fill=white] (4,1) circle (1mm);
\draw [line width=0.25mm, fill=white] (4,3) circle (1mm);
\draw [line width=0.25mm, fill=white] (4,5) circle (1mm);
\draw [line width=0.25mm, fill=white] (4,7) circle (1mm);
\draw [line width=0.25mm, fill=black] (2,2) circle (1mm);
\draw [line width=0.25mm, fill=black] (2,4) circle (1mm);
\draw [line width=0.25mm, fill=black] (2,6) circle (1mm);
\draw [line width=0.25mm, fill=black] (2,8) circle (1mm);
\draw [line width=0.25mm, fill=black] (4,2) circle (1mm);
\draw [line width=0.25mm, fill=black] (4,4) circle (1mm);
\draw [line width=0.25mm, fill=black] (4,6) circle (1mm);
\draw [line width=0.25mm, fill=black] (4,8) circle (1mm);
\end{tikzpicture}
\end{center}

$T_e$ gives us the following values:

$\displaystyle x_1'=\frac{y_4\lambda_4(x,y,z)}{\lambda_1(x,y,z)}=\frac{y_4(x_2x_3x_4z_4+x_2x_3z_3y_4+x_2z_2y_3y_3+z_1y_2y_3y_4)}{x_2x_3x_4z_4+x_2x_3z_3y_4+x_2z_2y_3y_4+z_1y_2y_3y_4}$

$\displaystyle x_2'=\frac{y_1\lambda_1(x,y,z)}{\lambda_2(x,y,z)}=\frac{y_1(x_2x_3x_4z_4+x_2x_3z_3y_4+x_2z_2y_3y_4+z_1y_2y_3y_4)}{x_3x_4x_1z_1+x_3x_4z_4y_1+x_3z_3y_4y_1+z_2y_3y_1}$

$\displaystyle x_3'=\frac{y_2\lambda_2(x,y,z)}{\lambda_3(x,y,z)}=\frac{y_2(x_3x_4x_1z_1+x_3x_4z_4y_1+x_3z_3y_4y_1+z_2y_3y_1)}{x_4x_1x_2z_2+x_4x_1z_1y_2+x_4z_4y_1y_2+z_3y_4y_1y_2}$

$\displaystyle x_4'=\frac{y_3\lambda_3(x,y,z)}{\lambda_4(x,y,z)}=\frac{y_3(x_4x_1x_2z_2+x_4x_1z_1y_2+x_4z_4y_1y_2+z_3y_4y_1y_2)}{x_2x_3x_4z_4+x_2x_3z_3y_4+x_2z_2y_3y_3+z_1y_2y_3y_4}$

$\displaystyle y_1'=\frac{x_2\lambda_2(x,y,z)}{\lambda_1(x,y,z)}=\frac{x_2(x_3x_4x_1z_1+x_3x_4z_4y_1+x_3z_3y_4y_1+z_2y_3y_1)}{x_2x_3x_4z_4+x_2x_3z_3y_4+x_2z_2y_3y_4+z_1y_2y_3y_4}$

$\displaystyle y_2'=\frac{x_3\lambda_3(x,y,z)}{\lambda_2(x,y,z)}=\frac{x_3(x_4x_1x_2z_2+x_4x_1z_1y_2+x_4z_4y_1y_2+z_3y_4y_1y_2)}{x_3x_4x_1z_1+x_3x_4z_4y_1+x_3z_3y_4y_1+z_2y_3y_1}$

$\displaystyle y_3'=\frac{x_4\lambda_4(x,y,z)}{\lambda_3(x,y,z)}=\frac{x_4(x_2x_3x_4z_4+x_2x_3z_3y_4+x_2z_2y_3y_3+z_1y_2y_3y_4)}{x_4x_1x_2z_2+x_4x_1z_1y_2+x_4z_4y_1y_2+z_3y_4y_1y_2}$

$\displaystyle y_3'=\frac{x_1\lambda_1(x,y,z)}{\lambda_4(x,y,z)}=\frac{x_1(x_2x_3x_4z_4+x_2x_3z_3y_4+x_2z_2y_3y_4+z_1y_2y_3y_4)}{x_2x_3x_4z_4+x_2x_3z_3y_4+x_2z_2y_3y_3+z_1y_2y_3y_4}$

$z_1'=z_1$

$z_2'=z_2$

$z_3'=z_3$

$z_4'=z_4$

Notice that in this example, the directed plabic network is the same as the expanded directed plabic network.  These types of networks are of particular interest to us, both because of their simplicity and because they correspond to the wiring diagrams studied in Section 6 of~\cite{LP1} where the horizontal wires are all in the same direction and either all the wire cycles are whirls or all the wire cycles are curls.  Thus, if we let all of our $z$ variables equal 1, we have recovered the whurl relation of~\cite{LP1}, which is also the geometric R-matrix.
\end{ex}

\vspace{2in}
\begin{ex}\label{ex:GS-edge}
Consider the network below:

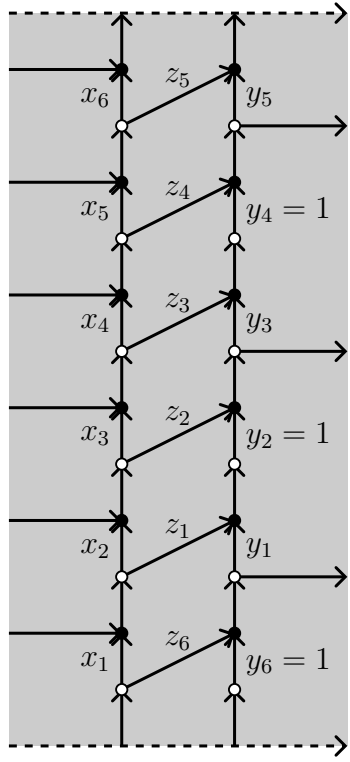
\begin{center}
\begin{tikzpicture}[scale=0.75]
\fill[gray!40!white] (0,0) rectangle (6,13);
\path[->,font=\large, >=angle 90, line width=0.4mm]
(2,0) edge (2,1)
(2,1) edge node[left] {$x_1$} (2,2)
(2,2) edge (2,3)
(2,3) edge node[left] {$x_2$} (2,4)
(2,4) edge (2,5)
(2,5) edge node[left] {$x_3$} (2,6)
(2,6) edge (2,7)
(4,0) edge (4,1)
(4,1) edge node[right] {$y_6=1$} (4,2)
(4,2) edge (4,3)
(4,3) edge node[right] {$y_1$} (4,4)
(4,4) edge (4,5)
(4,5) edge node[right] {$y_2=1$} (4,6)
(4,6) edge (4,7)
(2,1) edge node[above] {$z_6$} (4,2)
(2,3) edge node[above] {$z_1$} (4,4)
(2,5) edge node[above] {$z_2$} (4,6)
(0,2) edge (2,2)
(0,4) edge (2,4)
(0,6) edge (2,6)
(4,3) edge (6,3)
(2,7) edge node[left] {$x_4$} (2,8)
(2,8) edge (2,9)
(2,9) edge node[left] {$x_5$} (2,10)
(2,10) edge (2,11)
(2,11) edge node[left] {$x_6$} (2,12)
(2,12) edge (2,13)
(4,7) edge node[right] {$y_3$} (4,8)
(4,8) edge (4,9)
(4,9) edge node[right] {$y_4=1$} (4,10)
(4,10) edge (4,11)
(4,11) edge node[right] {$y_5$} (4,12)
(4,12) edge (4,13)
(2,7) edge node[above] {$z_3$} (4,8)
(2,9) edge node[above] {$z_4$} (4,10)
(2,11) edge node[above] {$z_5$} (4,12)
(0,8) edge (2,8)
(0,10) edge (2,10)
(0,12) edge (2,12)
(4,7) edge (6,7)
(4,11) edge (6,11);
\path[dashed,->,font=\large, >=angle 90, line width=0.4mm]
(0,0) edge (6,0)
(0,13) edge (6,13);
\draw [line width=0.25mm, fill=white] (2,1) circle (1mm);
\draw [line width=0.25mm, fill=white] (2,3) circle (1mm);
\draw [line width=0.25mm, fill=white] (2,5) circle (1mm);
\draw [line width=0.25mm, fill=white] (4,1) circle (1mm);
\draw [line width=0.25mm, fill=white] (4,3) circle (1mm);
\draw [line width=0.25mm, fill=white] (4,5) circle (1mm);
\draw [line width=0.25mm, fill=black] (2,2) circle (1mm);
\draw [line width=0.25mm, fill=black] (2,4) circle (1mm);
\draw [line width=0.25mm, fill=black] (2,6) circle (1mm);
\draw [line width=0.25mm, fill=black] (4,2) circle (1mm);
\draw [line width=0.25mm, fill=black] (4,4) circle (1mm);
\draw [line width=0.25mm, fill=black] (4,6) circle (1mm);
\draw [line width=0.25mm, fill=white] (2,7) circle (1mm);
\draw [line width=0.25mm, fill=white] (2,9) circle (1mm);
\draw [line width=0.25mm, fill=white] (2,11) circle (1mm);
\draw [line width=0.25mm, fill=white] (4,7) circle (1mm);
\draw [line width=0.25mm, fill=white] (4,9) circle (1mm);
\draw [line width=0.25mm, fill=white] (4,11) circle (1mm);
\draw [line width=0.25mm, fill=black] (2,8) circle (1mm);
\draw [line width=0.25mm, fill=black] (2,10) circle (1mm);
\draw [line width=0.25mm, fill=black] (2,12) circle (1mm);
\draw [line width=0.25mm, fill=black] (4,8) circle (1mm);
\draw [line width=0.25mm, fill=black] (4,10) circle (1mm);
\draw [line width=0.25mm, fill=black] (4,12) circle (1mm);
\end{tikzpicture}
\end{center}

$T_e$ gives us the following values:
\begin{align*}
x_1'&=\frac{y_6\lambda_6(x,y,z)}{\lambda_1(x,y,z)}\\
&=\frac{x_1x_2x_3x_4x_5z_5+x_1x_2x_3x_4z_4y_5+x_1x_2x_3z_3y_5+x_1x_2z_2y_3y_5+x_1z_1y_3y_5+z_6y_1y_3y_5}{x_2x_3x_4x_5x_6z_6+x_2x_3x_4x_5z_5+x_2x_3x_4z_4y_5+x_2x_3z_3y_5+x_2z_2y_3y_5+z_1y_3y_5}\\
x_2'&=\frac{y_1\lambda_1(x,y,z)}{\lambda_2(x,y,z)}\\
&=\frac{y_1(x_2x_3x_4x_5x_6z_6+x_2x_3x_4x_5z_5+x_2x_3x_4z_4y_5+x_2x_3z_3y_5+x_2z_2y_3y_5+z_1y_3y_5)}{x_3x_4x_5x_6x_1z_1+x_3x_4x_5x_6z_6y_1+x_3x_4x_5z_5y_1+x_3x_4z_4y_5y_1+x_3z_3y_5y_1+z_2y_3y_5y_1}\\
x_3'&=\frac{y_2\lambda_2(x,y,z)}{\lambda_3(x,y,z)}\\
&=\frac{x_3x_4x_5x_6x_1z_1+x_3x_4x_5x_6z_6y_1+x_3x_4x_5z_5y_1+x_3x_4z_4y_5y_1+x_3z_3y_5y_1+z_2y_3y_5y_1}{x_4x_5x_6x_1x_2z_2+x_4x_5x_6x_1z_1+x_4x_5x_6z_6y_1+x_4x_5z_5y_1+x_4z_4y_5y_1+z_3y_5y_1}
\end{align*}
\begin{align*}
x_4'&=\frac{y_3\lambda_3(x,y,z)}{\lambda_4(x,y,z)}\\
&=\frac{y_3(x_4x_5x_6x_1x_2z_2+x_4x_5x_6x_1z_1+x_4x_5x_6z_6y_1+x_4x_5z_5y_1+x_4z_4y_5y_1+z_3y_5y_1)}{x_5x_6x_1x_2x_3z_3+x_5x_6x_1x_2z_2y_3+x_5x_6x_1z_1y_3+x_5x_6z_6y_1y_3+x_5z_5y_1y_3+z_4y_5y_1y_3}\\
x_5'&=\frac{y_4\lambda_4(x,y,z)}{\lambda_5(x,y,z)}\\
&=\frac{x_5x_6x_1x_2x_3z_3+x_5x_6x_1x_2z_2y_3+x_5x_6x_1z_1y_3+x_5x_6z_6y_1y_3+x_5z_5y_1y_3+z_4y_5y_1y_3}{x_6x_1x_2x_3x_4z_4+x_6x_1x_2x_3z_3+x_6x_1x_2z_2y_3+x_6x_1z_1y_3+x_6x_1z_1y_3+x_6z_6y_1y_3+z_5y_1y_3}\\
x_6'&=\frac{y_5\lambda_5(x,y,z)}{\lambda_6(x,y,z)}\\
&=\frac{y_5(x_6x_1x_2x_3x_4z_4+x_6x_1x_2x_3z_3+x_6x_1x_2z_2y_3+x_6x_1z_1y_3+x_6x_1z_1y_3+x_6z_6y_1y_3+z_5y_1y_3)}{x_1x_2x_3x_4x_5z_5+x_1x_2x_3x_4z_4y_5+x_1x_2x_3z_3y_5+x_1x_2z_2y_3y_5+x_1z_1y_3y_5+z_6y_1y_3y_5}\\
y_1'&=\frac{x_2x_3\lambda_3(x,y,z)}{\lambda_1(x,y,z)}\\
&=\frac{x_2x_3(x_4x_5x_6x_1x_2z_2+x_4x_5x_6x_1z_1+x_4x_5x_6z_6y_1+x_4x_5z_5y_1+x_4z_4y_5y_1+z_3y_5y_1)}{x_2x_3x_4x_5x_6z_6+x_2x_3x_4x_5z_5+x_2x_3x_4z_4y_5+x_2x_3z_3y_5+x_2z_2y_3y_5+z_1y_3y_5}\\
y_2'&=1\\
y_3'&=\frac{x_4x_5\lambda_5(x,y,z)}{\lambda_3(x,y,z)}\\
&=\frac{x_4x_5(x_6x_1x_2x_3x_4z_4+x_6x_1x_2x_3z_3+x_6x_1x_2z_2y_3+x_6x_1z_1y_3+x_6x_1z_1y_3+x_6z_6y_1y_3+z_5y_1y_3)}{x_4x_5x_6x_1x_2z_2+x_4x_5x_6x_1z_1+x_4x_5x_6z_6y_1+x_4x_5z_5y_1+x_4z_4y_5y_1+z_3y_5y_1}\\
y_4'&=1\\
y_5'&=\frac{x_6x_1\lambda_1(x,y,z)}{\lambda_5(x,y,z)}\\
&=\frac{x_6x_1(x_2x_3x_4x_5x_6z_6+x_2x_3x_4x_5z_5+x_2x_3x_4z_4y_5+x_2x_3z_3y_5+x_2z_2y_3y_5+z_1y_3y_5)}{x_6x_1x_2x_3x_4z_4+x_6x_1x_2x_3z_3+x_6x_1x_2z_2y_3+x_6x_1z_1y_3+x_6x_1z_1y_3+x_6z_6y_1y_3+z_5y_1y_3}\\
y_6'&=1\\
z_1'&=\frac{z_1x_3\lambda_3(x,y,z)}{\lambda_2(x,y,z)}\\
&=\frac{z_1x_3(x_4x_5x_6x_1x_2z_2+x_4x_5x_6x_1z_1+x_4x_5x_6z_6y_1+x_4x_5z_5y_1+x_4z_4y_5y_1+z_3y_5y_1)}{x_3x_4x_5x_6x_1z_1+x_3x_4x_5x_6z_6y_1+x_3x_4x_5z_5y_1+x_3x_4z_4y_5y_1+x_3z_3y_5y_1+z_2y_3y_5y_1}\\
z_2'&=z_2\\
z_3'&=\frac{z_3x_5\lambda_5(x,y,z)}{\lambda_4(x,y,z)}\\
&=\frac{z_3x_5(x_6x_1x_2x_3x_4z_4+x_6x_1x_2x_3z_3+x_6x_1x_2z_2y_3+x_6x_1z_1y_3+x_6x_1z_1y_3+x_6z_6y_1y_3+z_5y_1y_3)}{x_5x_6x_1x_2x_3z_3+x_5x_6x_1x_2z_2y_3+x_5x_6x_1z_1y_3+x_5x_6z_6y_1y_3+x_5z_5y_1y_3+z_4y_5y_1y_3}\\
z_4'&=z_4
\end{align*}
\begin{align*}
z_5'&=\frac{z_5x_1\lambda_1(x,y,z)}{\lambda_6(x,y,z)}\\
&=\frac{z_5x_1(x_2x_3x_4x_5x_6z_6+x_2x_3x_4x_5z_5+x_2x_3x_4z_4y_5+x_2x_3z_3y_5+x_2z_2y_3y_5+z_1y_3y_5)}{x_1x_2x_3x_4x_5z_5+x_1x_2x_3x_4z_4y_5+x_1x_2x_3z_3y_5+x_1x_2z_2y_3y_5+x_1z_1y_3y_5+z_6y_1y_3y_5}\\
z_6'&=z_6
\end{align*}
\end{ex}

We will now rewrite $T_e$ in terms of face and trail weights.  We begin with a cylindric 2-loop plabic graph with the canonical orientation.  Choose an edge adjacent to the left boundary.  Follow this edge, and then go up the left string.  At the first opportunity, make a right to cross to the right string.  Follow the right string up and at the first opportunity make a right to the right boundary.  If there is more than one edge on the right string in this path, then the path passes a face between strings that has no edge to the right boundary.  This face must have an edge to the left boundary.  Choose this edge to the left boundary to start with and repeat.  When this process yields a path that has only 5 edges, we will select that to be our trail (going from the left to right boundary).  Let the trail weight be $t$.  Label the faces on the left $a_1,a_2,...,a_\ell$ starting above the trail and going up.  In the same way, label the faces on the right $b_1, b_2,...,b_m$ and label the faces in the center $c_1,c_2,...,c_{n-1},c_n=\frac{1}{a_1a_2...a_\ell b_1b_2...b_mc_1c_2...c_{n-1}}$.  We will consider all of these indices to be modular.

We will say $a_j$ is associated to $i$ if the highest edge on the left string bordering the face labeled $a_j$ also borders the face $c_i$.  Similarly, $b_j$ is associated to $i$ if the lowest edge on the right string bordering the face labeled $b_j$ also borders the face $c_i$.

\begin{lemma}\label{lemma:face-to-edge}
Suppose we have an expanded cylindric 2-loop plabic graph with the face and trail weights as above.  We can turn this into a directed graph with the canonical orientation and the following edge weights:
\begin{itemize}
\item For a face labeled $a_j$, give the edge from a white vertex to a black vertex on the left string that is highest on the face the weight $a_j^{-1}$.  If $a_j$ is associated to $i$, this should be the edge that borders both the face labeled $a_j$ and the face labeled $c_i$.

\item For a face labeled $b_j$, give the edge from a white vertex to a black vertex on the left string that is lowest on the face the weight $b_j$.  If $b_j$ is associated to $i$, this should be the edge that borders both the face labeled $b_j$ and the face labeled $c_i$.

\item Give all other edges on the strings weight 1.

\item Give the edge between the two strings that is part of the trail the weight $t$.

\item Give each edge between the two strings the weight of the edge below it multiplied by the weight of the face in between and the weight of the edges of the face on the right string, and then divided by the weight of the edges of the face on the left string.  That is, if an edge is between the faces labeled $c_i$ and $c_{i+1}$, give it the weight $$t\left(\prod_{\substack{j\text{ where }a_j\text{ is }\\\text{associated to }k\leq i}}a_j\right)\left(\prod_{\substack{j\text{ where }b_j\text{ is }\\\text{associated to }k\leq i}}b_j\right)\left(\prod_{j \leq i}c_j\right).$$
\end{itemize}
\end{lemma}
\begin{proof}
Clear by computation.
\end{proof}

Using this, we can define $\widehat{\lambda}_i(a,b,c)$ to be $\lambda_i(z,y,z)$, where $x,y,z$ are defined from $a,b,c$ as in Lemma~\ref{lemma:face-to-edge} and we choose our indexing so that $y_1=b_1$.

\begin{figure}[htp]
\centering
\begin{tikzpicture}[scale=0.75]
\fill[gray!40!white] (0,0) rectangle (5,13);
\path[->,font=\large, >=angle 90, line width=0.4mm]
(0,4) edge (1,4)
(0,6) edge (1,6)
(0,12) edge (1,12)
(1,0) edge (1,1)
(1,1) edge (1,2)
(1,3) edge (1,4)
(1,4) edge (1,5)
(1,5) edge (1,6)
(1,6) edge (1,7)
(1,7) edge (1,8)
(1,8) edge (1,9)
(1,9) edge (1,10)
(1,10) edge (1,11)
(1,11) edge (1,12)
(1,12) edge (1,13)
(4,0) edge (4,1)
(4,1) edge (4,2)
(4,2) edge (4,3)
(4,3) edge (4,4)
(4,5) edge (4,6)
(4,6) edge (4,7)
(4,7) edge (4,8)
(4,8) edge (4,9)
(4,9) edge (4,10)
(4,10) edge (4,11)
(4,11) edge (4,12)
(4,12) edge (4,13)
(4,1) edge (5,1)
(4,9) edge (5,9)
(4,11) edge (5,11)
(1,1) edge (4,2)
(1,5) edge (4,6)
(1,7) edge (4,8)
(1,9) edge (4,10)
(1,11) edge (4,12);
\path[->,font=\large, >=angle 90, line width=0.4mm, blue]
(0,2) edge (1,2)
(1,2) edge (1,3)
(1,3) edge (4,4)
(4,4) edge (4,5)
(4,5) edge (5,5);
\node at (0.5,3) {$a_1$};
\node at (0.5,5) {$a_2$};
\node at (0.5,9) {$a_3$};
\node at (0.5,1) {$a_4$};
\node at (4.5,3) {$b_4$};
\node at (4.5,7) {$b_1$};
\node at (4.5,10) {$b_2$};
\node at (4.5,12) {$b_3$};
\node at (2.5,0.5) {$c_5$};
\node at (2.5,2.5) {$c_6$};
\node at (2.5,4.5) {$c_1$};
\node at (2.5,6.5) {$c_2$};
\node at (2.5,8.5) {$c_3$};
\node at (2.5,10.5) {$c_4$};
\node at (5.5,5) {\textcolor{blue}{$t$}};
\path[dashed,->,font=\large, >=angle 90, line width=0.4mm]
(0,0) edge (5,0)
(0,13) edge (5,13);
\draw [line width=0.25mm, fill=white] (1,1) circle (1mm);
\draw [line width=0.25mm, fill=white] (1,3) circle (1mm);
\draw [line width=0.25mm, fill=white] (1,5) circle (1mm);
\draw [line width=0.25mm, fill=white] (1,7) circle (1mm);
\draw [line width=0.25mm, fill=white] (1,9) circle (1mm);
\draw [line width=0.25mm, fill=white] (1,11) circle (1mm);
\draw [line width=0.25mm, fill=white] (4,1) circle (1mm);
\draw [line width=0.25mm, fill=white] (4,3) circle (1mm);
\draw [line width=0.25mm, fill=white] (4,5) circle (1mm);
\draw [line width=0.25mm, fill=white] (4,7) circle (1mm);
\draw [line width=0.25mm, fill=white] (4,8) circle (1mm);
\draw [line width=0.25mm, fill=white] (4,11) circle (1mm);
\draw [line width=0.25mm, fill=black] (1,2) circle (1mm);
\draw [line width=0.25mm, fill=black] (1,4) circle (1mm);
\draw [line width=0.25mm, fill=black] (1,6) circle (1mm);
\draw [line width=0.25mm, fill=black] (1,8) circle (1mm);
\draw [line width=0.25mm, fill=black] (1,10) circle (1mm);
\draw [line width=0.25mm, fill=black] (1,12) circle (1mm);
\draw [line width=0.25mm, fill=black] (4,2) circle (1mm);
\draw [line width=0.25mm, fill=black] (4,4) circle (1mm);
\draw [line width=0.25mm, fill=black] (4,6) circle (1mm);
\draw [line width=0.25mm, fill=black] (4,8) circle (1mm);
\draw [line width=0.25mm, fill=black] (4,10) circle (1mm);
\draw [line width=0.25mm, fill=black] (4,12) circle (1mm);
\path[<->,font=\large, >=angle 90]
(6.5,6.5) edge (7.5,6.5);
\fill[gray!40!white] (8.5,0) rectangle (20,13);
\path[->,font=\large, >=angle 90, line width=0.4mm]
(8.5,2) edge (11,2)
(8.5,4) edge (11,4)
(8.5,6) edge (11,6)
(8.5,12) edge (11,12)
(18,1) edge (20,1)
(18,5) edge (20,5)
(18,9) edge (20,9)
(18,11) edge (20,11)
(11,0) edge (11,1)
(11,1) edge node[left] {$x_6=a_4^{-1}$}(11,2)
(11,2) edge (11,3)
(11,3) edge node[left] {$x_1=a_1^{-1}$} (11,4)
(11,4) edge (11,5)
(11,5) edge node[left] {$x_2=a_2^{-1}$} (11,6)
(11,6) edge (11,7)
(11,7) edge node[left] {$x_3=1$} (11,8)
(11,8) edge (11,9)
(11,9) edge node[left] {$x_4=1$} (11,10)
(11,10) edge (11,11)
(11,11) edge node[left] {$x_5=a_3^{-1}$} (11,12)
(11,12) edge (11,13)
(18,0) edge (18,1)
(18,1) edge node[right] {$y_5=b_4$} (18,2)
(18,2) edge (18,3)
(18,3) edge node[right] {$y_6=1$} (18,4)
(18,4) edge (18,5)
(18,5) edge node[right] {$y_1=b_1$} (18,6)
(18,6) edge (18,7)
(18,7) edge node[right] {$y_2=1$} (18,8)
(18,8) edge (18,9)
(18,9) edge node[right] {$y_3=b_2$} (18,10)
(18,10) edge (18,11)
(18,11) edge node[right] {$y_4=b_3$} (18,12)
(18,12) edge (18,13)
(11,1) edge (18,2)
(11,3) edge (18,4)
(11,5) edge (18,6)
(11,7) edge (18,8)
(11,9) edge (18,10)
(11,11) edge (18,12);
\node at (14.5,2) {\rotatebox{8}{$z_5=a_1a_2a_3b_1b_2b_3b_4c_1c_2c_3c_4c_5t$}};
\node at (14.5,4) {\rotatebox{8}{$z_6=t$}};
\node at (14.5,6) {\rotatebox{8}{$z_1=a_1b_1c_1t$}};
\node at (14.5,8) {\rotatebox{8}{$z_2=a_1a_2b_1c_1c_2t$}};
\node at (14.5,10) {\rotatebox{8}{$z_3=a_1a_2b_1b_2c_1c_2c_3t$}};
\node at (14.5,12) {\rotatebox{8}{$z_4=a_1a_2b_1b_2b_3c_1c_2c_3c_4t$}};
\path[dashed,->,font=\large, >=angle 90, line width=0.4mm]
(8.5,0) edge (20,0)
(8.5,13) edge (20,13);
\draw [line width=0.25mm, fill=white] (11,1) circle (1mm);
\draw [line width=0.25mm, fill=white] (11,3) circle (1mm);
\draw [line width=0.25mm, fill=white] (11,5) circle (1mm);
\draw [line width=0.25mm, fill=white] (11,7) circle (1mm);
\draw [line width=0.25mm, fill=white] (11,9) circle (1mm);
\draw [line width=0.25mm, fill=white] (11,11) circle (1mm);
\draw [line width=0.25mm, fill=white] (18,1) circle (1mm);
\draw [line width=0.25mm, fill=white] (18,3) circle (1mm);
\draw [line width=0.25mm, fill=white] (18,5) circle (1mm);
\draw [line width=0.25mm, fill=white] (18,7) circle (1mm);
\draw [line width=0.25mm, fill=white] (18,8) circle (1mm);
\draw [line width=0.25mm, fill=white] (18,11) circle (1mm);
\draw [line width=0.25mm, fill=black] (11,2) circle (1mm);
\draw [line width=0.25mm, fill=black] (11,4) circle (1mm);
\draw [line width=0.25mm, fill=black] (11,6) circle (1mm);
\draw [line width=0.25mm, fill=black] (11,8) circle (1mm);
\draw [line width=0.25mm, fill=black] (11,10) circle (1mm);
\draw [line width=0.25mm, fill=black] (11,12) circle (1mm);
\draw [line width=0.25mm, fill=black] (18,2) circle (1mm);
\draw [line width=0.25mm, fill=black] (18,4) circle (1mm);
\draw [line width=0.25mm, fill=black] (18,6) circle (1mm);
\draw [line width=0.25mm, fill=black] (18,8) circle (1mm);
\draw [line width=0.25mm, fill=black] (18,10) circle (1mm);
\draw [line width=0.25mm, fill=black] (18,12) circle (1mm);
\end{tikzpicture}
\caption{Changing an expanded cylindric 2-loop plabic network with face and trail weights into one with edge weights as in Lemma~\ref{lemma:face-to-edge}.}
\label{fig:face-to-edge}
\end{figure}
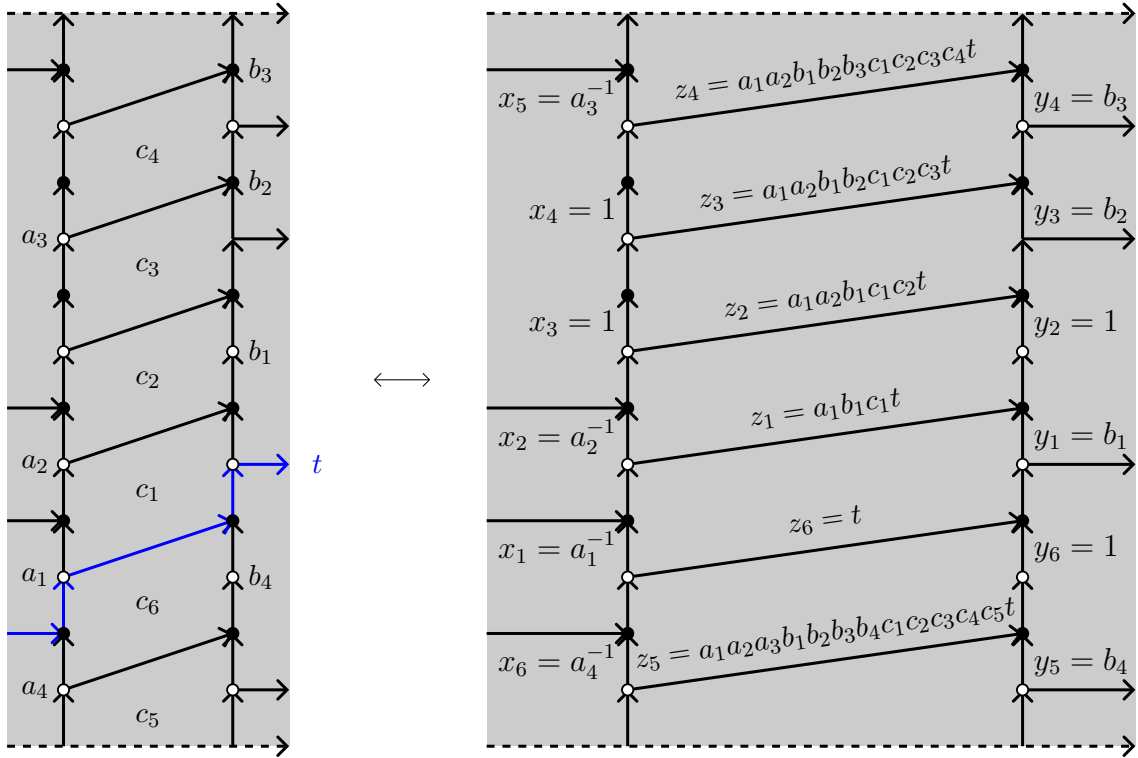

\vspace{2in}
\begin{defn}\label{defn:Tf}
Define $T_f$ to be the transformation on face weights from $(a,b,c)$ to $(a',b',c')$ where
\begin{align*}
a_i'&=\frac{\widehat{\lambda}_j(a,b,c)}{\widehat{\lambda}_p(a,b,c)\left(\prod_{\substack{b_r\text{ associated to }s\\\text{where }p\leq s<j}} b_r\right)}&&\hspace{-1.3in}\text{where }a_i\text{ is associated to }j,a_{i-1}\text{ is associated to }p\\
b_i'&=\frac{\widehat{\lambda}_q(a,b,c)}{\widehat{\lambda}_j(a,b,c)\left(\prod_{\substack{a_r\text{ associated to }s\\\text{where }j< s\leq q}} a_r\right)}&&\hspace{-1.3in}\text{where }b_i\text{ is associated to }j, b_{i+1}\text{ is associated to }q\\
c_i'&=\frac{c_i\widehat{\lambda}_{i-1}(a,b,c)\left(\prod_{\substack{a_r\text{ associated to }s\\\text{where }i\leq s\leq i+1}} a_r\right)\left(\prod_{\substack{b_r\text{ associated to }s\\\text{where }i-1\leq s\leq i}} b_r\right)}{\widehat{\lambda}_{i+1}(a,b,c)}
\end{align*}
We call $T_f$ the face weighted \emph{plabic R-matrix}.
\end{defn}

\begin{thm}\label{thm:Tf}
$T_f$ has the following properties:
\begin{enumerate}
\item It preserves the boundary measurements.

\item It is an involution.

\item $(a,b,c,t)$ and $(a',b',c',t)$ are the only choices of face and trail weights on a fixed cylindric 2-loop plabic graph that preserve the boundary measurements.

\item It satisfies the braid relation.
\end{enumerate}
\end{thm}

See Section~\ref{sec:R-matThm} for proof.

\begin{ex}\label{ex:LP-face}
We revisit Example~\ref{ex:LP-edge}, but with face variables this time.

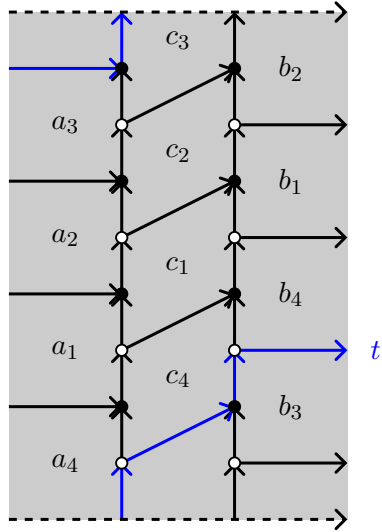
\begin{center}
\begin{tikzpicture}[scale=0.75]
\fill[gray!40!white] (0,0) rectangle (6,9);
\node at (1,3) {$a_1$};
\node at (1,5) {$a_2$};
\node at (1,7) {$a_3$};
\node at (1,1) {$a_4$};
\node at (5,6) {$b_1$};
\node at (5,8) {$b_2$};
\node at (5,2) {$b_3$};
\node at (5,4) {$b_4$};
\node at (3,4.5) {$c_1$};
\node at (3,6.5) {$c_2$};
\node at (3,8.5) {$c_3$};
\node at (3,2.5) {$c_4$};
\node at (6.5,3) {\textcolor{blue}{$t$}};
\path[->,font=\large, >=angle 90, line width=0.4mm]
(2,1) edge (2,2)
(2,3) edge (2,4)
(2,4) edge (2,5)
(2,5) edge (2,6)
(2,6) edge (2,7)
(2,7) edge (2,8)
(2,2) edge (2,3)
(4,0) edge (4,1)
(4,1) edge (4,2)
(4,4) edge (4,5)
(4,3) edge (4,4)
(4,5) edge (4,6)
(4,6) edge (4,7)
(4,7) edge (4,8)
(4,8) edge (4,9)
(2,3) edge (4,4)
(2,5) edge (4,6)
(2,7) edge (4,8)
(0,4) edge (2,4)
(0,6) edge (2,6)
(0,2) edge (2,2)
(4,1) edge (6,1)
(4,5) edge (6,5)
(4,7) edge (6,7);
\path[->,font=\large, >=angle 90, line width=0.4mm, blue]
(0,8) edge (2,8)
(2,8) edge (2,9)
(2,0) edge (2,1)
(2,1) edge (4,2)
(4,2) edge (4,3)
(4,3) edge (6,3);
\path[dashed,->,font=\large, >=angle 90, line width=0.4mm]
(0,0) edge (6,0)
(0,9) edge (6,9);
\draw [line width=0.25mm, fill=white] (2,1) circle (1mm);
\draw [line width=0.25mm, fill=white] (2,3) circle (1mm);
\draw [line width=0.25mm, fill=white] (2,5) circle (1mm);
\draw [line width=0.25mm, fill=white] (2,7) circle (1mm);
\draw [line width=0.25mm, fill=white] (4,1) circle (1mm);
\draw [line width=0.25mm, fill=white] (4,3) circle (1mm);
\draw [line width=0.25mm, fill=white] (4,5) circle (1mm);
\draw [line width=0.25mm, fill=white] (4,7) circle (1mm);
\draw [line width=0.25mm, fill=black] (2,2) circle (1mm);
\draw [line width=0.25mm, fill=black] (2,4) circle (1mm);
\draw [line width=0.25mm, fill=black] (2,6) circle (1mm);
\draw [line width=0.25mm, fill=black] (2,8) circle (1mm);
\draw [line width=0.25mm, fill=black] (4,2) circle (1mm);
\draw [line width=0.25mm, fill=black] (4,4) circle (1mm);
\draw [line width=0.25mm, fill=black] (4,6) circle (1mm);
\draw [line width=0.25mm, fill=black] (4,8) circle (1mm);
\end{tikzpicture}
\end{center}

$T_f$ gives us the following values:

$\displaystyle a_1'=\frac{\widehat{\lambda}_1(a,b,c)}{b_4\widehat{\lambda}_4(a,b,c)}=\frac{a_1a_2a_3a_4b_1b_2b_3b_4(c_1+c_1c_2+c_1c_2c_3)+1}{a_2a_3a_4b_1b_2b_3b_4(1+c_1+c_1c_2+c_1c_2c_3)}$

$\displaystyle a_2'=\frac{\widehat{\lambda}_2(a,b,c)}{b_1\widehat{\lambda}_1(a,b,c)}=\frac{a_2(a_1a_2a_3a_4b_1b_2b_3b_4(c_1c_2+c_1c_2c_3)+1+c_1)}{a_1a_2a_3a_4b_1b_2b_3b_4(c_1+c_1c_2+c_1c_2c_3)+1}$

$\displaystyle a_3'=\frac{\widehat{\lambda}_3(a,b,c)}{b_2\widehat{\lambda}_2(a,b,c)}=\frac{a_3(a_1a_2a_3a_4b_1b_2b_3b_4c_1c_2c_3+1+c_1+c_1c_2)}{a_1a_2a_3a_4b_1b_2b_3b_4(c_1c_2+c_1c_2c_3)+1+c_1}$

$\displaystyle a_4'=\frac{\widehat{\lambda}_4(a,b,c)}{b_3\widehat{\lambda}_3(a,b,c)}=\frac{a_1a_2a_3a_4^2b_1b_2b_3b_4(1+c_1+c_1c_2+c_1c_2c_3)}{a_1a_2a_3a_4b_1b_2b_3b_4c_1c_2c_3+1+c_1+c_1c_2}$

$\displaystyle b_1'=\frac{\widehat{\lambda}_2(a,b,c)}{a_2\widehat{\lambda}_1(a,b,c)}=\frac{b_1(a_1a_2a_3a_4b_1b_2b_3b_4(c_1c_2+c_1c_2c_3)+1+c_1}{a_1a_2a_3a_4b_1b_2b_3b_4(c_1+c_1c_2+c_1c_2c_3)+1}$

$\displaystyle b_2'=\frac{\widehat{\lambda}_3(a,b,c)}{a_3\widehat{\lambda}_2(a,b,c)}=\frac{b_2(a_1a_2a_3a_4b_1b_2b_3b_4c_1c_2c_3+1+c_1+c_1c_2)}{a_1a_2a_3a_4b_1b_2b_3b_4(c_1c_2+c_1c_2c_3)+1+c_1}$

$\displaystyle b_3'=\frac{\widehat{\lambda}_4(a,b,c)}{a_4\widehat{\lambda}_3(a,b,c)}=\frac{b_3(1+c_1+c_1c_2+c_1c_2c_3)}{a_1a_2a_3a_4b_1b_2b_3b_4c_1c_2c_3+1+c_1+c_1c_2}$

$\displaystyle b_4'=\frac{\widehat{\lambda}_1(a,b,c)}{a_1\widehat{\lambda}_4(a,b,c)}=\frac{a_1a_2a_3a_4b_1b_2b_3b_4(c_1+c_1c_2+c_1c_2c_3)+1}{a_1a_2a_3a_4b_1b_2b_3(1+c_1+c_1c_2+c_1c_2c_3)}$

$\displaystyle c_1'=\frac{a_1a_2b_1b_4c_1\widehat{\lambda}_4(a,b,c)}{\widehat{\lambda}_2(a,b,c)}=\frac{a_1a_2a_3a_4b_1b_2b_3b_4c_1(1+c_1+c_1c_2+c_1c_2c_3)}{a_1a_2a_3a_4b_1b_2b_3b_4(c_1c_2+c_1c_2c_3)+1+c_1}$

$\displaystyle c_2'=\frac{a_2a_3b_1b_2c_2\widehat{\lambda}_1(a,b,c)}{\widehat{\lambda}_3(a,b,c)}=\frac{c_2(a_1a_2a_3a_4b_1b_2b_3b_4(c_1+c_1c_2+c_1c_2c_3)+1)}{a_1a_2a_3a_4b_1b_2b_3b_4c_1c_2c_3+1+c_1+c_1c_2}$

$\displaystyle c_3'=\frac{a_3a_4b_2b_3c_3\widehat{\lambda}_2(a,b,c)}{\widehat{\lambda}_4(a,b,c)}=\frac{c_3(a_1a_2a_3a_4b_1b_2b_3b_4(c_1c_2+c_1c_2c_3)+1+c_1)}{1+c_1+c_1c_2+c_1c_2c_3}$

$\displaystyle c_4'=\frac{a_1a_4b_3b_4c_4\widehat{\lambda}_3(a,b,c)}{\widehat{\lambda}_1(a,b,c)}=\frac{a_1a_2a_3a_4b_1b_2b_3b_4c_4(a_1a_2a_3a_4b_1b_2b_3b_4c_1c_2c_3+1+c_1+c_1c_2)}{a_1a_2a_3a_4b_1b_2b_3b_4(c_1+c_1c_2+c_1c_2c_3)+1}$
\end{ex}

\vspace{2in}
\begin{ex}\label{ex:GS-face}
We revisit Example~\ref{ex:GS-edge}, but with face variables this time.

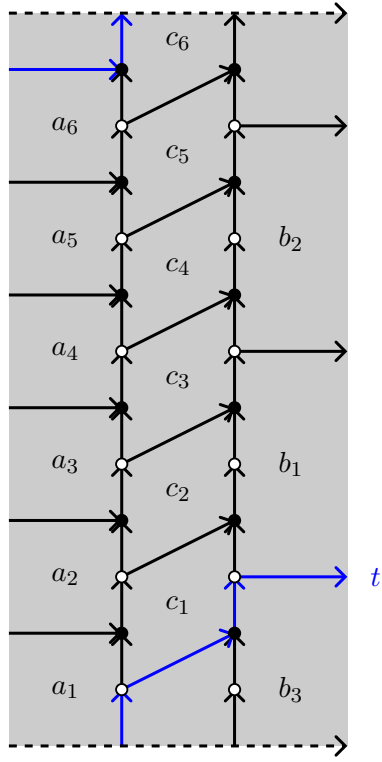
\begin{center}
\begin{tikzpicture}[scale=0.75]
\fill[gray!40!white] (0,0) rectangle (6,13);
\node at (1,1) {$a_1$};
\node at (1,3) {$a_2$};
\node at (1,5) {$a_3$};
\node at (1,7) {$a_4$};
\node at (1,9) {$a_5$};
\node at (1,11) {$a_6$};
\node at (5,5) {$b_1$};
\node at (5,9) {$b_2$};
\node at (5,1) {$b_3$};
\node at (3,2.5) {$c_1$};
\node at (3,4.5) {$c_2$};
\node at (3,6.5) {$c_3$};
\node at (3,8.5) {$c_4$};
\node at (3,10.5) {$c_5$};
\node at (3,12.5) {$c_6$};
\node at (6.5,3) {\textcolor{blue}{$t$}};
\path[->,font=\large, >=angle 90, line width=0.4mm]
(2,1) edge (2,2)
(2,2) edge (2,3)
(2,3) edge (2,4)
(2,4) edge (2,5)
(2,5) edge (2,6)
(2,6) edge (2,7)
(4,0) edge (4,1)
(4,1) edge (4,2)
(4,3) edge (4,4)
(4,4) edge (4,5)
(4,5) edge (4,6)
(4,6) edge (4,7)
(2,3) edge (4,4)
(2,5) edge (4,6)
(0,2) edge (2,2)
(0,4) edge (2,4)
(0,6) edge (2,6)
(2,7) edge (2,8)
(2,8) edge (2,9)
(2,9) edge (2,10)
(2,10) edge (2,11)
(2,11) edge (2,12)
(4,7) edge (4,8)
(4,8) edge (4,9)
(4,9) edge (4,10)
(4,10) edge (4,11)
(4,11) edge (4,12)
(4,12) edge (4,13)
(2,7) edge (4,8)
(2,9) edge (4,10)
(2,11) edge (4,12)
(0,8) edge (2,8)
(0,10) edge (2,10)
(4,7) edge (6,7)
(4,11) edge (6,11);
\path[->,font=\large, >=angle 90, line width=0.4mm, blue]
(0,12) edge (2,12)
(2,12) edge (2,13)
(2,0) edge (2,1)
(2,1) edge (4,2)
(4,2) edge (4,3)
(4,3) edge (6,3);
\path[dashed,->,font=\large, >=angle 90, line width=0.4mm]
(0,0) edge (6,0)
(0,13) edge (6,13);
\draw [line width=0.25mm, fill=white] (2,1) circle (1mm);
\draw [line width=0.25mm, fill=white] (2,3) circle (1mm);
\draw [line width=0.25mm, fill=white] (2,5) circle (1mm);
\draw [line width=0.25mm, fill=white] (4,1) circle (1mm);
\draw [line width=0.25mm, fill=white] (4,3) circle (1mm);
\draw [line width=0.25mm, fill=white] (4,5) circle (1mm);
\draw [line width=0.25mm, fill=black] (2,2) circle (1mm);
\draw [line width=0.25mm, fill=black] (2,4) circle (1mm);
\draw [line width=0.25mm, fill=black] (2,6) circle (1mm);
\draw [line width=0.25mm, fill=black] (4,2) circle (1mm);
\draw [line width=0.25mm, fill=black] (4,4) circle (1mm);
\draw [line width=0.25mm, fill=black] (4,6) circle (1mm);
\draw [line width=0.25mm, fill=white] (2,7) circle (1mm);
\draw [line width=0.25mm, fill=white] (2,9) circle (1mm);
\draw [line width=0.25mm, fill=white] (2,11) circle (1mm);
\draw [line width=0.25mm, fill=white] (4,7) circle (1mm);
\draw [line width=0.25mm, fill=white] (4,9) circle (1mm);
\draw [line width=0.25mm, fill=white] (4,11) circle (1mm);
\draw [line width=0.25mm, fill=black] (2,8) circle (1mm);
\draw [line width=0.25mm, fill=black] (2,10) circle (1mm);
\draw [line width=0.25mm, fill=black] (2,12) circle (1mm);
\draw [line width=0.25mm, fill=black] (4,8) circle (1mm);
\draw [line width=0.25mm, fill=black] (4,10) circle (1mm);
\draw [line width=0.25mm, fill=black] (4,12) circle (1mm);
\end{tikzpicture}
\end{center}

$T_f$ gives us the following values:

$\displaystyle a_1'=\frac{\widehat{\lambda}_1(a,b,c)}{\widehat{\lambda}_6(a,b,c)}=\frac{a_1a_2a_3a_4a_5a_6b_1b_2b_3(c_1+c_1c_2+c_1c_2c_3+c_1c_2c_3c_4+c_1c_2c_3c_4c_5)+1}{a_2a_3a_4a_5a_6b_1b_2b_3(1+c_1+c_1c_2+c_1c_2c_3+c_1c_2c_3c_4+c_1c_2c_3c_4c_5)}$

$\displaystyle a_2'=\frac{\widehat{\lambda}_2(a,b,c)}{b_1\widehat{\lambda}_1(a,b,c)}=\frac{a_2(a_1a_2a_3a_4a_5a_6b_1b_2b_3(c_1c_2+c_1c_2c_3+c_1c_2c_3c_4+c_1c_2c_3c_4c_5)+1+c_1)}{a_1a_2a_3a_4a_5a_6b_1b_2b_3(c_1+c_1c_2+c_1c_2c_3+c_1c_2c_3c_4+c_1c_2c_3c_4c_5)+1}$

$\displaystyle a_3'=\frac{\widehat{\lambda}_3(a,b,c)}{\widehat{\lambda}_2(a,b,c)}=\frac{a_3(a_1a_2a_3a_4a_5a_6b_1b_2b_3(c_1c_2c_3+c_1c_2c_3c_4+c_1c_2c_3c_4c_5)+1+c_1+c_1c_2)}{a_1a_2a_3a_4a_5a_6b_1b_2b_3(c_1c_2+c_1c_2c_3+c_1c_2c_3c_4+c_1c_2c_3c_4c_5)+1+c_1}$

$\displaystyle a_4'=\frac{\widehat{\lambda}_4(a,b,c)}{b_2\widehat{\lambda}_3(a,b,c)}=\frac{a_4(a_1a_2a_3a_4a_5a_6b_1b_2b_3(c_1c_2c_3c_4+c_1c_2c_3c_4c_5)+1+c_1+c_1c_2+c_1c_2c_3)}{a_1a_2a_3a_4a_5a_6b_1b_2b_3(c_1c_2c_3+c_1c_2c_3c_4+c_1c_2c_3c_4c_5)+1+c_1+c_1c_2}$

$\displaystyle a_5'=\frac{\widehat{\lambda}_5(a,b,c)}{\widehat{\lambda}_4(a,b,c)}=\frac{a_5(a_1a_2a_3a_4a_5a_6b_1b_2b_3c_1c_2c_3c_4c_5+1+c_1+c_1c_2+c_1c_2c_3+c_1c_2c_3c_4)}{a_1a_2a_3a_4a_5a_6b_1b_2b_3(c_1c_2c_3c_4+c_1c_2c_3c_4c_5)+1+c_1+c_1c_2+c_1c_2c_3}$

$\displaystyle a_6'=\frac{\widehat{\lambda}_6(a,b,c)}{b_3\widehat{\lambda}_5(a,b,c)}=\frac{a_6(1+c_1+c_1c_2+c_1c_2c_3+c_1c_2c_3c_4+c_1c_2c_3c_4c_5)}{a_1a_2a_3a_4a_5a_6b_1b_2b_3c_1c_2c_3c_4c_5+1+c_1+c_1c_2+c_1c_2c_3+c_1c_2c_3c_4}$

$\displaystyle b_1'=\frac{\widehat{\lambda}_3(a,b,c)}{a_2a_3\widehat{\lambda}_1(a,b,c)}=\frac{b_1(a_1a_2a_3a_4a_5a_6b_1b_2b_3(c_1c_2c_3+c_1c_2c_3c_4+c_1c_2c_3c_4c_5)+1+c_1+c_1c_2)}{a_1a_2a_3a_4a_5a_6b_1b_2b_3(c_1+c_1c_2+c_1c_2c_3+c_1c_2c_3c_4+c_1c_2c_3c_4c_5)+1}$

$\displaystyle b_2'=\frac{\widehat{\lambda}_5(a,b,c)}{a_4a_5\widehat{\lambda}_3(a,b,c)}=\frac{b_2(a_1a_2a_3a_4a_5a_6b_1b_2b_3c_1c_2c_3c_4c_5+1+c_1+c_1c_2+c_1c_2c_3+c_1c_2c_3c_4)}{a_1a_2a_3a_4a_5a_6b_1b_2b_3(c_1c_2c_3+c_1c_2c_3c_4+c_1c_2c_3c_4c_5)+1+c_1+c_1c_2}$

$\displaystyle b_3'=\frac{\widehat{\lambda}_1(a,b,c)}{a_6a_1\widehat{\lambda}_5(a,b,c)}$

\hspace{5mm}$\displaystyle=\frac{b_3(a_1a_2a_3a_4a_5a_6b_1b_2b_3(c_1+c_1c_2+c_1c_2c_3+c_1c_2c_3c_4+c_1c_2c_3c_4c_5)+1)}{a_1a_2a_3a_4a_5a_6b_1b_2b_3(a_1a_2a_3a_4a_5a_6b_1b_2b_3c_1c_2c_3c_4c_5+1+c_1+c_1c_2+c_1c_2c_3+c_1c_2c_3c_4)}$

$\displaystyle c_1'=\frac{c_1a_1a_2b_1\widehat{\lambda}_6(a,b,c)}{\widehat{\lambda}_2(a,b,c)}=\frac{c_1a_1a_2a_3a_4a_5a_6b_1b_2b_3(1+c_1+c_1c_2+c_1c_2c_3+c_1c_2c_3c_4+c_1c_2c_3c_4c_5)}{a_1a_2a_3a_4a_5a_6b_1b_2b_3(c_1c_2+c_1c_2c_3+c_1c_2c_3c_4+c_1c_2c_3c_4c_5)+1+c_1}$

$\displaystyle c_2'=\frac{c_2a_2a_3b_1\widehat{\lambda}_1(a,b,c)}{\widehat{\lambda}_3(a,b,c)}=\frac{c_2(a_1a_2a_3a_4a_5a_6b_1b_2b_3(c_1+c_1c_2+c_1c_2c_3+c_1c_2c_3c_4+c_1c_2c_3c_4c_5)+1)}{a_1a_2a_3a_4a_5a_6b_1b_2b_3(c_1c_2c_3+c_1c_2c_3c_4+c_1c_2c_3c_4c_5)+1+c_1+c_1c_2}$

$\displaystyle c_3'=\frac{c_3a_3a_4b_2\widehat{\lambda}_2(a,b,c)}{\widehat{\lambda}_4(a,b,c)}=\frac{c_3(a_1a_2a_3a_4a_5a_6b_1b_2b_3(c_1c_2+c_1c_2c_3+c_1c_2c_3c_4+c_1c_2c_3c_4c_5)+1+c_1)}{a_1a_2a_3a_4a_5a_6b_1b_2b_3(c_1c_2c_3c_4+c_1c_2c_3c_4c_5)+1+c_1+c_1c_2+c_1c_2c_3}$

$\displaystyle c_4'=\frac{c_4a_4a_5b_2\widehat{\lambda}_3(a,b,c)}{\widehat{\lambda}_5(a,b,c)}=\frac{c_4(a_1a_2a_3a_4a_5a_6b_1b_2b_3(c_1c_2c_3+c_1c_2c_3c_4+c_1c_2c_3c_4c_5)+1+c_1+c_1c_2)}{a_1a_2a_3a_4a_5a_6b_1b_2b_3c_1c_2c_3c_4c_5+1+c_1+c_1c_2+c_1c_2c_3+c_1c_2c_3c_4}$

$\displaystyle c_5'=\frac{c_5a_5a_6b_3\widehat{\lambda}_4(a,b,c)}{\widehat{\lambda}_6(a,b,c)}=\frac{c_5(a_1a_2a_3a_4a_5a_6b_1b_2b_3(c_1c_2c_3c_4+c_1c_2c_3c_4c_5)+1+c_1+c_1c_2+c_1c_2c_3)}{1+c_1+c_1c_2+c_1c_2c_3+c_1c_2c_3c_4+c_1c_2c_3c_4c_5}$

$\displaystyle c_6'=\frac{c_6a_6a_1b_3\widehat{\lambda}_5(a,b,c)}{\widehat{\lambda}_1(a,b,c)}$

\hspace{5mm}$\displaystyle=\frac{a_1a_2a_3a_4a_5a_6b_1b_2b_3c_1c_2c_3c_4c_5+1+c_1+c_1c_2+c_1c_2c_3+c_1c_2c_3c_4}{c_1c_2c_3c_4c_5(a_1a_2a_3a_4a_5a_6b_1b_2b_3(c_1+c_1c_2+c_1c_2c_3+c_1c_2c_3c_4+c_1c_2c_3c_4c_5)+1)}$
\end{ex}

\section{Cluster Algebra Background}\label{sec:bg2}

\begin{defn}\label{defn:quiver}
A \emph{quiver} $Q$ is a directed graph with vertices labeled $1,...,n$ and no loops or 2-cycles.
\end{defn}

\begin{defn}\label{defn:quiver-mutation}
If $k$ is a vertex in a quiver $Q$, a \emph{quiver mutation} at $k$, $\mu_k(Q)$ is defined from $Q$ as follows:
\begin{enumerate}[(1)]
\item for each pair of edges $i\to k$ and $k\to j$, add a new edge $i\to j$,

\item reverse any edges incident to $k$,

\item remove any 2-cycles.
\end{enumerate}
\end{defn}

\begin{defn}\label{defn:seed}
A \emph{seed} is a pair $(Q,{\bf x})$ where $Q$ is a quiver and ${\bf x}=(x_1,...,x_n)$ with $n$ the number of vertices of $Q$.
\end{defn}

\begin{defn}\label{defn:seed-mutation}
If $k$ is a vertex in $Q$, we can define a \emph{seed mutation} of $(Q,{\bf x})$ at $k$, $\mu_k(Q,{\bf x})=(Q',{\bf x}')$ by $Q'=\mu_k(Q)$ and ${\bf x}'=(x_1',...,x_n')$ where $$x_i'=\begin{cases} x_i & i\neq k, \\ \prod_{j=1}^n x_j^{\#\{\text{edges }k\to j\text{ in }Q\}}+\prod_{j=1}^n x_j^{\#\{\text{edges }j\to k\text{ in }Q\}} & i=k. \end{cases}$$  Recall that $Q$ has no 2-cycles, so $x_j$ shows up in at most one of the products in the formula for $x_k'$.
\end{defn}

\begin{defn}\label{defn:seed-with-coeff}
A \emph{seed with coefficients} is a triple $(Q,{\bf x},{\bf y})$ where $Q$ is a quiver, ${\bf x}=(x_1,...,x_n)$, and ${\bf y}=(y_1,...,y_n)$  with $n$ the number of vertices of $Q$.
\end{defn}

\begin{defn}\label{defn:seed-with-coeff-mutation}
If $k$ is a vertex in $Q$, we can define a \emph{seed mutation} of $(Q,{\bf x},{\bf y})$ at $k$, $\mu_k(Q,{\bf x},{\bf y})=(Q',{\bf x}',{\bf y}')$ by $Q'=\mu_k(Q),{\bf x}'=(x_1',...,x_n')$, and ${\bf y}'=(y_1',...,y_n')$ where $$x_i'=\begin{cases} x_i & i\neq k, \\ \prod_{j=1}^n x_j^{\#\{\text{edges }k\to j\text{ in }Q\}}+\prod_{j=1}^n x_j^{\#\{\text{edges }j\to k\text{ in }Q\}} & i=k, \end{cases}$$ $$y_i'=\begin{cases} y_k^{-1} & i=k \\ y_i(1+y_k^{-1})^{-\#\{\text{edges }k\to i\text{ in }Q\}} & i\neq k, \#\{\text{edges }k\to i\text{ in }Q\}\geq 0, \\ y_i(1+y_k)^{\#\{\text{edges }i\to k\text{ in }Q\}} & i\neq k, \#\{\text{edges }i\to k\text{ in }Q\}\geq 0. \end{cases}$$
\end{defn}

\begin{defn}\label{defn:y-seed}
A \emph{$y$-seed}, $(Q,{\bf y})$ and its mutations are defined as above, but without the {\bf x} variables.
\end{defn}

Frequently the $y$-dynamics above will be defined over a semifield that may have a different addition.  However, for our purposes, we will use these less general definitions.

\section{Spider Web Quivers}\label{sec:quivers}

In~\cite{ILP}, Inoue, Lam, and Pylyavskyy obtain the geometric R-matrix from a sequence of cluster mutations of a triangular grid quiver.  Triangular grid quivers are exactly the dual quivers to the plabic graph associated to certain wiring diagrams, such as the plabic graph in Example~\ref{ex:LP-edge}.  In~\cite{GS}, Goncharov and Shen show that in another family of quivers, the same mutation sequence gives a Weyl group action.  The rest of this section will explore how these cases generalize.

\begin{defn}\label{defn:spider-web-quiver}
A \emph{spider web quiver} is a quiver constructed as follows.  We begin with three or more concentric circles.  Place as many vertices as desired (at least 2 so as to avoid loops) on each circle.  Orient the edges of each circle counter clockwise.  Then add edges in the diagram between vertices on adjacent circles so that each face is oriented, contains two edges between circles, and has at least 3 sides.

For the purposes of this paper, we will assume a spider web quiver has 3 circles, unless stated otherwise.  Label the vertices of the middle circle $1,...,n$ following the arrows around the circle.  If the lowest index vertex in the middle circle with an edge to the outer circle is $i$, label the vertex in the outer circle which has an arrow to $i$ as $1^-$.  Continue labeling the vertices of the outer circle $2^-,3^-,...$ following the arrows around the circle.  If the lowest index vertex in the middle circle with an edge to the inner circle is $j$, label the vertex in the inner circle which has an arrow to $j$ as $1^+$.  Continue labeling the vertices of the outer circle $2^+,3^+,...$ following the arrows around the circle.  See Figure~\ref{fig:spider-web-quiver-ex} for examples.
\end{defn}

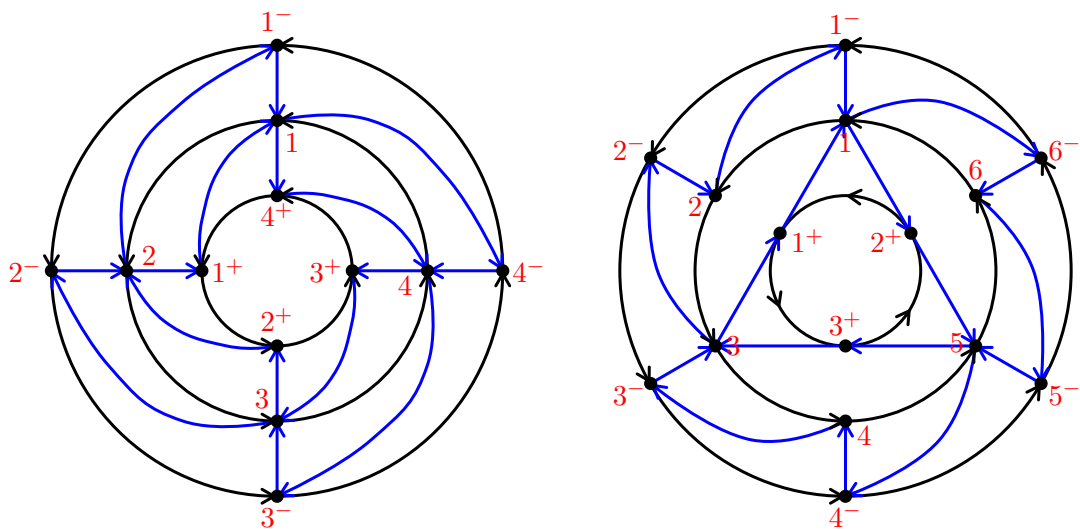
\begin{figure}[htp]
\begin{center}
\begin{tikzpicture}
\draw [line width=0.4mm] (3,3) circle (1cm);
\draw [line width=0.4mm] (3,3) circle (2cm);
\draw [line width=0.4mm] (3,3) circle (3cm);
\draw [line width=0.4mm, blue] plot [smooth, tension=0.8] coordinates { (3,5) (4.25,5) (5.25,4.25) (6,3) };
\draw [line width=0.4mm,blue] plot [smooth, tension=0.8] coordinates { (5,3) (5,1.75) (4.25,0.75) (3,0) };
\draw [line width=0.4mm,blue] plot [smooth, tension=0.8] coordinates { (3,1) (1.75,1) (0.75,1.75) (0,3) };
\draw [line width=0.4mm,blue] plot [smooth, tension=0.8] coordinates { (1,3) (1,4.25) (1.75,5.25) (3,6) };
\draw [line width=0.4mm,blue] plot [smooth, tension=0.8] coordinates { (3,4) (3.67,4) (4.5,3.67) (5,3) };
\draw [line width=0.4mm,blue] plot [smooth, tension=0.8] coordinates { (4,3) (4,2.33) (3.67,1.5) (3,1) };
\draw [line width=0.4mm,blue] plot [smooth, tension=0.8] coordinates { (3,2) (2.33,2) (1.5,2.33) (1,3) };
\draw [line width=0.4mm,blue] plot [smooth, tension=0.8] coordinates { (2,3) (2,3.67) (2.33,4.5) (3,5) };
\path[->,font=\large, >=angle 60, line width=0.4mm,blue]
(3,6) edge (3,5)
(3,5) edge (3,4)
(0,3) edge (1,3)
(1,3) edge (2,3)
(3,0) edge (3,1)
(3,1) edge (3,2)
(6,3) edge (5,3)
(5,3) edge (4,3)
(5.95,3.1) edge (6,3)
(3.1,0.05) edge (3,0)
(4.95,3.1) edge (5,3)
(3.1,1.05) edge (3,1)
(0.05,2.9) edge (0,3)
(2.9,5.95) edge (3,6)
(1.05,2.9) edge (1,3)
(2.9,4.95) edge (3,5);
\path[->,font=\large, >=angle 60, line width=0.4mm]
(0,3.1) edge (0,3)
(2.9,0) edge (3,0)
(6,2.9) edge (6,3)
(3.1,6) edge (3,6)
(1,3.1) edge (1,3)
(2.9,1) edge (3,1)
(5,2.9) edge (5,3)
(3.1,5) edge (3,5)
(2,3.1) edge (2,3)
(2.9,2) edge (3,2)
(4,2.9) edge (4,3)
(3.1,4) edge (3,4);
\draw [line width=0.25mm, fill=black] (3,6) circle (0.75mm);
\draw [line width=0.25mm, fill=black] (0,3) circle (0.75mm);
\draw [line width=0.25mm, fill=black] (3,0) circle (0.75mm);
\draw [line width=0.25mm, fill=black] (6,3) circle (0.75mm);
\draw [line width=0.25mm, fill=black] (3,5) circle (0.75mm);
\draw [line width=0.25mm, fill=black] (1,3) circle (0.75mm);
\draw [line width=0.25mm, fill=black] (3,1) circle (0.75mm);
\draw [line width=0.25mm, fill=black] (5,3) circle (0.75mm);
\draw [line width=0.25mm, fill=black] (3,4) circle (0.75mm);
\draw [line width=0.25mm, fill=black] (2,3) circle (0.75mm);
\draw [line width=0.25mm, fill=black] (3,2) circle (0.75mm);
\draw [line width=0.25mm, fill=black] (4,3) circle (0.75mm);
\node at (3.2,4.7) {\textcolor{red}{$1$}};
\node at (1.3,3.2) {\textcolor{red}{$2$}};
\node at (2.8,1.3) {\textcolor{red}{$3$}};
\node at (4.7,2.8) {\textcolor{red}{$4$}};
\node at (2.35,3) {\textcolor{red}{$1^+$}};
\node at (3,2.35) {\textcolor{red}{$2^+$}};
\node at (3.65,3) {\textcolor{red}{$3^+$}};
\node at (3,3.75) {\textcolor{red}{$4^+$}};
\node at (3,6.35) {\textcolor{red}{$1^-$}};
\node at (-0.35,3) {\textcolor{red}{$2^-$}};
\node at (3,-0.25) {\textcolor{red}{$3^-$}};
\node at (6.35,3) {\textcolor{red}{$4^-$}};
\end{tikzpicture}
\hspace{0.15in}
\begin{tikzpicture}
\draw [line width=0.4mm] (3,3) circle (1cm);
\draw [line width=0.4mm] (3,3) circle (2cm);
\draw [line width=0.4mm] (3,3) circle (3cm);
\draw [line width=0.4mm, blue] plot [smooth, tension=0.8] coordinates { (3,5) (4.25,5.25) (5.6,4.5) };
\draw [line width=0.4mm, blue] plot [smooth, tension=0.8] coordinates { (4.73,4) (5.5,3) (5.6,1.5) };
\draw [line width=0.4mm, blue] plot [smooth, tension=0.8] coordinates { (4.73,2) (4.25,0.75) (3,0) };
\draw [line width=0.4mm, blue] plot [smooth, tension=0.8] coordinates { (3,1) (1.75,0.75) (0.41,1.5) };
\draw [line width=0.4mm, blue] plot [smooth, tension=0.8] coordinates { (1.27,2) (0.5,3) (0.41,4.5) };
\draw [line width=0.4mm, blue] plot [smooth, tension=0.8] coordinates { (1.27,4) (1.75,5.25) (3,6) };
\path[->,font=\large, >=angle 60, line width=0.4mm,blue]
(3,6) edge (3,5)
(5.6,4.5) edge (4.73,4)
(5.6,1.5) edge (4.73,2)
(3,0) edge (3,1)
(0.41,1.5) edge (1.27,2)
(0.41,4.5) edge (1.27,4)
(2.9,5.95) edge (3,6)
(5.5,4.58) edge (5.6,4.5)
(5.6,1.6) edge (5.6,1.5)
(3.1,0.05) edge (3,0)
(0.51,1.42) edge (0.41,1.5)
(0.41,4.4) edge (0.41,4.5)
(1.27,2) edge (2.13,3.5)
(2.13,3.5) edge (3,5)
(3,5) edge (3.87,3.5)
(3.87,3.5) edge (4.73,2)
(4.73,2) edge (3,2)
(3,2) edge (1.27,2);
\path[->,font=\large, >=angle 60, line width=0.4mm]
(3.1,6) edge (3,6)
(0.44,4.55) edge (0.41,4.5)
(0.37,1.57) edge (0.41,1.5)
(2.9,0) edge (3,0)
(5.54,1.41) edge (5.6,1.5)
(5.65,4.41) edge (5.6,4.5)
(1.3,4.05) edge (1.27,4)
(1.24,2.05) edge (1.27,2)
(2.9,1) edge (3,1)
(4.7,1.95) edge (4.73,2)
(4.76,3.95) edge (4.73,4)
(3.1,5) edge (3,5)
(3.1,3.99) edge (3,4)
(2.09,2.6) edge (2.13,2.5)
(3.8,2.4) edge (3.87,2.5);
\draw [line width=0.25mm, fill=black] (3,6) circle (0.75mm);
\draw [line width=0.25mm, fill=black] (0.41,4.5) circle (0.75mm);
\draw [line width=0.25mm, fill=black] (0.41,1.5) circle (0.75mm);
\draw [line width=0.25mm, fill=black] (3,0) circle (0.75mm);
\draw [line width=0.25mm, fill=black] (5.6,1.5) circle (0.75mm);
\draw [line width=0.25mm, fill=black] (5.6,4.5) circle (0.75mm);
\draw [line width=0.25mm, fill=black] (1.27,2) circle (0.75mm);
\draw [line width=0.25mm, fill=black] (3,5) circle (0.75mm);
\draw [line width=0.25mm, fill=black] (1.27,4) circle (0.75mm);
\draw [line width=0.25mm, fill=black] (3,1) circle (0.75mm);
\draw [line width=0.25mm, fill=black] (4.73,2) circle (0.75mm);
\draw [line width=0.25mm, fill=black] (4.73,4) circle (0.75mm);
\draw [line width=0.25mm, fill=black] (3,2) circle (0.75mm);
\draw [line width=0.25mm, fill=black] (2.13,3.5) circle (0.75mm);
\draw [line width=0.25mm, fill=black] (3.87,3.5) circle (0.75mm);
\node at (3,6.3) {\textcolor{red}{$1^-$}};
\node at (0.13,4.63) {\textcolor{red}{$2^-$}};
\node at (0.13,1.37) {\textcolor{red}{$3^-$}};
\node at (3,-0.25) {\textcolor{red}{$4^-$}};
\node at (5.92,1.37) {\textcolor{red}{$5^-$}};
\node at (5.92,4.63) {\textcolor{red}{$6^-$}};
\node at (3,4.7) {\textcolor{red}{$1$}};
\node at (1,3.85) {\textcolor{red}{$2$}};
\node at (1.51,2) {\textcolor{red}{$3$}};
\node at (3.25,0.8) {\textcolor{red}{$4$}};
\node at (4.48,2.05) {\textcolor{red}{$5$}};
\node at (4.73,4.35) {\textcolor{red}{$6$}};
\node at (2.5,3.4) {\textcolor{red}{$1^+$}};
\node at (3.55,3.4) {\textcolor{red}{$2^+$}};
\node at (3,2.3) {\textcolor{red}{$3^+$}};
\end{tikzpicture}
\end{center}
\caption{Two spider web quivers.  The quiver on the left is a triangular grid quiver, as in~\cite{ILP}.  The quiver on the right is a 3-triangulation with 3 ideal triangles, as in Section 8 of~\cite{GS}.}
\label{fig:spider-web-quiver-ex}
\end{figure}

We will always let $n$ be the number of vertices on the middle circle of a spider web quiver.  We'll be considering the transformation $\tau:=\mu_1\mu_2...\mu_{n-2}s_{n-1,n}\mu_n\mu_{n-1}...\mu_1$ where $s_{n-1,n}$ is the operation that transposes vertices $n$ and $n-1$.

\begin{prop}\label{prop:mutation-seq}
If $Q$ is a spider web quiver, then $\tau(Q)=Q.$
\end{prop}

For proof, see Section~\ref{sec:proof-mutation-seq}.

For a spider web quiver, let $\alpha$ be the largest vertex on the outer circle connected to in middle circle and $\beta$ be the largest vertex on the inner circle connected to the middle circle.

Let $x_{[i]^-}:=\begin{cases} x_a & \text{if there is at least one vertex on the outer circle with an edge to an}\\ & \text{element of }\{1,...,i\},\text{ and the largest such vertex is }a, \\ 1 & \text{otherwise}.\end{cases}$

Let $x_{[i]^+}:=\begin{cases} x_b & \text{if there is at least one vertex on the inner circle with an edge to an}\\ & \text{element of }\{1,...,i\},\text{ and the largest such vertex is }b, \\ 1 & \text{otherwise}.\end{cases}$

Let $x_{[i]^-_*}:=\begin{cases} x_a & \text{if there is at least one vertex on the outer circle with an edge to an}\\ & \text{element of }\{1,...,i\},\text{ and the largest such vertex is }a, \\ x_\alpha & \text{otherwise}.\end{cases}$

Let $x_{[i]^+_*}:=\begin{cases} x_b & \text{if there is at least one vertex on the inner circle with an edge to an}\\ & \text{element of }\{1,...,i\},\text{ and the largest such vertex is }b, \\ x_\beta & \text{otherwise}.\end{cases}$

\begin{thm}\label{thm:x-vars}
For a spider web quiver $Q$, let $$\overline{x}=\frac{\sum_{j=1}^{n-1}\left(\left(\prod_{k=1}^{j-1} x_k\right)\left(\prod_{k=j+2}^n x_k\right)x_{[j]^-_*}x_{[j]^+_*}\right)+\left(\prod_{k=2}^{n-1} x_k\right)x_{[n]^-}x_{[n]^+}}{\prod_{k=1}^n x_k}.$$  Then, $\tau(x_i)=x_i\overline{x}.$
\end{thm}

For proof, see Section~\ref{sec:proof-x-vars}.

\begin{ex}\label{ex:LP-x}
Consider the quiver on the left of Figure~\ref{fig:spider-web-quiver-ex}.  In this case, we have
\begin{align*}
\overline{x}&=\frac{x_3x_4x_{[1]^-_*}x_{[1]^+_*}+x_1x_4x_{[2]^-_*}x_{[2]^+_*}+x_1x_2x_{[3]^-_*}x_{[3]^+_*}+x_2x_3x_{[4]^-}x_{[4]^+}}{x_1x_2x_3x_4}\\
&=\frac{x_3x_4x_{1^-}x_{1^+}+x_1x_4x_{2^-}x_{2^+}+x_1x_2x_{3^-}x_{3^+}+x_2x_3x_{4^-}x_{4^+}}{x_1x_2x_3x_4}
\end{align*}
So, the $x$-variables after applying $\tau$ are as follows:

$\displaystyle x_1'=\frac{x_3x_4x_{1^-}x_{1^+}+x_1x_4x_{2^-}x_{2^+}+x_1x_2x_{3^-}x_{3^+}+x_2x_3x_{4^-}x_{4^+}}{x_2x_3x_4}$

$\displaystyle x_2'=\frac{x_3x_4x_{1^-}x_{1^+}+x_1x_4x_{2^-}x_{2^+}+x_1x_2x_{3^-}x_{3^+}+x_2x_3x_{4^-}x_{4^+}}{x_1x_3x_4}$

$\displaystyle x_3'=\frac{x_3x_4x_{1^-}x_{1^+}+x_1x_4x_{2^-}x_{2^+}+x_1x_2x_{3^-}x_{3^+}+x_2x_3x_{4^-}x_{4^+}}{x_1x_2x_4}$

$\displaystyle x_4'=\frac{x_3x_4x_{1^-}x_{1^+}+x_1x_4x_{2^-}x_{2^+}+x_1x_2x_{3^-}x_{3^+}+x_2x_3x_{4^-}x_{4^+}}{x_1x_2x_3}$

All other $x$-variables remain the same because there are no mutations are the corresponding vertices.
\end{ex}

\begin{ex}\label{ex:GS-x}
Consider the quiver on the right of Figure~\ref{fig:spider-web-quiver-ex}.  In this case, we have

$\overline{x}=\frac{x_3x_4x_5x_6x_{1^-}x_{1^+}+x_1x_4x_5x_6x_{2^-}x_{1^+}+x_1x_2x_5x_6x_{3^-}x_{2^+}+x_1x_2x_3x_6x_{4^-}x_{2^+}+x_1x_2x_3x_4x_{5^-}x_{3^+}+x_2x_3x_4x_5x_{6^-}x_{3^+}}{x_1x_2x_3x_4x_5x_6}$

So, the $x$-variables after applying $\tau$ are as follows:

$x_1'=\frac{x_3x_4x_5x_6x_{1^-}x_{1^+}+x_1x_4x_5x_6x_{2^-}x_{1^+}+x_1x_2x_5x_6x_{3^-}x_{2^+}+x_1x_2x_3x_6x_{4^-}x_{2^+}+x_1x_2x_3x_4x_{5^-}x_{3^+}+x_2x_3x_4x_5x_{6^-}x_{3^+}}{x_2x_3x_4x_5x_6}$

$x_2'=\frac{x_3x_4x_5x_6x_{1^-}x_{1^+}+x_1x_4x_5x_6x_{2^-}x_{1^+}+x_1x_2x_5x_6x_{3^-}x_{2^+}+x_1x_2x_3x_6x_{4^-}x_{2^+}+x_1x_2x_3x_4x_{5^-}x_{3^+}+x_2x_3x_4x_5x_{6^-}x_{3^+}}{x_1x_3x_4x_5x_6}$

$x_3'=\frac{x_3x_4x_5x_6x_{1^-}x_{1^+}+x_1x_4x_5x_6x_{2^-}x_{1^+}+x_1x_2x_5x_6x_{3^-}x_{2^+}+x_1x_2x_3x_6x_{4^-}x_{2^+}+x_1x_2x_3x_4x_{5^-}x_{3^+}+x_2x_3x_4x_5x_{6^-}x_{3^+}}{x_1x_2x_4x_5x_6}$

$x_4'=\frac{x_3x_4x_5x_6x_{1^-}x_{1^+}+x_1x_4x_5x_6x_{2^-}x_{1^+}+x_1x_2x_5x_6x_{3^-}x_{2^+}+x_1x_2x_3x_6x_{4^-}x_{2^+}+x_1x_2x_3x_4x_{5^-}x_{3^+}+x_2x_3x_4x_5x_{6^-}x_{3^+}}{x_1x_2x_3x_5x_6}$

$x_5'=\frac{x_3x_4x_5x_6x_{1^-}x_{1^+}+x_1x_4x_5x_6x_{2^-}x_{1^+}+x_1x_2x_5x_6x_{3^-}x_{2^+}+x_1x_2x_3x_6x_{4^-}x_{2^+}+x_1x_2x_3x_4x_{5^-}x_{3^+}+x_2x_3x_4x_5x_{6^-}x_{3^+}}{x_1x_2x_3x_4x_6}$

$x_6'=\frac{x_3x_4x_5x_6x_{1^-}x_{1^+}+x_1x_4x_5x_6x_{2^-}x_{1^+}+x_1x_2x_5x_6x_{3^-}x_{2^+}+x_1x_2x_3x_6x_{4^-}x_{2^+}+x_1x_2x_3x_4x_{5^-}x_{3^+}+x_2x_3x_4x_5x_{6^-}x_{3^+}}{x_1x_2x_3x_4x_5}$

All other $x$-variables remain the same because there are no mutations are the corresponding vertices.
\end{ex}

\begin{thm}\label{thm:x-involution}
For a spider web quiver, $\tau$ is an involution on the $x$-variables.
\end{thm}
\begin{proof}
We know that for each $k$, $\mu_k$ is an involution and also $s_{n-1,n}$ is an involution.  This means that for each $i$, $$\mu_1\mu_2...\mu_ns_{n-1,n}\mu_{n-2}\mu_{n-3}...\mu_1\mu_1\mu_2...\mu_{n-2}s_{n-1,n}\mu_n\mu_{n-1}...\mu_1(x_i)=x_i.$$  Since $\mu_ns_{n-1,n}=s_{n-1,n}\mu_{n-1}$ and $\mu_{n-1}s_{n-1,n}=s_{n-1,n}\mu_n$, we have $$\left(\mu_1\mu_2...\mu_{n-2}s_{n-1,n}\mu_n\mu_{n-1}...\mu_1\right)^2(x_i)=x_i.$$ 
\end{proof}

\begin{thm}\label{thm:x-braid}
Suppose we have a spider web quiver with any number of concentric circles.  Choose 2 adjacent circles.  If we apply $\tau$ to one circle, then the other, then the first again, we obtain the same $x$-variables as if we had applied $\tau$ to the second circle, then the first, then the second again.
\end{thm}
\begin{proof}
Suppose the first circle has vertices labeled with variables $a_1,...,a_n$ and the second circle has vertices labeled with $b_1,...,b_m$.  Without loss of generality, we can assume the first circle is outside the second.  If we apply the involution to the first circle, the only variables that change are $a_1,...,a_n$, which get replaced with $\overline{a}a_1,...,\overline{a}a_n$.  We denote the new variables with primes, that is $a_i'=\overline{a}a_i$ and $b_i'=b_i$.

Next we apply the involution to the second circle.  This means $b_1',...,b_m'$ get replaced with $\overline{b'}b_1',...,\overline{b'}b_m'$.  Let's compute $\overline{b'}$.
\begin{align*}
\overline{b'}&=\frac{\sum_{j=1}^{m-1}\left(\left(\prod_{k=1}^{j-1} b_k'\right)\left(\prod_{k=j+2}^m b_k'\right)b_{[j]^-_*}'b_{[j]^+_*}'\right)+\left(\prod_{k=2}^{m-1} b_k'\right)b_{[m]^-}'b_{[m]^+}'}{\prod_{k=1}^m b_k'}\\
&=\frac{\sum_{j=1}^{m-1}\left(\left(\prod_{k=1}^{j-1} b_k\right)\left(\prod_{k=j+2}^m b_k\right)\overline{a}b_{[j]^-_*}b_{[j]^+_*}\right)+\left(\prod_{k=2}^{m-1} b_k\right)\overline{a}b_{[m]^-}b_{[m]^+}}{\prod_{k=1}^m b_k}\\
&=\overline{a}\overline{b}
\end{align*}
So, $\overline{b'}b_i'=\overline{a}\overline{b}b_i$.  We denote the new variables with additional primes, that is $a_i''=\overline{a}a_i$ and $b_i''=\overline{a}\overline{b}b_i$.

Finally, we apply the involution to the first circle again.  This means $a_1'',...,a_n''$ get replaced with $\overline{a''}a_1'',...,\overline{a''}a_n''$.  Let's compute $\overline{a''}$.
\begin{align*}
\overline{a''}&=\frac{\sum_{j=1}^{m-1}\left(\left(\prod_{k=1}^{j-1} a_k''\right)\left(\prod_{k=j+2}^m a_k''\right)a_{[j]^-_*}''a_{[j]^+_*}''\right)+\left(\prod_{k=2}^{m-1} a_k''\right)a_{[m]^-}''a_{[m]^+}''}{\prod_{k=1}^m a_k''}\\
&=\frac{\sum_{j=1}^{m-1}\left(\left(\prod_{k=1}^{j-1} \overline{a}a_k\right)\left(\prod_{k=j+2}^m \overline{a}a_k\right)a_{[j]^-_*}\overline{a}\overline{b}a_{[j]^+_*}\right)+\left(\prod_{k=2}^{m-1} \overline{a}a_k\right)a_{[m]^-}\overline{a}\overline{b}a_{[m]^+}}{\prod_{k=1}^m \overline{a}a_k}\\
&=\frac{\sum_{j=1}^{m-1}\left(\left(\prod_{k=1}^{j-1} a_k\right)\left(\prod_{k=j+2}^m a_k\right)a_{[j]^-_*}\overline{b}a_{[j]^+_*}\right)+\left(\prod_{k=2}^{m-1} a_k\right)a_{[m]^-}\overline{b}a_{[m]^+}}{\overline{a}\prod_{k=1}^m a_k}\\
&=\overline{b}
\end{align*}
So, $\overline{a''}a_i''=\overline{a}\overline{b}a_i$.

Thus, our variables after applying the involution three times are $\overline{a}\overline{b}a_1,...,\overline{a}\overline{b}a_n$ and $\overline{a}\overline{b}b_1,...,\overline{a}\overline{b}b_m$.  By symmetry, we can see that if we had applied the involution to the second circle, then the first, then the second again, we would have the same variables.
\end{proof}

Now we will discuss the $y$-dynamics of the quiver.  Define $y_0:=1$ for ease of notation.

\begin{thm}\label{thm:y-vars}
For any vertex $i^-$ connected to the middle circle, let $d_i$ be the minimal $j$ so that $i^-$ is connected to $j$ and $e_i$ be the maximal $j$ so that $i^-$ is connected to $j$.  Similarly, For any vertex $i^+$ connected to the middle circle, let $f_i$ be the minimal $j$ so that $i^+$ is connected to $j$ and $g_i$ be the maximal $j$ so that $i^+$ is connected to $j$.  If we apply $\tau$ to a spider web quiver $Q$, we obtain the following $y$-variables:
\begin{enumerate}[(1)]
\item If $1\leq i\leq n-1$, $$y_i'=\frac{\sum_{j=i}^n\prod_{k=i}^{j-1}y_k+\left(\prod_{k=i}^n y_k\right)\left(\sum_{j=0}^{i-2}\prod_{k=0}^j y_k\right)}{\sum_{j=i+1}^{n-1}\prod_{k=i+1}^j y_k+\left(\prod_{k=i+1}^n y_k\right)\left(\sum_{j=0}^{i}\prod_{k=0}^j y_k\right)}$$
\item $\displaystyle y_n'=\frac{1+y_n\sum_{j=0}^{n-2}\prod_{k=0}^j y_k}{\sum_{j=1}^{n-1}\prod_{k=1}^j y_k+\prod_{k=1}^n y_k}$
\item If there are no edges between $i^-$ and the middle circle, then $y_{i^-}'=y_{i^-}$
\item If $i$ is maximal so that there are edges between $i^-$ and the middle circle, then $$y_{i^-}'=y_{i^-}\frac{\left(\prod_{k=e_i+1}^n y_k\right)\left(\sum_{j=d_1}^{n-1}\prod_{k=0}^j y_k+\left(\prod_{k=0}^n y_k\right)\left(\sum_{j=0}^{d_1-1}\prod_{k=0}^j y_k\right)\right)}{\sum_{j=e_i+1}^n\prod_{k=e_i+1}^{j-1} y_k+\left(\prod_{k=e_i+1}^n y_k\right)\left(\sum_{j=0}^{e_i-1}\prod_{k=0}^j y_k\right)}$$
\item Otherwise, $$y_{i^-}'=y_{i^-}\frac{\sum_{j=e_i}^{n-1}\prod_{k=d_i+1}^j y_k+\left(\prod_{k=d_i+1}^n y_k\right)\left(\sum_{j=0}^{e_i-1}\prod_{k=0}^j y_k\right)}{\sum_{j=d_i+1}^n\prod_{k=d_i+1}^{j-1} y_k+\left(\prod_{k=d_i+1}^n y_k\right)\left(\sum_{j=0}^{d_i-1}\prod_{k=0}^j y_k\right)}$$
\item If there are no edges between $i^+$ and the middle circle, then $y_{i^+}'=y_{i^+}$
\item If $i$ is maximal so that there are edges between $i^+$ and the middle circle, then $$y_{i^+}'=y_{i^+}\frac{\left(\prod_{k=g_i+1}^n y_k\right)\left(\sum_{j=f_1}^{n-1}\prod_{k=0}^j y_k+\left(\prod_{k=0}^n y_k\right)\left(\sum_{j=0}^{f_1-1}\prod_{k=0}^j y_k\right)\right)}{\sum_{j=g_i+1}^n\prod_{k=g_i+1}^{j-1} y_k+\left(\prod_{k=g_i+1}^n y_k\right)\left(\sum_{j=0}^{g_i-1}\prod_{k=0}^j y_k\right)}$$
\item Otherwise, $$y_{i^+}'=y_{i^+}\frac{\sum_{j=g_i}^{n-1}\prod_{k=f_i+1}^j y_k+\left(\prod_{k=f_i+1}^n y_k\right)\left(\sum_{j=0}^{g_i-1}\prod_{k=0}^j y_k\right)}{\sum_{j=f_i+1}^n\prod_{k=f_i+1}^{j-1} y_k+\left(\prod_{k=f_i+1}^n y_k\right)\left(\sum_{j=0}^{f_i-1}\prod_{k=0}^j y_k\right)}$$
\end{enumerate}
\end{thm}

For proof, see Section~\ref{sec:proof-y-vars}.

\begin{thm}\label{thm:y-involution}
For a spider web quiver, $\tau$ is an involution on the $y$-variables.
\end{thm}
\begin{proof}
Theorem 4.1 of~\cite{N} tells us that periodicity of $x$- and $y$-variables are equivalent.  So, this follows from Theorem~\ref{thm:x-involution}.
\end{proof}

\begin{thm}\label{thm:y-braid}
Suppose we have a spider web quiver with any number of concentric circles.  Choose 2 adjacent circles.  If we apply $\tau$ to one circle, then the other, then the first again, we obtain the same $y$-variables as if we had applied $\tau$ to the second circle, then the first, then the second again.
\end{thm}
\begin{proof}
By Theorem 4.1 of~\cite{N}, this follows from Theorem~\ref{thm:x-braid}.
\end{proof}

\begin{ex}\label{ex:LP-y}
Consider the quiver on the left of Figure~\ref{fig:spider-web-quiver-ex}.  In this case, we have:

$\displaystyle y_1'=\frac{1+y_1+y_1y_2+y_1y_2y_3}{y_2+y_2y_3+y_2y_3y_4+y_1y_2y_3y_4}$

$\displaystyle y_2'=\frac{1+y_2+y_2y_3+y_2y_3y_4}{y_3+y_3y_4+y_1y_3y_4+y_1y_2y_3y_4}$

$\displaystyle y_3'=\frac{1+y_3+y_3y_4+y_1y_3y_4}{y_4+y_1y_4+y_1y_2y_4+y_1y_2y_3y_4}$

$\displaystyle y_4'=\frac{1+y_4+y_1y_4+y_1y_2y_4}{y_1+y_1y_2+y_1y_2y_3+y_1y_2y_3y_4}$

$\displaystyle y_{1^-}'=\frac{y_{1^-}(y_2+y_2y_3+y_2y_3y_4+y_1y_2y_3y_4)}{1+y_2+y_2y_3+y_2y_3y_4}$

$\displaystyle y_{2^-}'=\frac{y_{2^-}(y_3+y_3y_4+y_1y_3y_4+y_1y_2y_3y_4)}{1+y_3+y_3y_4+y_1y_3y_4}$

$\displaystyle y_{3^-}'=\frac{y_{3^-}(y_4+y_1y_4+y_1y_2y_4+y_1y_2y_3y_4)}{1+y_4+y_1y_4+y_1y_2y_4}$

$\displaystyle y_{4^-}'=\frac{y_{4^-}(y_1+y_1y_2+y_1y_2y_3+y_1y_2y_3y_4)}{1+y_1+y_1y_2+y_1y_2y_3}$

$\displaystyle y_{1^+}'=\frac{y_{1^+}(y_2+y_2y_3+y_2y_3y_4+y_1y_2y_3y_4)}{1+y_2+y_2y_3+y_2y_3y_4}$

$\displaystyle y_{2^+}'=\frac{y_{2^+}(y_3+y_3y_4+y_1y_3y_4+y_1y_2y_3y_4)}{1+y_3+y_3y_4+y_1y_3y_4}$

$\displaystyle y_{3^+}'=\frac{y_{3^+}(y_4+y_1y_4+y_1y_2y_4+y_1y_2y_3y_4)}{1+y_4+y_1y_4+y_1y_2y_4}$

$\displaystyle y_{4^+}'=\frac{y_{1^+}(y_1+y_1y_2+y_1y_2y_3+y_1y_2y_3y_4)}{1+y_1+y_1y_2+y_1y_2y_3}$
\end{ex}

\begin{ex}\label{ex:GS-y}
Consider the quiver on the right of Figure~\ref{fig:spider-web-quiver-ex}.  In this case, we have:

$\displaystyle y_1'=\frac{1+y_1+y_1y_2+y_1y_2y_3+y_1y_2y_3y_4+y_1y_2y_3y_4y_5}{y_2+y_2y_3+y_2y_3y_4+y_2y_3y_4y_5+y_2y_3y_4y_5y_6+y_1y_2y_3y_4y_5y_6}$

$\displaystyle y_2'=\frac{1+y_2+y_2y_3+y_2y_3y_4+y_2y_3y_4y_5+y_2y_3y_4y_5y_6}{y_3+y_3y_4+y_3y_4y_5+y_3y_4y_5y_6+y_1y_3y_4y_5y_6+y_1y_2y_3y_4y_5y_6}$

$\displaystyle y_3'=\frac{1+y_3+y_3y_4+y_3y_4y_5+y_3y_4y_5y_6+y_1y_3y_4y_5y_6}{y_4+y_4y_5+y_4y_5y_6+y_1y_4y_5y_6+y_1y_2y_4y_5y_6+y_1y_2y_3y_4y_5y_6}$

$\displaystyle y_4'=\frac{1+y_4+y_4y_5+y_4y_5y_6+y_1y_4y_5y_6+y_1y_2y_4y_5y_6}{y_5+y_5y_6+y_1y_5y_6+y_1y_2y_5y_6+y_1y_2y_3y_5y_6+y_1y_2y_3y_4y_5y_6}$

$\displaystyle y_5'=\frac{1+y_5+y_5y_6+y_1y_5y_6+y_1y_2y_5y_6+y_1y_2y_3y_5y_6}{y_6+y_1y_6+y_1y_2y_6+y_1y_2y_3y_6+y_1y_2y_3y_4y_6+y_1y_2y_3y_4y_5y_6}$

$\displaystyle y_6'=\frac{1+y_6+y_1y_6+y_1y_2y_6+y_1y_2y_3y_6+y_1y_2y_3y_4y_6}{y_1+y_1y_2+y_1y_2y_3+y_1y_2y_3y_4+y_1y_2y_3y_4y_5+y_1y_2y_3y_4y_5y_6}$

$\displaystyle y_{1^-}'=\frac{y_{1^-}(y_2+y_2y_3+y_2y_3y_4+y_2y_3y_4y_5+y_2y_3y_4y_5y_6+y_1y_2y_3y_4y_5y_6)}{1+y_2+y_2y_3+y_2y_3y_4+y_2y_3y_4y_5+y_2y_3y_4y_5y_6}$

$\displaystyle y_{2^-}'=\frac{y_{2^-}(y_3+y_3y_4+y_3y_4y_5+y_3y_4y_5y_6+y_1y_3y_4y_5y_6+y_1y_2y_3y_4y_5y_6)}{1+y_3+y_3y_4+y_3y_4y_5+y_3y_4y_5y_6+y_1y_3y_4y_5y_6}$

$\displaystyle y_{3^-}'=\frac{y_{3^-}(y_4+y_4y_5+y_4y_5y_6+y_1y_4y_5y_6+y_1y_2y_4y_5y_6+y_1y_2y_3y_4y_5y_6)}{1+y_4+y_4y_5+y_4y_5y_6+y_1y_4y_5y_6+y_1y_2y_4y_5y_6}$

$\displaystyle y_{4^-}'=\frac{y_{4^-}(y_5+y_5y_6+y_1y_5y_6+y_1y_2y_5y_6+y_1y_2y_3y_5y_6+y_1y_2y_3y_4y_5y_6)}{1+y_5+y_5y_6+y_1y_5y_6+y_1y_2y_5y_6+y_1y_2y_3y_5y_6}$

$\displaystyle y_{5^-}'=\frac{y_{5^-}(y_6+y_1y_6+y_1y_2y_6+y_1y_2y_3y_6+y_1y_2y_3y_4y_6+y_1y_2y_3y_4y_5y_6)}{1+y_6+y_1y_6+y_1y_2y_6+y_1y_2y_3y_6+y_1y_2y_3y_4y_6}$

$\displaystyle y_{6^-}'=\frac{y_{6^-}(y_1+y_1y_2+y_1y_2y_3+y_1y_2y_3y_4+y_1y_2y_3y_4y_5+y_1y_2y_3y_4y_5y_6)}{1+y_1+y_1y_2+y_1y_2y_3+y_1y_2y_3y_4+y_1y_2y_3y_4y_5}$

$\displaystyle y_{1^+}'=\frac{y_{1^+}(y_2y_3+y_2y_3y_4+y_2y_3y_4y_5+y_2y_3y_4y_5y_6+y_1y_2y_3y_4y_5y_6+y_1y_2^2y_3y_4y_5y_6)}{1+y_2+y_2y_3+y_2y_3y_4+y_2y_3y_4y_5+y_2y_3y_4y_5y_6}$

$\displaystyle y_{2^+}'=\frac{y_{2^+}(y_4y_5+y_4y_5y_6+y_1y_4y_5y_6+y_1y_2y_4y_5y_6+y_1y_2y_3y_4y_5y_6+y_1y_2y_3y_4^2y_5y_6)}{1+y_4+y_4y_5+y_4y_5y_6+y_1y_4y_5y_6+y_1y_2y_4y_5y_6}$

$\displaystyle y_{3^+}'=\frac{y_{3^+}(y_1y_6+y_1y_2y_6+y_1y_2y_3y_6+y_1y_2y_3y_4y_6+y_1y_2y_3y_4y_5y_6+y_1y_2y_3y_4y_5y_6^2)}{1+y_6+y_1y_6+y_1y_2y_6+y_1y_2y_3y_6+y_1y_2y_3y_4y_6}$
\end{ex}

\section{Coincidence Between R-matrix and Mutation Sequence}\label{sec:R-mat-mutation}

Recall in Section~\ref{sec:R-mat} that we labeled the faces in the left column of a cylindric 2-loop plabic network $a_1,...,a_\ell$ starting above the trail and going up.  We labeled the faces on the right $b_1,...,b_m$ and the faces in the center $c_1,...,c_{n-1},c_n=\frac{1}{a_1...a_\ell b_1...b_mc_1...c_{n-1}}$.

\begin{prop}
Let $Q$ be the dual quiver to a cylindric 2-loop plabic network.  Label the face and trail weights of the plabic network as in Section~\ref{sec:R-mat}.  Let $c_0=1$ for ease of notation.  Setting each $y$-variable equal to the corresponding face weight and applying $\tau$ yields the following $y$-variables:
\begin{enumerate}[(1)]
\item For $i\leq n$, $$y_i=c_i\frac{\left(\prod_{k=1}^{i-1} c_k\right)\left(\prod_{k=1}^\ell a_k\right)\left(\prod_{k=1}^m b_k\right)\left(\sum_{j=i}^n \prod_{k=i}^{j-1} c_k\right)+\sum_{j=0}^{i-2}\prod_{k=0}^j c_k}{\left(\prod_{k=1}^{i+1} c_k\right)\left(\prod_{k=1}^\ell a_k\right)\left(\prod_{k=1}^m b_k\right)\left(\sum_{j=i+2}^n\prod_{k=i+2}^{j-1} c_k\right)+\sum_{j=0}^i \prod_{k=0}^j c_k}$$
\item $\displaystyle y_n=\frac{\left(\prod_{k=1}^{n-1} c_k\right)\left(\prod_{k=1}^\ell a_k\right)\left(\prod_{k=1}^m b_k\right)+\sum_{j=0}^{n-2}\prod_{k=0}^j c_k}{\left(\prod_{k=1}^{n-1} c_k\right)\left(\prod_{k=1}^\ell a_k\right)\left(\prod_{k=1}^m b_k\right)\left(\sum_{j=1}^{n-1}\prod_{k=1}^j c_k\right)+\prod_{k=1}^k c_k}$
\item If $i$ is maximal, $$y_{i^-}=\frac{\sum_{j=0}^{d_1-1}\prod_{k=0}^j c_k+\left(\prod_{k=1}^\ell a_k\right)\left(\prod_{k=1}^m b_k\right)\left(\sum_{j=d_1}^{n-1} \prod_{k=0}^j c_k\right)}{\left(\prod_{k=1}^\ell a_k\right)\left(\prod_{k=1}^m b_k\right)\left(\sum_{j=0}^{n-1} \prod_{k=0}^j c_k\right)}$$
\item If $i$ is second-largest, $$y_{i^-}=\frac{\sum_{j=0}^{n-1} \prod_{k=0}^j c_k}{\left(\prod_{k=1}^\ell a_k\right)\left(\prod_{k=1}^m b_k\right)\left(\sum_{j=d_i+1}^n \prod_{k=1}^{j-1} c_k\right)+\sum_{j=0}^{d_i-1}\prod_{k=1}^j c_k}$$
\item For other $i$, $$y_{i^-}=\frac{\left(\prod_{k=1}^{n-1} c_k\right)\left(\prod_{k=1}^\ell a_k\right)\left(\prod_{k=1}^m b_k\right)\left(\sum_{j=e_i+1}^n\prod_{k=d_i+1}^{j-1}c_k\right)+\left(\prod_{k=d_i+1}^{n-1} c_k\right)\left(\sum_{j=0}^{e_i-1} \prod_{k=0}^j c_k\right)}{\left(\prod_{k=1}^{n-1} c_k\right)\left(\prod_{k=1}^\ell a_k\right)\left(\prod_{k=1}^m b_k\right)\left(\sum_{j=d_i+1}^n \prod_{k=d_i+1}^{j-1}c_k\right)+\left(\prod_{k=d_i+1}^{n-1} c_k\right)\left(\sum_{j=0}^n\prod_{k=d_i-1}^j c_k\right)}$$
\item If $i$ is maximal and $n\in B$, $$y_{i^+}=\frac{1+\left(\prod_{k=1}^\ell a_k\right)\left(\prod_{k=1}^m b_k\right)\left(\sum_{j=1}^{n-1} \prod_{k=0}^j c_k\right)}{\left(\prod_{k=1}^\ell a_k\right)\left(\prod_{k=1}^m b_k\right)\left(\sum_{j=0}^{n-1} \prod_{k=0}^j c_k\right)}$$
\item If $i$ is second largest and $n\in B$, $$y_{i^+}=\frac{\sum_{j=0}^{n-1}\prod_{k=0}^s c_k}{\left(\prod_{k=1}^\ell a_k\right)\left(\prod_{k=1}^m b_k\right)\left(\sum_{j=f_i+1}^n\prod_{k=1}^{j-1} c_k\right)+\sum_{j=0}^{f_i-1}\prod_{k=1}^s c_k}$$
\item If $i$ is maximal and $n\not\in B$, $$y_{i^+}=\frac{1+\left(\prod_{k=1}^\ell a_k\right)\left(\prod_{k=1}^m b_k\right)\left(\sum_{j=1}^{n-1} \prod_{k=0}^j c_k\right)}{\left(\prod_{k=1}^\ell a_k\right)\left(\prod_{k=1}^m b_k\right)\left(\left(\prod_{k=1}^\ell a_k\right)\left(\prod_{k=1}^m b_k\right)\left(\sum_{j=g_i}^{n-1} \prod_{k=0}^j c_k\right)+\sum_{j=0}^{g_i-1}\prod_{k=0}^j c_k\right)}$$
\item For other $i$, $$y_{i^+}=\frac{\left(\prod_{k=1}^{n-1} c_k\right)\left(\prod_{k=1}^\ell a_k\right)\left(\prod_{k=1}^m b_k\right)\left(\sum_{j=g_i+1}^n\prod_{k=f_i+1}^{j-1}c_k\right)+\left(\prod_{k=f_i+1}^{n-1} c_k\right)\left(\sum_{j=0}^{g_i-1} \prod_{k=0}^j c_k\right)}{\left(\prod_{k=1}^{n-1} c_k\right)\left(\prod_{k=1}^\ell a_k\right)\left(\prod_{k=1}^m b_k\right)\left(\sum_{j=f_i+1}^n\prod_{k=f_i+1}^{j-1}c_k\right)+\left(\prod_{j=f_i+1}^{n-1} c_k\right)\left(\sum_{j=0}^n\prod_{k=f_i-1}^j c_k\right)}$$
\end{enumerate}
\end{prop}
\begin{proof}
Notice that because of the way the faces were numbered, $n\in A$ and $f_1=1$ for any such quiver.  Then we can find the formulas for the $y$-variables by computation and Theorem~\ref{thm:y-vars}.
\end{proof}

\begin{thm}\label{thm:coincidence}
Let $Q$ be the dual quiver to a cylindric 2-loop plabic network.  Label the face and trail weights of the plabic network as in Section~\ref{sec:R-mat}.  If we set each $y$-variable equal to the corresponding face weight and apply $\tau$, the $y$-variables we obtain are the same as the face variables with the transformation $T_f$ applied to them.
\end{thm}
\begin{proof}
Let's investigate $\widehat{\lambda}_i(a,b,c)$.  There are $n$ terms in this sum; each one crosses from the left string to the right string at a different edge.  If we first calculate the term where we go across as soon as possible, then our first term is $$\left(\prod_{k=1}^i c_k\right)\left(\prod_{k=1}^m b_k\right)\left(\prod_{b_k\text{ assoc. to }j<i} b_k\right)\left(\prod_{a_k\text{ assoc. to }j\leq i} a_k\right).$$  If we compute the rest of our terms, each time crossing one slanted edge later, then each time we pick up one additional face variable.  First we pick up $c_{i+1}$, then $c_{i+2}$, all the way throughout $c_n$, and then we cycle back to the beginning and pick up $c_1$, then $c_2$, up to $c_{i-1}$.  Since multiplying by $c_n$ is the same as dividing by $a_1...a_\ell b_1...b_mc_1...c_{n-1}$, we get the following expression for $\widehat{\lambda}_i(a,b,c)$: $$\left(\prod_{k=1}^i c_k\right)\left(\prod_{k=1}^m b_k\right)\left(\prod_{b_k\text{ assoc. to }j<i} b_k\right)\left(\prod_{a_k\text{ assoc. to }j\leq i} a_k\right)\left(\sum_{j=i+1}^{n}\prod_{k=i+1}^{j-1} c_k\right)+$$$$\frac{\left(\prod_{b_k\text{ assoc. to }j<i} b_k\right)\left(\sum_{j=0}^{i-1}\prod_{k=0}^j c_k\right)}{\left(\prod_{b_k\text{ assoc. to }j> i} a_k\right)}$$  Substituting this into our expressions for the face variables under $T_f$ and using the previous proposition proves the theorem.
\end{proof}

\begin{ex}\label{ex:LP-coincidence}
Notice that the quiver in Example~\ref{ex:LP-y} is dual to the plabic graph in Example~\ref{ex:LP-face}.  We can see that the formulas we have calculated are the same if we make the following substitutions:

$y_1=c_1$, $y_2=c_2$, $y_3=c_3$, $y_4=\frac{1}{a_1a_2a_3a_4b_1b_2b_3b_4c_1c_2c_3}$

$y_{1^-}=a_2$, $y_{2^-}=a_3$, $y_{3^-}=a_4$, $y_{4^-}=a_1$

$y_{1^+}=b_1$, $y_{2^+}=b_2$, $y_{3^+}=b_3$, $y_{4^+}=b_4$
\end{ex}

\begin{ex}\label{ex:GS-coincidence}
Notice that the quiver in Example~\ref{ex:GS-y} is dual to the plabic graph in Example~\ref{ex:GS-face}.  We can see that the formulas we have calculated are the same if we make the following substitutions:

$y_1=c_1$, $y_2=c_2$, $y_3=c_3$, $y_4=c_4$, $y_5=c_5$, $y_6=\frac{1}{a_1a_2a_3a_4a_5a_6b_1b_2b_3c_1c_2c_3c_4c_5}$

$y_{1^-}=a_2$, $y_{2^-}=a_3$, $y_{3^-}=a_4$, $y_{4^-}=a_5$, $y_{5^-}=a_6$, $y_{6^-}=a_1$

$y_{1^+}=b_1$, $y_{2^+}=b_2$, $y_{3^+}=b_3$
\end{ex}

\section{Postnikov Diagram Proofs}\label{sec:AltStr}

\subsection{Proof of Theorem~\ref{thm:alt-str-reduced-gen}}\label{sec:AltStr-reduced}

\begin{defn}[Section 13 of~\cite{P}]\label{defn:trip}
For a plabic graph $G$, a \emph{trip} is a walk in $G$ that turns right at each black vertex and left at each white vertex.
\end{defn}

\begin{defn}[Section 13 of~\cite{P}]\label{defn:round-trip}
A trip in a plabic graph is a \emph{round trip} if it is a closed cycle.
\end{defn}

\begin{defn}[Section 13 of~\cite{P}]\label{defn:essential-int}
Two trips in a plabic graph have an \emph{essential intersection} if there is an edge $e$ with vertices of different colors such that the two trips pass through $e$ in different directions.  A trip in a plabic graph has an \emph{essential self-intersection} if there is an edge $e$ with vertices of different colors such that the trip passes through $e$ in different directions.
\end{defn}

\begin{defn}[Section 13 of~\cite{P}]\label{defn:bad-double-crossing}
Two trips in a plabic graph have an \emph{bad double crossing} if they have essential intersections at edges $e_1$ and $e_2$ where both trips are directed from $e_1$ to $e_2$.
\end{defn}

\begin{thm}\label{thm:reduced-trips}
Let $G$ be a leafless reduced plabic graph on a cylinder without isolated components.  We will consider $\widetilde{G}$ to be $G$ drawn on the universal cover of the cylinder.  Then $G$ is reduced if and only if the following are true:
\begin{enumerate}[(1)]
\item There are no round trips in $\widetilde{G}$.

\item $\widetilde{G}$ has no trips with essential self-intersections.

\item There are no pairs of trips in $\widetilde{G}$ with a bad double crossing.

\item If a trip begins and ends at the same boundary vertex, then either $\widetilde{G}$ has a boundary leaf at that vertex.
\end{enumerate}
\end{thm}

The above theorem is analogous to Theorem 13.2 of~\cite{P}.
\begin{proof}
Notice that $G$ is reduced if and only if $\widetilde{G}$ is reduced.  The proof for Theorem 13.2 of~\cite{P} holds to show that $\widetilde{G}$ is reduced if and only if conditions (1) - (4) hold.
\end{proof}

Now we can prove Theorem~\ref{thm:alt-str-reduced-gen}.
\begin{proof}
Notice that the trips in a plabic graph follow the same paths as the strands in the associated Postnikov diagram.  The conditions from Theorem~\ref{thm:reduced-trips} correspond exactly to the conditions we require in the definition of a Postnikov diagram.
\end{proof}

\subsection{Proof of Theorem~\ref{thm:int-vert-k-loop}}\label{sec:IntVert}

\begin{lemma}\label{lemma:sq-move}
Suppose a cylindric 2-loop plabic graph has an interior vertex that has one edge to a vertex on a string and one edge to a different vertex on the same string, such that there is only one vertex on the string between these two vertices, and the square formed by these four vertices is the boundary of a single face.  Then, we can reduce the number of strand crossings in between the two loops in the associated Postnikov diagram using the square move.
\end{lemma}
\begin{proof}
Without loss of generality, assume the interior vertex is black.  Then we can apply transformations to our plabic graph as follows:

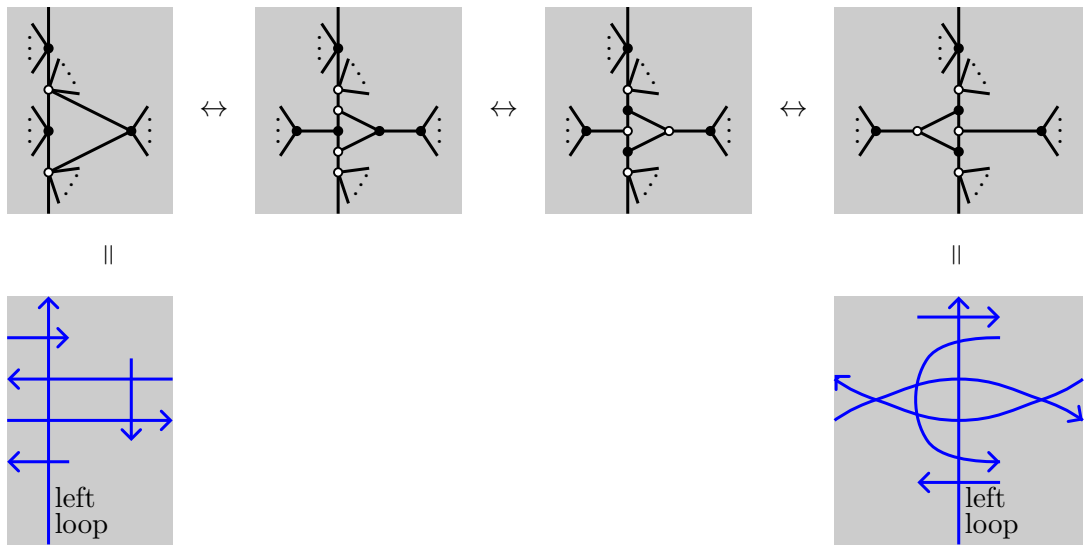
\begin{center}
\begin{tikzpicture}[scale=0.55]
\fill[gray!40!white] (0,-1) rectangle (4,5);
\path[->,font=\large, >=angle 90, line width=0.4mm,blue]
(1,-1) edge (1,5)
(1.5,1) edge (0,1)
(0,2) edge (4,2)
(4,3) edge (0,3)
(0,4) edge (1.5,4)
(3,3.5) edge (3,1.5);
\node at (1.65,0.16) {left};
\node at (1.8,-0.55) {loop};
\node at (2.5,6) {\rotatebox{90}{$=$}};
\fill[gray!40!white] (0,7) rectangle (4,12);
\path[-,font=\large, >=angle 90, line width=0.4mm,]
(1,7) edge (1,12)
(1,8) edge (3,9)
(1,10) edge (3,9)
(1,10) edge (1.75,9.9)
(1,10) edge (1.25,10.75)
(1,8) edge (1.75,8.1)
(1,8) edge (1.25,7.25)
(3,9) edge (3.4,9.6)
(3,9) edge (3.4,8.4)
(1,9) edge (0.6,9.6)
(1,9) edge (0.6,8.4)
(1,11) edge (0.6,11.6)
(1,11) edge (0.6,10.4);
\node at (1.4,10.55) {\rotatebox{30}{$\vdots$}};
\node at (1.65,7.87) {\rotatebox{-30}{$\vdots$}};
\node at (3.45,9.2) {$\vdots$};
\node at (0.55,9.2) {$\vdots$};
\node at (0.55,11.2) {$\vdots$};
\draw [line width=0.25mm, fill=white] (1,8) circle (1mm);
\draw [line width=0.25mm, fill=black] (1,9) circle (1mm);
\draw [line width=0.25mm, fill=white] (1,10) circle (1mm);
\draw [line width=0.25mm, fill=black] (1,11) circle (1mm);
\draw [line width=0.25mm, fill=black] (3,9) circle (1mm);
\node at (5,9.5) {$\leftrightarrow$};
\fill[gray!40!white] (6,7) rectangle (11,12);
\path[-,font=\large, >=angle 90, line width=0.4mm,]
(8,7) edge (8,12)
(8,8.5) edge (9,9)
(8,9.5) edge (9,9)
(10,9) edge (9,9)
(8,10) edge (8.75,9.9)
(8,10) edge (8.25,10.75)
(8,8) edge (8.75,8.1)
(8,8) edge (8.25,7.25)
(10,9) edge (10.4,9.6)
(10,9) edge (10.4,8.4)
(7,9) edge (6.6,9.6)
(7,9) edge (6.6,8.4)
(8,11) edge (7.6,11.6)
(8,11) edge (7.6,10.4)
(7,9) edge (8,9);
\node at (8.4,10.55) {\rotatebox{30}{$\vdots$}};
\node at (8.65,7.87) {\rotatebox{-30}{$\vdots$}};
\node at (10.45,9.2) {$\vdots$};
\node at (6.55,9.2) {$\vdots$};
\node at (7.55,11.2) {$\vdots$};
\draw [line width=0.25mm, fill=white] (8,8) circle (1mm);
\draw [line width=0.25mm, fill=white] (8,8.5) circle (1mm);
\draw [line width=0.25mm, fill=black] (8,9) circle (1mm);
\draw [line width=0.25mm, fill=white] (8,9.5) circle (1mm);
\draw [line width=0.25mm, fill=white] (8,10) circle (1mm);
\draw [line width=0.25mm, fill=black] (8,11) circle (1mm);
\draw [line width=0.25mm, fill=black] (9,9) circle (1mm);
\draw [line width=0.25mm, fill=black] (10,9) circle (1mm);
\draw [line width=0.25mm, fill=black] (7,9) circle (1mm);
\node at (12,9.5) {$\leftrightarrow$};
\fill[gray!40!white] (13,7) rectangle (18,12);
\path[-,font=\large, >=angle 90, line width=0.4mm,]
(15,7) edge (15,12)
(15,8.5) edge (16,9)
(15,9.5) edge (16,9)
(17,9) edge (16,9)
(15,10) edge (15.75,9.9)
(15,10) edge (15.25,10.75)
(15,8) edge (15.75,8.1)
(15,8) edge (15.25,7.25)
(17,9) edge (17.4,9.6)
(17,9) edge (17.4,8.4)
(14,9) edge (13.6,9.6)
(14,9) edge (13.6,8.4)
(15,11) edge (14.6,11.6)
(15,11) edge (14.6,10.4)
(14,9) edge (15,9);
\node at (15.4,10.55) {\rotatebox{30}{$\vdots$}};
\node at (15.65,7.87) {\rotatebox{-30}{$\vdots$}};
\node at (17.45,9.2) {$\vdots$};
\node at (13.55,9.2) {$\vdots$};
\node at (14.55,11.2) {$\vdots$};
\draw [line width=0.25mm, fill=white] (15,8) circle (1mm);
\draw [line width=0.25mm, fill=black] (15,8.5) circle (1mm);
\draw [line width=0.25mm, fill=white] (15,9) circle (1mm);
\draw [line width=0.25mm, fill=black] (15,9.5) circle (1mm);
\draw [line width=0.25mm, fill=white] (15,10) circle (1mm);
\draw [line width=0.25mm, fill=black] (15,11) circle (1mm);
\draw [line width=0.25mm, fill=white] (16,9) circle (1mm);
\draw [line width=0.25mm, fill=black] (17,9) circle (1mm);
\draw [line width=0.25mm, fill=black] (14,9) circle (1mm);
\node at (19,9.5) {$\leftrightarrow$};
\fill[gray!40!white] (20,7) rectangle (26,12);
\path[-,font=\large, >=angle 90, line width=0.4mm,]
(23,7) edge (23,12)
(25,9) edge (23,9)
(23,10) edge (23.75,9.9)
(23,10) edge (23.25,10.75)
(23,8) edge (23.75,8.1)
(23,8) edge (23.25,7.25)
(25,9) edge (25.4,9.6)
(25,9) edge (25.4,8.4)
(21,9) edge (20.6,9.6)
(21,9) edge (20.6,8.4)
(23,11) edge (22.6,11.6)
(23,11) edge (22.6,10.4)
(22,9) edge (21,9)
(22,9) edge (23,9.5)
(22,9) edge (23,8.5);
\node at (23.4,10.55) {\rotatebox{30}{$\vdots$}};
\node at (23.65,7.87) {\rotatebox{-30}{$\vdots$}};
\node at (25.45,9.2) {$\vdots$};
\node at (20.55,9.2) {$\vdots$};
\node at (22.55,11.2) {$\vdots$};
\draw [line width=0.25mm, fill=white] (23,8) circle (1mm);
\draw [line width=0.25mm, fill=black] (23,8.5) circle (1mm);
\draw [line width=0.25mm, fill=white] (23,9) circle (1mm);
\draw [line width=0.25mm, fill=black] (23,9.5) circle (1mm);
\draw [line width=0.25mm, fill=white] (23,10) circle (1mm);
\draw [line width=0.25mm, fill=black] (23,11) circle (1mm);
\draw [line width=0.25mm, fill=black] (25,9) circle (1mm);
\draw [line width=0.25mm, fill=white] (22,9) circle (1mm);
\draw [line width=0.25mm, fill=black] (21,9) circle (1mm);
\node at (23,6) {\rotatebox{90}{$=$}};
\fill[gray!40!white] (20,-1) rectangle (26,5);
\path[->,font=\large, >=angle 90, line width=0.4mm,blue]
(23,-1) edge (23,5)
(24,0.5) edge (22,0.5)
(22,4.5) edge (24,4.5);
\draw [line width=0.4mm,blue] plot [smooth, tension=0.8] coordinates { (24,4) (22.25,3.5) (22.25,1.5) (24,1) };
\path[->,font=\large, >=angle 90, line width=0.4mm, blue]
(23.9,1) edge (24,1);
\draw [line width=0.4mm,blue] plot [smooth, tension=0.8] coordinates { (26,3) (25,2.5) (23,2) (21,2.5) (20,3)};
\path[->,font=\large, >=angle 90, line width=0.4mm, blue]
(20.1,3) edge (20,3.1);
\draw [line width=0.4mm,blue] plot [smooth, tension=0.8] coordinates { (26,2) (25,2.5) (23,3) (21,2.5) (20,2)};
\path[->,font=\large, >=angle 90, line width=0.4mm, blue]
(25.9,2.07) edge (26,2);
\node at (23.65,0.16) {left};
\node at (23.8,-0.55) {loop};
\end{tikzpicture}
\end{center}

The Postnikov diagram on the left has two crossings between the left and the right loop, in addition to those we can't see in the picture.  The Postnikov diagram on the right has one crossing between the left and the right loop, aside from those we can't see in the picture.  So, we have reduced the number of crossings.
\end{proof}

\begin{lemma}\label{lemma:mult-edges}
Suppose a cylindric 2-loop plabic graph has at least one interior vertex.  Assume no vertices in the plabic graph have degree two.  Then at least one interior vertex must have multiple edges to vertices on a string.
\end{lemma}

\begin{proof}
Note that only white vertices on the left string and black vertices on the right string can have edges to interior vertices.  As the graph is bipartite, this means an interior vertex cannot have edges to vertices on both the left and right string.

Since there are interior vertices, there must be a vertex on a string attached to an interior vertex.  Without loss of generality, assume there's such a vertex on the left string.  Expand this vertex so that we have a vertex on the left string that is attached only to one interior vertex and two vertices on the left string.  We allow some vertices of degree 2 to be created to keep the graph bipartite.  Now we have an interior vertex must arise from a part of the alternating strand diagram that looks as follows:
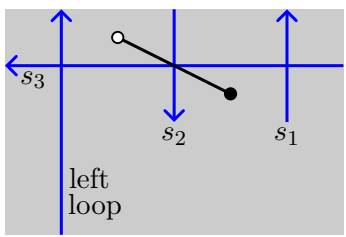
\begin{center}
\begin{tikzpicture}[scale=0.75]
\fill[gray!40!white] (0,-1) rectangle (6,3);
\path[->,font=\large, >=angle 90, line width=0.4mm,blue]
(1,-1) edge (1,3)
(6,2) edge (0,2)
(3,3) edge (3,1)
(5,1) edge (5,3);
\path[-,font=\large, >=angle 90, line width=0.4mm]
(4,1.5) edge (2,2.5);
\draw [line width=0.25mm, fill=white] (2,2.5) circle (1mm);
\draw [line width=0.25mm, fill=black] (4,1.5) circle (1mm);
\node at (1.5,0) {left};
\node at (1.6,-0.5) {loop};
\node at (5,0.75) {$s_1$};
\node at (3,0.75) {$s_2$};
\node at (0.5,1.75) {$s_3$};
\end{tikzpicture}
\end{center}

Call the strands $s_1, s_2$, and $s_3$, as denoted in the diagram.  The black vertex in the diagram is not on the right string, so $s_1$ cannot be the right loop.  Suppose $s_1$ starts and ends on the right boundary of the cylinder.  If $s_1$ does not make a turn and head downward, then $s_1$ and the right loop would intersect in two places on the universal cover, and both would be oriented in the same direction (from one crossing to the other).  So, $s_1$ must turn downwards at some point.  If $s_1$ turns to the right to do this, it will have a self-crossing, which is not allowed.  If $s_1$ turns to the left to do this, then $s_2=s_1$.  Then $s_3$ and $s_1$ intersect twice, and both are oriented from the crossing on the right to the crossing on the left.  This is not allowed, so $s_1$ must cross the left loop.

Suppose $s_1$ crosses the left loop from left to right (the argument is very similar for $s_1$ crossing from right to left).  There must be a strand oriented from right to left that crosses the left loop just above $s_1$.  Call this strand $s_4$.  Either $s_4$ must cross $s_1$ or $s_4=s_2$.  Suppose it's the former.  The face above $s_4$ and to the right of the left loop corresponds to a white vertex on the left string, which must be connected to an interior vertex.  Suppose there is another strand that crosses $s_4$ between the left loop and $s_1$ in the same direction as $s_1$.  Since this strand did not cross $s_3$ in the same direction as $s_1$ between the left loop and $s_1$, it must cross $s_1$ somewhere above $s_4$.  However, this strand would also have to cross $s_1$ below $s_4$, as nothing crosses the left loop between $s_1$ and $s_4$.  This would introduce two crossings of $s_1$ and the additional upward pointing strand, and both of the strands would be oriented in the same direction.  This cannot happen, so we have this section of the alternating strand diagram that looks as follows:
\begin{center}
\begin{tikzpicture}[scale=0.75]
\fill[gray!40!white] (-1,-1) rectangle (6,3);
\path[->,font=\large, >=angle 90, line width=0.4mm,blue]
(1,-1) edge (1,3)
(6,2) edge (-1,2)
(3,3) edge (3,-1);
\draw [line width=0.4mm,blue] plot [smooth, tension=0.8] coordinates { (-1,0.7) (3.75,1) (5,3) };
\path[->,font=\large, >=angle 90, line width=0.4mm, blue]
(5,2.9) edge (5,3);
\path[-,font=\large, >=angle 90, line width=0.4mm]
(4,1.5) edge (2,2.5)
(4,1.5) edge (2,0.35);
\draw [line width=0.25mm, fill=white] (2,0.35) circle (1mm);
\draw [line width=0.25mm, fill=white] (2,2.5) circle (1mm);
\draw [line width=0.25mm, fill=black] (4,1.5) circle (1mm);
\node at (1.5,0) {left};
\node at (1.6,-0.5) {loop};
\node at (-0.25,0.4) {$s_1$};
\node at (-0.25,2.25) {$s_4$};
\end{tikzpicture}
\end{center}

Now suppose we are in the second case: $s_4=s_2$.  We get a sequence of connected vertices beginning with the white vertex in the face bounded by $s_2,s_3$, and the left loop and ending with the white vertex in the face bounded by $s_2$ and the left loop such that vertices alternate being in faces to the left and right of $s_2$.  This sequence and the section of the string bounded below by $s_2$ and above by $s_3$ form a closed cycle, where there may or may not be vertices shared by sequence zig-zagging across $s_2$ and the left string:
\begin{center}
\begin{tikzpicture}[scale=0.75]
\fill[gray!40!white] (-1,-1) rectangle (5,10);
\path[->,font=\large, >=angle 90, line width=0.4mm,blue]
(1,-1) edge (1,10)
(5,9) edge (-1,9)
(2,1) edge (-1,1);
\draw [line width=0.4mm,blue] plot [smooth, tension=0.8] coordinates { (4,5) (4,10) };
\draw [line width=0.4mm,blue] plot [smooth, tension=0.8] coordinates { (2,1) (3.6,2) (4,5) };
\path[-,font=\large, >=angle 90, line width=0.4mm]
(2.5,9.5) edge (4.5,8)
(4.5,8) edge (3.5,6.75)
(3.5,6.75) edge (4.5,5.5)
(4.5,5.5) edge (2.5,4.25)
(2.5,4.25) edge (4.5,3)
(4,1.5) edge (2,1.5)
(2,1.5) edge (0.5,2)
(0.5,3.5) edge (2.5,4.25)
(2.5,4.25) edge (0.5,5.5)
(0.5,5.5) edge (1.5,6.125)
(1.5,6.125) edge (0.5,6.75)
(0.5,6.75) edge (1.5,7.375) 
(1.5,7.375)  edge (0.5,8)
(0.5,8) edge (2.5,9.5);
\draw [line width=0.25mm, fill=white] (2.5,9.5) circle (1mm);
\draw [line width=0.25mm, fill=black] (4.5,8) circle (1mm);
\draw [line width=0.25mm, fill=white] (3.5,6.75) circle (1mm);
\draw [line width=0.25mm, fill=black] (4.5,5.5) circle (1mm);
\draw [line width=0.25mm, fill=white] (2.5,4.25) circle (1mm);
\draw [line width=0.25mm, fill=black] (4.5,3) circle (1mm);
\draw [line width=0.25mm, fill=black] (4,1.5) circle (1mm);
\draw [line width=0.25mm, fill=white] (2,1.5) circle (1mm);
\draw [line width=0.25mm, fill=black] (0.5,2) circle (1mm);
\draw [line width=0.25mm, fill=black] (0.5,3.5) circle (1mm);
\draw [line width=0.25mm, fill=black] (0.5,5.5) circle (1mm);
\draw [line width=0.25mm, fill=white] (1.5,6.125) circle (1mm);
\draw [line width=0.25mm, fill=black] (0.5,6.75) circle (1mm);
\draw [line width=0.25mm, fill=white] (1.5,7.375) circle (1mm);
\draw [line width=0.25mm, fill=black] (0.5,8) circle (1mm);
\node at (1.5,0) {left};
\node at (1.6,-0.5) {loop};
\node at (-0.25,0.75) {$s_2$};
\node at (-0.25,9.25) {$s_3$};
\node at (4.3,2.3) {\rotatebox{-20}{$\vdots$}};
\node at (0.5,2.85) {$\vdots$};
\end{tikzpicture}
\end{center}

If there are no interior vertices inside this cycle, then we are done, as any black vertex in this cycle not on the left string, of which there is at least one, is adjacent to multiple vertices on the left string.  Otherwise, we can repeat our process from the beginning, this time choosing an interior vertex that is inside the cycle.  We will obtain an $s_1',s_2',s_3'$, and $s_4'$.  If $s_2'\neq s_4'$, then we get a contradiction, as above.  Otherwise the cycle of vertices we obtain is inside the cycle we obtained from $s_1,s_2$ and $s_3$.  Since the graph is finite, this process must eventually terminate, and we will have found an interior vertex with multiple non-parallel edges to vertices on the left string.
\end{proof}

\begin{lemma}\label{lemma:same-int-vert}
Suppose a cylindric 2-loop plabic graph with no vertices of degree 2 has an interior vertex that has (at least) two edges to vertices on a string.  If there are any other vertices of the same color on the string between these two vertices, they must also have edges to the same interior vertex.
\end{lemma}
\begin{proof}
Assume for contradiction we have a plabic graph with an interior vertex that has (at least) two edges to vertices on a string and the vertices of the same color on the string between these two vertices do not have edges to the same interior vertex. Without loss of generality assume the interior vertex is a black vertex.  We know this part of the Postnikov diagram looks as shown, where the interior of the darker box is unknown:
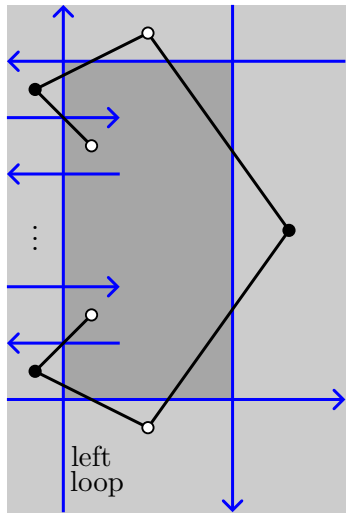
\begin{center}
\begin{tikzpicture}[scale=0.75]
\fill[gray!40!white] (0,-1) rectangle (6,8);
\fill[gray!70!white] (1,1) rectangle (4,7);
\path[->,font=\large, >=angle 90, line width=0.4mm,blue]
(1,-1) edge (1,8)
(4,8) edge (4,-1)
(0,1) edge (6,1)
(6,7) edge (0,7)
(2,2) edge (0,2)
(0,3) edge (2,3)
(2,5) edge (0,5)
(0,6) edge (2,6);
\path[-,font=\large, >=angle 90, line width=0.4mm]
(0.5,1.5) edge (1.5,2.5)
(0.5,6.5) edge (1.5,5.5)
(0.5,1.5) edge (2.5,0.5) 
(0.5,6.5) edge (2.5,7.5)
(5,4) edge (2.5,0.5) 
(5,4) edge (2.5,7.5);
\draw [line width=0.25mm, fill=white] (1.5,2.5) circle (1mm);
\draw [line width=0.25mm, fill=white] (1.5,5.5) circle (1mm);
\draw [line width=0.25mm, fill=white] (2.5,7.5) circle (1mm);
\draw [line width=0.25mm, fill=white] (2.5,0.5) circle (1mm);
\draw [line width=0.25mm, fill=black] (0.5,1.5) circle (1mm);
\draw [line width=0.25mm, fill=black] (0.5,6.5) circle (1mm);
\draw [line width=0.25mm, fill=black] (5,4) circle (1mm);
\node at (1.5,0) {left};
\node at (1.6,-0.5) {loop};
\node at (0.5,4) {$\vdots$};
\end{tikzpicture}
\end{center}

If any of the strands on the left exit the shaded face out the top or bottom and don't return, they'll disconnect the original edges from the interior vertex to the two vertices on the string.  If they exit out the strand on the right and don't return, they'll split the one interior vertex into multiple.  Thus, all the strands crossing the left loop from the left must also be the same strands that cross from the right.  All strands like this must be oriented the opposite direction of the loop.  The only way to pair up these strands in that way is for each strand going to the right is paired with the strand immediately below it.  There is no way to do this without self-crossings and without creating vertices of degree 2.  So, we have a contradiction.
\end{proof}

\begin{thm}\label{thm:no-int-vert}
Any cylindric 2-loop plabic graph can be transformed by moves to one that has no interior vertices.
\end{thm}
\begin{proof}
If there are interior vertices, then by Lemmas~\ref{lemma:mult-edges} and~\ref{lemma:same-int-vert}, we are always in a situation where, by Lemma~\ref{lemma:sq-move}, we can reduce the number of edges until we get to a plabic graph where there are no interior vertices.
\end{proof}

\begin{lemma}\label{lemma:black-left}
If there are interior vertices in a cylindric 2-loop plabic graph, then there must be a black interior vertex that is adjacent to the left string.
\end{lemma}
\begin{proof}
Assume for contradiction we have a cylindric 2-loop plabic graph with interior vertices, but no black interior vertex connected to the left string.  Then each white vertex on the left string is connected only to two black vertices on left string and some positive number of black vertices on the right string.  If we look at the area of the graph between two consecutive edges that connect the strands, we have three possibilities:
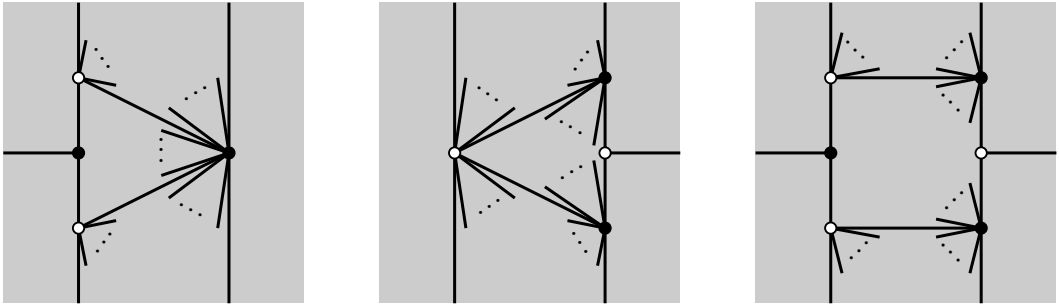
\begin{center}
\begin{tikzpicture}[scale=1]
\fill[gray!40!white] (0,0) rectangle (4,4);
\path[-,font=\large, >=angle 90, line width=0.4mm]
(1,0) edge (1,4)
(3,0) edge (3,4)
(0,2) edge (1,2)
(1,1) edge (3,2)
(1,3) edge (3,2)
(1,1) edge (1.5,1.1)
(1,1) edge (1.1,0.5)
(1,3) edge (1.5, 2.9)
(1,3) edge (1.1, 3.5)
(3,2) edge (2.1, 2.3)
(3,2) edge (2.1, 1.7)
(3,2) edge (2.2, 1.4)
(3,2) edge (2.85, 1)
(3,2) edge (2.2, 2.6)
(3,2) edge (2.85, 3);
\draw [line width=0.25mm, fill=white] (1,1) circle (0.75mm);
\draw [line width=0.25mm, fill=white] (1,3) circle (0.75mm);
\draw [line width=0.25mm, fill=black] (1,2) circle (0.75mm);
\draw [line width=0.25mm, fill=black] (3,2) circle (0.75mm);
\node at (1.4,0.9) {\rotatebox{-35}{$\vdots$}};
\node at (1.25,3.35) {\rotatebox{35}{$\vdots$}};
\node at (2.65,2.85) {\rotatebox{-60}{$\vdots$}};
\node at (2.4,1.3) {\rotatebox{60}{$\vdots$}};
\node at (2.1,2.15) {$\vdots$};
\fill[gray!40!white] (5,0) rectangle (9,4);
\path[-,font=\large, >=angle 90, line width=0.4mm]
(6,0) edge (6,4)
(8,0) edge (8,4)
(8,2) edge (9,2)
(8,1) edge (6,2)
(8,3) edge (6,2)
(8,1) edge (7.5,1.1)
(8,1) edge (7.9,0.5)
(8,1) edge (7.85, 1.9)
(8,1) edge (7.2, 1.55)
(8,3) edge (7.5, 2.9)
(8,3) edge (7.9, 3.5)
(8,3) edge (7.85, 2.1)
(8,3) edge (7.2, 2.45)
(6,2) edge (6.8, 1.4)
(6,2) edge (6.15, 1)
(6,2) edge (6.8, 2.6)
(6,2) edge (6.15, 3);
\draw [line width=0.25mm, fill=white] (6,2) circle (0.75mm);
\draw [line width=0.25mm, fill=white] (8,2) circle (0.75mm);
\draw [line width=0.25mm, fill=black] (8,1) circle (0.75mm);
\draw [line width=0.25mm, fill=black] (8,3) circle (0.75mm);
\node at (7.75,3.32) {\rotatebox{-35}{$\vdots$}};
\node at (7.65,1.8) {\rotatebox{-60}{$\vdots$}};
\node at (7.6,0.9) {\rotatebox{35}{$\vdots$}};
\node at (7.45,2.4) {\rotatebox{60}{$\vdots$}};
\node at (6.55,1.35) {\rotatebox{-55}{$\vdots$}};
\node at (6.35,2.85) {\rotatebox{55}{$\vdots$}};
\fill[gray!40!white] (10,0) rectangle (14,4);
\path[-,font=\large, >=angle 90, line width=0.4mm]
(11,0) edge (11,4)
(13,0) edge (13,4)
(10,2) edge (11,2)
(13,2) edge (14,2)
(11,1) edge (13,1)
(11,3) edge (13,3)
(11,1) edge (11.65,0.88)
(11,1) edge (11.15,0.4)
(11,3) edge (11.65,3.12)
(11,3) edge (11.15,3.6)
(13,1) edge (12.4,0.88)
(13,1) edge (12.85,0.4)
(13,1) edge (12.4,1.12)
(13,1) edge (12.85,1.6)
(13,3) edge (12.4,3.12)
(13,3) edge (12.85,3.6)
(13,3) edge (12.4,2.88)
(13,3) edge (12.85,2.4);
\draw [line width=0.25mm, fill=white] (11,1) circle (0.75mm);
\draw [line width=0.25mm, fill=white] (11,3) circle (0.75mm);
\draw [line width=0.25mm, fill=white] (13,2) circle (0.75mm);
\draw [line width=0.25mm, fill=black] (11,2) circle (0.75mm);
\draw [line width=0.25mm, fill=black] (13,1) circle (0.75mm);
\draw [line width=0.25mm, fill=black] (13,3) circle (0.75mm);
\node at (11.45,0.78) {\rotatebox{-45}{$\vdots$}};
\node at (11.25,3.45) {\rotatebox{45}{$\vdots$}};
\node at (12.52,0.75) {\rotatebox{45}{$\vdots$}};
\node at (12.72,1.45) {\rotatebox{-45}{$\vdots$}};
\node at (12.52,2.75) {\rotatebox{45}{$\vdots$}};
\node at (12.72,3.45) {\rotatebox{-45}{$\vdots$}};
\end{tikzpicture}
\end{center}

We can draw in parts of some strands based on what we know of the graph:
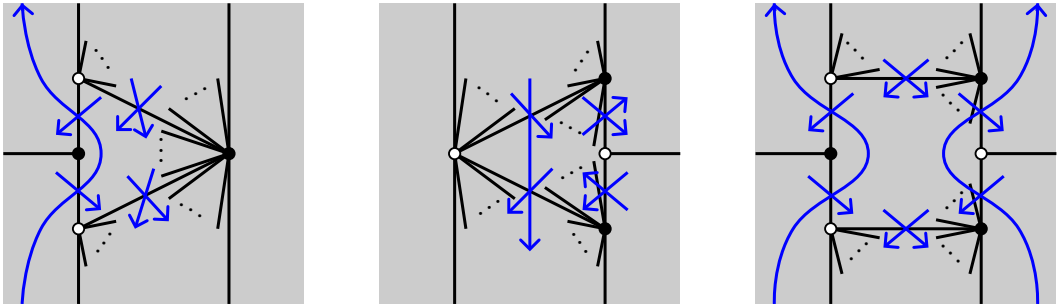
\begin{center}
\begin{tikzpicture}[scale=1]
\fill[gray!40!white] (0,0) rectangle (4,4);
\path[-,font=\large, >=angle 90, line width=0.4mm]
(1,0) edge (1,4)
(3,0) edge (3,4)
(0,2) edge (1,2)
(1,1) edge (3,2)
(1,3) edge (3,2)
(1,1) edge (1.5,1.1)
(1,1) edge (1.1,0.5)
(1,3) edge (1.5, 2.9)
(1,3) edge (1.1, 3.5)
(3,2) edge (2.1, 2.3)
(3,2) edge (2.1, 1.7)
(3,2) edge (2.2, 1.4)
(3,2) edge (2.85, 1)
(3,2) edge (2.2, 2.6)
(3,2) edge (2.85, 3);
\draw [line width=0.25mm, fill=white] (1,1) circle (0.75mm);
\draw [line width=0.25mm, fill=white] (1,3) circle (0.75mm);
\draw [line width=0.25mm, fill=black] (1,2) circle (0.75mm);
\draw [line width=0.25mm, fill=black] (3,2) circle (0.75mm);
\node at (1.4,0.9) {\rotatebox{-35}{$\vdots$}};
\node at (1.25,3.35) {\rotatebox{35}{$\vdots$}};
\node at (2.65,2.85) {\rotatebox{-60}{$\vdots$}};
\node at (2.4,1.3) {\rotatebox{60}{$\vdots$}};
\node at (2.1,2.15) {$\vdots$};
\path[->,font=\large, >=angle 90, line width=0.4mm,blue]
(1.65,1.7) edge (2.2, 1.1)
(2,1.8) edge (1.75, 1)
(1.7,3) edge (1.9,2.2)
(2.1,2.9) edge (1.5,2.3)
(1.3,2.75) edge (0.7, 2.25)
(0.7, 1.75) edge (1.3,1.25)
(0.26,3.9) edge (0.25,4);
\draw [line width=0.4mm,blue] plot [smooth, tension=0.8] coordinates { (0.25,0) (0.5,1) (1.3,2) (0.5,3) (0.25,4) };
\fill[gray!40!white] (5,0) rectangle (9,4);
\path[-,font=\large, >=angle 90, line width=0.4mm]
(6,0) edge (6,4)
(8,0) edge (8,4)
(8,2) edge (9,2)
(8,1) edge (6,2)
(8,3) edge (6,2)
(8,1) edge (7.5,1.1)
(8,1) edge (7.9,0.5)
(8,1) edge (7.85, 1.9)
(8,1) edge (7.2, 1.55)
(8,3) edge (7.5, 2.9)
(8,3) edge (7.9, 3.5)
(8,3) edge (7.85, 2.1)
(8,3) edge (7.2, 2.45)
(6,2) edge (6.8, 1.4)
(6,2) edge (6.15, 1)
(6,2) edge (6.8, 2.6)
(6,2) edge (6.15, 3);
\draw [line width=0.25mm, fill=white] (6,2) circle (0.75mm);
\draw [line width=0.25mm, fill=white] (8,2) circle (0.75mm);
\draw [line width=0.25mm, fill=black] (8,1) circle (0.75mm);
\draw [line width=0.25mm, fill=black] (8,3) circle (0.75mm);
\node at (7.75,3.32) {\rotatebox{-35}{$\vdots$}};
\node at (7.65,1.8) {\rotatebox{-60}{$\vdots$}};
\node at (7.6,0.9) {\rotatebox{35}{$\vdots$}};
\node at (7.45,2.4) {\rotatebox{60}{$\vdots$}};
\node at (6.55,1.35) {\rotatebox{-55}{$\vdots$}};
\node at (6.35,2.85) {\rotatebox{55}{$\vdots$}};
\path[->,font=\large, >=angle 90, line width=0.4mm,blue]
(6.75,2.8) edge (7.3, 2.2)
(7.3,1.8) edge (6.7,1.2)
(7.7, 2.75) edge (8.3,2.25)
(7.7, 2.25) edge (8.3,2.75)
(8.3,1.25) edge (7.7, 1.75)
(8.3,1.75) edge (7.7, 1.25)
(7,3) edge (7,0.7);
\fill[gray!40!white] (10,0) rectangle (14,4);
\path[-,font=\large, >=angle 90, line width=0.4mm]
(11,0) edge (11,4)
(13,0) edge (13,4)
(10,2) edge (11,2)
(13,2) edge (14,2)
(11,1) edge (13,1)
(11,3) edge (13,3)
(11,1) edge (11.65,0.88)
(11,1) edge (11.15,0.4)
(11,3) edge (11.65,3.12)
(11,3) edge (11.15,3.6)
(13,1) edge (12.4,0.88)
(13,1) edge (12.85,0.4)
(13,1) edge (12.4,1.12)
(13,1) edge (12.85,1.6)
(13,3) edge (12.4,3.12)
(13,3) edge (12.85,3.6)
(13,3) edge (12.4,2.88)
(13,3) edge (12.85,2.4);
\draw [line width=0.25mm, fill=white] (11,1) circle (0.75mm);
\draw [line width=0.25mm, fill=white] (11,3) circle (0.75mm);
\draw [line width=0.25mm, fill=white] (13,2) circle (0.75mm);
\draw [line width=0.25mm, fill=black] (11,2) circle (0.75mm);
\draw [line width=0.25mm, fill=black] (13,1) circle (0.75mm);
\draw [line width=0.25mm, fill=black] (13,3) circle (0.75mm);
\node at (11.45,0.78) {\rotatebox{-45}{$\vdots$}};
\node at (11.25,3.45) {\rotatebox{45}{$\vdots$}};
\node at (12.52,0.75) {\rotatebox{45}{$\vdots$}};
\node at (12.72,1.45) {\rotatebox{-45}{$\vdots$}};
\node at (12.52,2.75) {\rotatebox{45}{$\vdots$}};
\node at (12.72,3.45) {\rotatebox{-45}{$\vdots$}};
\path[->,font=\large, >=angle 90, line width=0.4mm,blue]
(11.7, 1.25) edge (12.3, 0.75)
(12.3, 1.25) edge (11.7, 0.75)
(11.7, 3.25) edge (12.3, 2.75)
(12.3, 3.25) edge (11.7, 2.75)
(11.3,2.8) edge (10.7, 2.3)
(10.7, 1.7) edge (11.3,1.2)
(12.7, 2.8) edge (13.3,2.3)
(13.3,1.7) edge (12.7, 1.2)
(10.26,3.9) edge (10.25,4)
(13.74,3.9) edge (13.75,4);
\draw [line width=0.4mm,blue] plot [smooth, tension=0.8] coordinates { (10.25,0) (10.5,1) (11.5,2) (10.5,3) (10.25,4) };
\draw [line width=0.4mm,blue] plot [smooth, tension=0.8] coordinates { (13.75,0) (13.5,1) (12.5,2) (13.5,3) (13.75,4) };
\end{tikzpicture}
\end{center}

For each area enclosed by edges in these figures, the strands drawn above must be the only strands that enter or exit the area.  If we test all the ways to connect the entering strands to the exiting strands, we find that in each case, there is only one way to do this while maintaining the rules of Postnikov diagrams and without creating additional edges between the vertices on the strands:
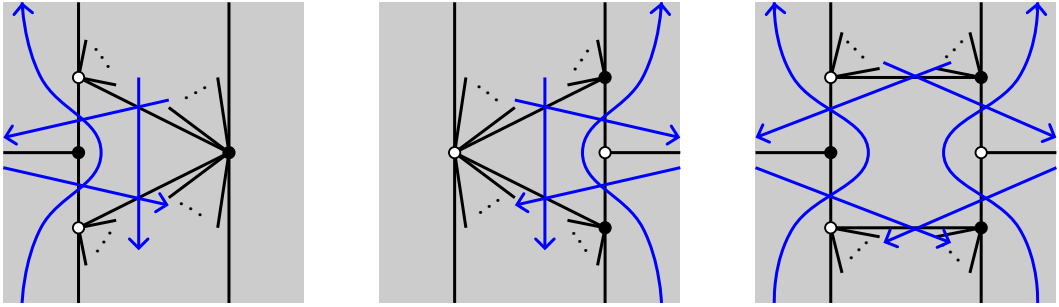
\begin{center}
\begin{tikzpicture}[scale=1]
\fill[gray!40!white] (0,0) rectangle (4,4);
\path[-,font=\large, >=angle 90, line width=0.4mm]
(1,0) edge (1,4)
(3,0) edge (3,4)
(0,2) edge (1,2)
(1,1) edge (3,2)
(1,3) edge (3,2)
(1,1) edge (1.5,1.1)
(1,1) edge (1.1,0.5)
(1,3) edge (1.5, 2.9)
(1,3) edge (1.1, 3.5)
(3,2) edge (2.2, 1.4)
(3,2) edge (2.85, 1)
(3,2) edge (2.2, 2.6)
(3,2) edge (2.85, 3);
\draw [line width=0.25mm, fill=white] (1,1) circle (0.75mm);
\draw [line width=0.25mm, fill=white] (1,3) circle (0.75mm);
\draw [line width=0.25mm, fill=black] (1,2) circle (0.75mm);
\draw [line width=0.25mm, fill=black] (3,2) circle (0.75mm);
\node at (1.4,0.9) {\rotatebox{-35}{$\vdots$}};
\node at (1.25,3.35) {\rotatebox{35}{$\vdots$}};
\node at (2.65,2.85) {\rotatebox{-60}{$\vdots$}};
\node at (2.4,1.3) {\rotatebox{60}{$\vdots$}};
\path[->,font=\large, >=angle 90, line width=0.4mm,blue]
(2.2, 2.7) edge (0,2.2)
(0,1.8) edge (2.2, 1.3)
(1.8,3) edge (1.8,0.7)
(0.26,3.9) edge (0.25,4);
\draw [line width=0.4mm,blue] plot [smooth, tension=0.8] coordinates { (0.25,0) (0.5,1) (1.3,2) (0.5,3) (0.25,4) };
\fill[gray!40!white] (5,0) rectangle (9,4);
\path[-,font=\large, >=angle 90, line width=0.4mm]
(6,0) edge (6,4)
(8,0) edge (8,4)
(8,2) edge (9,2)
(8,1) edge (6,2)
(8,3) edge (6,2)
(8,1) edge (7.5,1.1)
(8,1) edge (7.9,0.5)
(8,3) edge (7.5, 2.9)
(8,3) edge (7.9, 3.5)
(6,2) edge (6.8, 1.4)
(6,2) edge (6.15, 1)
(6,2) edge (6.8, 2.6)
(6,2) edge (6.15, 3);
\draw [line width=0.25mm, fill=white] (6,2) circle (0.75mm);
\draw [line width=0.25mm, fill=white] (8,2) circle (0.75mm);
\draw [line width=0.25mm, fill=black] (8,1) circle (0.75mm);
\draw [line width=0.25mm, fill=black] (8,3) circle (0.75mm);
\node at (7.75,3.32) {\rotatebox{-35}{$\vdots$}};
\node at (7.6,0.9) {\rotatebox{35}{$\vdots$}};
\node at (6.55,1.35) {\rotatebox{-55}{$\vdots$}};
\node at (6.35,2.85) {\rotatebox{55}{$\vdots$}};
\path[->,font=\large, >=angle 90, line width=0.4mm,blue]
(6.8, 2.7) edge (9,2.2)
(9,1.8) edge (6.8, 1.3)
(7.2,3) edge (7.2,0.7)
(8.74,3.9) edge (8.75,4);
\draw [line width=0.4mm,blue] plot [smooth, tension=0.8] coordinates { (8.75,0) (8.5,1) (7.7,2) (8.5,3) (8.75,4) };
\fill[gray!40!white] (10,0) rectangle (14,4);
\path[-,font=\large, >=angle 90, line width=0.4mm]
(11,0) edge (11,4)
(13,0) edge (13,4)
(10,2) edge (11,2)
(13,2) edge (14,2)
(11,1) edge (13,1)
(11,3) edge (13,3)
(11,1) edge (11.65,0.88)
(11,1) edge (11.15,0.4)
(11,3) edge (11.65,3.12)
(11,3) edge (11.15,3.6)
(13,1) edge (12.4,0.88)
(13,1) edge (12.85,0.4)
(13,3) edge (12.4,3.12)
(13,3) edge (12.85,3.6);
\draw [line width=0.25mm, fill=white] (11,1) circle (0.75mm);
\draw [line width=0.25mm, fill=white] (11,3) circle (0.75mm);
\draw [line width=0.25mm, fill=white] (13,2) circle (0.75mm);
\draw [line width=0.25mm, fill=black] (11,2) circle (0.75mm);
\draw [line width=0.25mm, fill=black] (13,1) circle (0.75mm);
\draw [line width=0.25mm, fill=black] (13,3) circle (0.75mm);
\node at (11.45,0.78) {\rotatebox{-45}{$\vdots$}};
\node at (11.25,3.45) {\rotatebox{45}{$\vdots$}};
\node at (12.52,0.75) {\rotatebox{45}{$\vdots$}};
\node at (12.72,3.45) {\rotatebox{-45}{$\vdots$}};
\path[->,font=\large, >=angle 90, line width=0.4mm,blue]
(12.6, 3.2) edge (10,2.2)
(10,1.8) edge (12.6, 0.8)
(11.7, 3.2) edge (14,2.2)
(14,1.8) edge (11.7, 0.8)
(10.26,3.9) edge (10.25,4)
(13.74,3.9) edge (13.75,4);
\draw [line width=0.4mm,blue] plot [smooth, tension=0.8] coordinates { (10.25,0) (10.5,1) (11.5,2) (10.5,3) (10.25,4) };
\draw [line width=0.4mm,blue] plot [smooth, tension=0.8] coordinates { (13.75,0) (13.5,1) (12.5,2) (13.5,3) (13.75,4) };
\end{tikzpicture}
\end{center}

These Postnikov diagrams don't have any interior vertices, which is a contradiction.
\end{proof}

Now we are ready to prove Theorem~\ref{thm:int-vert-k-loop}.
\begin{proof}
First, we'll just consider the first and second farthest right loops.  If there are interior vertices, then by Lemma~\ref{lemma:black-left}, there must be a black interior vertex adjacent to the left string.  The proof of Lemma~\ref{lemma:mult-edges} tells us that there must be a black vertex connected by multiple nonparallel edges to the left string.  Then, by Lemmas~\ref{lemma:same-int-vert} and~\ref{lemma:sq-move}, we can do a square move to reduce the number of crossings between the loops in the associated alternating strange diagram.  Since we are doing this square move on a black vertex, we are pushing these crossings to the left (as pictured in the proof of Lemma~\ref{lemma:sq-move}).  We can continue this process until there are no interior vertices between the first and second farthest right loops.  Now we'll look at the second and third farthest right loops and, using the same method, remove all interior vertices here.  Notice that we do not create any interior vertices between the first and second farthest right loops, as we are always pushing crossings to the left.  Repeating this process with each pair of loops, going from right to left, removes all interior vertices.
\end{proof}

\section{Proof of Theorems~\ref{thm:Te} and~\ref{thm:Tf}}\label{sec:R-matThm}

\begin{lemma}\label{lemma:lambda-rel}
We have the following relation between $\lambda_j$ and $\lambda_{j+1}$: $$\lambda_{j+1}(x,y,z)=\frac{y_j\lambda_j(x,y,z)+z_j\left(\prod_{i=1}^n x_i-\prod_{i=1}^n y_i\right)}{x_{j+1}}$$
\end{lemma}
\begin{proof}
If we take a path that gives a term in $\lambda_j$ and does not start with $z_j$, we can add two additional edges at the end on the right string and remove two edges at the beginning on the left string, to get a path that gives a term in $\lambda_{j+1}$.  This process gives us all the terms in $\lambda_{j+1}$ except the one that comes from the path that stays on the left string and crosses over at the last moment on $z_j$.
\end{proof}

We can choose one black vertex on the left string of a cylindric 2-loop extended plabic network and the white vertex on the right string that is part of the same interior face.  Adding one edge of weight $p$ and one edge of weight $-p$ from the white vertex to the black vertex does not change any of the boundary measurements, because every time there is a path that goes through the edge with weight $p$, there is another path exactly the same except it goes through the edge of weight $-p$.  The weights of these paths will cancel out when computing the boundary measurements.  We can expand these two vertices into 3 vertices each, as seen in Figure~\ref{fig:edge-p}, where the top white vertex on the right and bottom black vertex on the left may or may not have a third edge, depending on whether or not the original white and black vertices were of degree 2 or 3 before adding the edges weighted $p$ and $-p$.

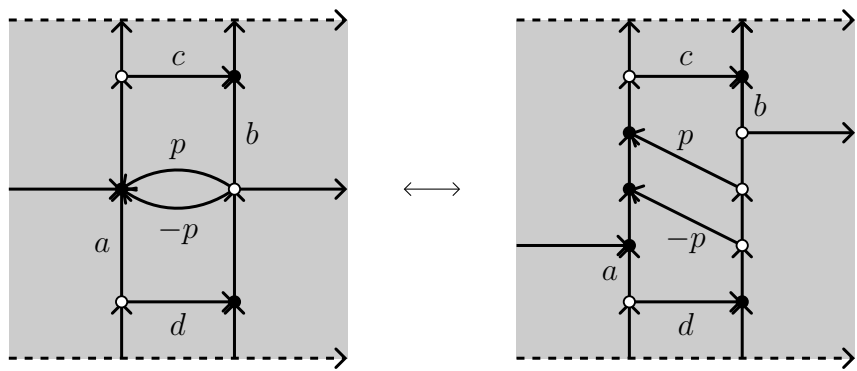
\begin{figure}[htp]
\begin{center}
\begin{tikzpicture}[scale=0.75]
\fill[gray!40!white] (0,0) rectangle (6,6);
\path[->,font=\large, >=angle 90, line width=0.4mm]
(0,3) edge (2,3)
(4,3) edge (6,3)
(2,0) edge (2,1)
(2,1) edge node[left] {$a$} (2,3)
(2,3) edge (2,5)
(2,5) edge (2,6)
(4,0) edge (4,1)
(4,1) edge (4,3)
(4,3) edge node[right] {$b$} (4,5)
(4,5) edge (4,6)
(2,1) edge node[below] {$d$} (4,1)
(2,5) edge node[above] {$c$} (4,5)
(4,3) edge[bend left=35] node[below] {$-p$} (2,3)
(4,3) edge[bend right=35] node[above] {$p$} (2,3);
\path[dashed,->,font=\large, >=angle 90, line width=0.4mm]
(0,0) edge (6,0)
(0,6) edge (6,6);
\draw [line width=0.25mm, fill=white] (2,1) circle (1mm);
\draw [line width=0.25mm, fill=white] (2,5) circle (1mm);
\draw [line width=0.25mm, fill=white] (4,3) circle (1mm);
\draw [line width=0.25mm, fill=black] (2,3) circle (1mm);
\draw [line width=0.25mm, fill=black] (4,1) circle (1mm);
\draw [line width=0.25mm, fill=black] (4,5) circle (1mm);
\path[<->,font=\large, >=angle 90]
(7,3) edge (8,3);
\fill[gray!40!white] (9,0) rectangle (15,6);
\path[->,font=\large, >=angle 90, line width=0.4mm]
(9,2) edge (11,2)
(13,4) edge (15,4)
(11,0) edge (11,1)
(11,1) edge node[left] {$a$} (11,2)
(11,2) edge (11,3)
(11,3) edge (11,4)
(11,4) edge (11,5)
(11,5) edge (11,6)
(13,0) edge (13,1)
(13,1) edge (13,2)
(13,2) edge (13,3)
(13,3) edge node[right] {$b$} (13,6)
(13,4) edge (13,5)
(13,5) edge (13,6)
(11,1) edge node[below] {$d$} (13,1)
(11,5) edge node[above] {$c$} (13,5)
(13,2) edge node[below] {$-p$} (11,3)
(13,3) edge node[above] {$p$} (11,4);
\path[dashed,->,font=\large, >=angle 90, line width=0.4mm]
(9,0) edge (15,0)
(9,6) edge (15,6);
\draw [line width=0.25mm, fill=white] (11,1) circle (1mm);
\draw [line width=0.25mm, fill=white] (11,5) circle (1mm);
\draw [line width=0.25mm, fill=white] (13,2) circle (1mm);
\draw [line width=0.25mm, fill=white] (13,3) circle (1mm);
\draw [line width=0.25mm, fill=white] (13,4) circle (1mm);
\draw [line width=0.25mm, fill=black] (11,2) circle (1mm);
\draw [line width=0.25mm, fill=black] (11,3) circle (1mm);
\draw [line width=0.25mm, fill=black] (11,4) circle (1mm);
\draw [line width=0.25mm, fill=black] (13,1) circle (1mm);
\draw [line width=0.25mm, fill=black] (13,5) circle (1mm);
\end{tikzpicture}
\end{center}
\caption{Adding the edges $p$ and $-p$ to a cylindric 2-loop directed plabic network.}
\label{fig:edge-p}
\end{figure}

Notice that what was a hexagon is now a pentagon on top of a square on top of a pentagon.  We can turn the upper pentagon into a square and perform the square move.  This will allow us to perform the square move on the face above that, and so on.  When we change original edge weights $x_{i+1}, y_i$, and $z_i$ by a square move, we will denote the new weights by $\widetilde{x_{i+1}}, \widetilde{y_i}$, and $\widetilde{z_i}$, as in Figure~\ref{fig:sq-move}.  We can continue doing square moves in this way until we end up with a graph that looks the same as the one we started with.  Now the edge weighted $-p$ will be the edge at the top of the square and the edge originally weighted $p$, which has a a new weight, has been pushed around the cylinder so that it is under the edge weighted $-p$.  Notice that each original edge weight is changed by a square move exactly once, so ignoring the two edges oriented from right to left, we end up with a new graph with edge weights $\widetilde{x},\widetilde{y}$, and $\widetilde{z}$.

Each square move looks as depicted in Figure~\ref{fig:sq-move}.  From the figure, we see that as we perform the sequence of square moves, there is always one edge, besides the fixed edge weighted $-p$, oriented from right to left.  

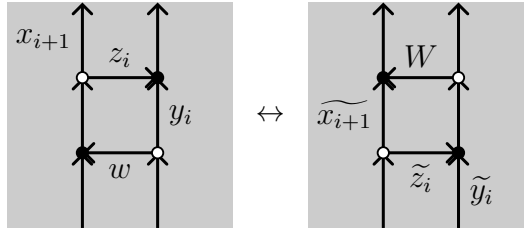
\begin{figure}[htb]
\begin{center}
\begin{tikzpicture}[scale=0.5]
\fill[gray!40!white] (-3,-3) rectangle (3,3);
\path[->,font=\large, >=angle 90, line width=0.4mm]
(1,-1) edge node[below] {$w$} (-1,-1)
(-1,-1) edge (-1,1)
(1,-1) edge node[right] {$y_i$} (1,1)
(-1,1) edge node[above] {$z_i$} (1,1)
(1,1) edge (1,3)
(-1,1) edge node[left] {$x_{i+1}$} (-1,3)
(1,-3) edge (1,-1)
(-1,-3) edge (-1,-1);
\draw [line width=0.25mm, fill=white] (-1,1) circle (1.5mm);
\draw [line width=0.25mm, fill=black] (1,1) circle (1.5mm);
\draw [line width=0.25mm, fill=white] (1,-1) circle (1.5mm);
\draw [line width=0.25mm, fill=black] (-1,-1) circle (1.5mm);
\node at (4,0) {$\leftrightarrow$};
\fill[gray!40!white] (5,-3) rectangle (11,3);
\path[->,font=\large, >=angle 90, line width=0.4mm]
(7,-1) edge node[below] {$\widetilde{z_i}$} (9,-1)
(7,-1) edge node[left] {$\widetilde{x_{i+1}}$} (7,1)
(9,-1) edge (9,1)
(9,1) edge node[above] {$W$} (7,1)
(9,1) edge (9,3)
(7,1) edge(7,3)
(9,-3) edge node[right] {$\widetilde{y_i}$} (9,-1)
(7,-3) edge (7,-1);
\draw [line width=0.25mm, fill=black] (7,1) circle (1.5mm);
\draw [line width=0.25mm, fill=white] (9,1) circle (1.5mm);
\draw [line width=0.25mm, fill=black] (9,-1) circle (1.5mm);
\draw [line width=0.25mm, fill=white] (7,-1) circle (1.5mm);
\end{tikzpicture}
\caption{A square move in the sequence of moves that push the edge weighted $p$ around the cylinder.}
\label{fig:sq-move}
\end{center}
\end{figure}

\begin{lemma}\label{lemma:unique-p}
There is at most one value of $p$ for which the weight of the bottom edge of the square in the network, after the square move has propagated all the way around the cylinder, is equal to $p$.
\end{lemma}
\begin{proof}
If we are preforming the move where the edge oriented from right to left is pushed from between $\widetilde{z_{i-1}}$ and $z_i$ to between $\widetilde{z_i}$ and $z_{i+1}$, then suppose this edge is weighted $w=\frac{ap}{bp+c}$ where $a,b,c$ don't depend on $p$.  Then the weight of the edge oriented from right to left after the square move is $\frac{x_{i+1}}{z_i(1+\frac{y_i}{z_iw})}=\frac{(ax_{i+1})p}{(az_i+by_i)p+cy_i}$.  So, by induction, we find that the weight of the bottom edge of the square in the network, after the square move has propagated all the way around the cylinder, is a constant times $p$ divided by a linear function of $p$.  Setting this equal to $p$, we get a linear equation, so there is at most one solution.
\end{proof}

\begin{thm}\label{thm:tilde-vars}
If we add the edges weighted by $p$ and $-p$ such that the the face above edge $p$ has the edge labeled $y_j$, and the face below edge $-p$ has the edge labeled $x_j$, then let $p=p_j=\displaystyle\frac{\prod_{i=1}^n x_i-\prod_{i=1}^n y_i}{\lambda_j(x,y,z)}$ and perform the sequence of square moves described above.  Then when the edge directed from the right string to the left string that is not the edge weighted by $-p$ is just above the face with the edge labeled $\widetilde{x_k}$, the weight of this edge is $p_k=\displaystyle\frac{\prod_{i=1}^n x_i-\prod_{i=1}^n y_i}{\lambda_k(x,y,z)}$.  After doing all the square moves, we have the following values of $\widetilde{x_i},\widetilde{y_i}$, and $\widetilde{z_i}$:$$\widetilde{x_i}=\frac{y_{i-1}\lambda_{i-1}(x,y,z)}{\lambda_i(x,y,z)},$$$$\widetilde{y_i}=\frac{x_{i+1}\lambda_{i+1}(x,y,z)}{\lambda_i(x,y,z)},$$$$\widetilde{z_i}=z_i.$$
\end{thm}
\begin{proof}
We only need to show this for one square move.
\begin{align*}
\widetilde{p_j}&=\frac{x_{j+1}}{z_j\left(1+\frac{y_j}{z_jp_j}\right)}\\
&=\frac{x_{j+1}\left(\frac{\prod_{i=1}^n x_i-\prod_{i=1}^n y_i}{\lambda_j(x,y,z)}\right)}{z_j\left(\frac{\prod_{i=1}^n x_i-\prod_{i=1}^n y_i}{\lambda_j(x,y,z)}\right)+y_j}\\
&=\frac{x_{j+1}\left(\prod_{i=1}^n x_i-\prod_{i=1}^n y_i\right)}{z_j\left(\prod_{i=1}^n x_i-\prod_{i=1}^n y_i\right)+y_j\lambda_j(x,y,z)}\\
&=\frac{\left(\prod_{i=1}^n x_i-\prod_{i=1}^n y_i\right)}{\lambda_{j+1}(x,y,z)},\text{ by Lemma~\ref{lemma:lambda-rel}}\\
\widetilde{x_{j+1}}&=\frac{x_{j+1}}{1+\frac{z_jp_j}{y_j}}\\
&=\frac{x_{j+1}}{1+\frac{z_j\left(\prod_{i=1}^n x_i-\prod_{i=1}^n y_i\right)}{y_j\lambda_j(x,y,z)}}\\
&=\frac{x_{j+1}y_j\lambda_j(x,y,z)}{y_j\lambda_j(x,y,z)+z_j\left(\prod_{i=1}^n x_i-\prod_{i=1}^n y_i\right)}\\
&=\frac{y_j\lambda_j(x,y,z)}{\lambda_{j+1}(x,y,z)},\text{ by Lemma~\ref{lemma:lambda-rel}}\\
\widetilde{y_j}&=y_j\left(1+\frac{z_jp_j}{y_j}\right)\\
&=y_j+\frac{z_j\left(\prod_{i=1}^n x_i-\prod_{i=1}^n y_i\right)}{\lambda_j(x,y,z)}\\
&=\frac{y_j\lambda_j(x,y,z)+z_j\left(\prod_{i=1}^n x_i-\prod_{i=1}^n y_i\right)}{\lambda_j(x,y,z)}\\
&=\frac{x_{j+1}\lambda_{j+1}(x,y,z)}{\lambda_j(x,y,z)},\text{ by Lemma~\ref{lemma:lambda-rel}}\\
\widetilde{z_j}&=\frac{\widetilde{y_j}}{p_j\left(1+\frac{y_j}{z_jp_j}\right)}\\
&=\frac{\frac{x_{j+1}\lambda_{j+1}(x,y,z)}{\lambda_j(x,y,z)}}{p_j+\frac{y_j}{z_j}}\\
&=\frac{x_{j+1}z_j\lambda_{j+1}(x,y,z)}{z_jp_j\lambda_j(x,y,z)+y_j\lambda_j(x,y,z)}\\
&=\frac{x_{j+1}z_j\lambda_{j+1}(x,y,z)}{z_j\left(\prod_{i=1}^n x_i-\prod_{i=1}^n y_i\right)+y_j\lambda_j(x,y,z)}\\
&=z_j,\text{ by Lemma~\ref{lemma:lambda-rel}}
\end{align*}
\end{proof}

\begin{thm}\label{thm:prime-vars}
After pushing the square move through every face, we can remove the edges weighted $p$ and $-p$.  Then we can use gauge transformations to force all the edge weights that were originally set to 1 equal to 1 again.  We obtain the new edge weights $x_i', y_i', z_i'$:
\begin{align*}
x_i'&=\begin{cases} 1 & x_i\text{ set to }1, \\ \frac{\left(\prod_{k=0}^{\alpha_i} y_{i-k-1}\right)\lambda_{i-\alpha_i-1}(x,y,z)}{\lambda_i(x,y,z)} & \text{otherwise,} \end{cases}\\
y_i'&=\begin{cases} 1 & y_i\text{ set to }1, \\ \frac{\left(\prod_{k=0}^{\beta_i} x_{i+k+1}\right)\lambda_{i+\beta_i+1}(x,y,z)}{\lambda_i(x,y,z)} & \text{otherwise,} \end{cases}\\
z_i'&=z_i\left(\prod_{\ell\in A_i} \widetilde{x_\ell}\right)\left(\prod_{m\in B_i} \widetilde{y_m}\right)\\
&=\frac{z_i\left(\prod_{k=1}^{\alpha_i} y_{i-k-1}\right)\lambda_{i-\alpha_i-1}(x,y,z)\left(\prod_{k=1}^{\beta_i} x_{i+k+1}\right)\lambda_{i+\beta_i+1}(x,y,z)}{\lambda_{i-1}(x,y,z)\lambda_{i+1}(x,y,z)}.
\end{align*}
\end{thm}
\begin{proof}
Clear by computation.
\end{proof}

We can now prove the first part of Theorem~\ref{thm:Te}.
\begin{proof}
$T_e$ is the transformation we obtain when we do the process from Theorem~\ref{thm:tilde-vars} and then Theorem~\ref{thm:prime-vars}.  Since adding the $p$ and $-p$ edges, applying square moves, removing the $p$ and $-p$ edges, and applying gauge transformations all preserve the boundary measurements, $T_e$ preserves the boundary measurements.
\end{proof}

\begin{lemma}\label{lemma:lambdas-equal}
For any $i, \lambda_i(x,y,z)=\lambda_i(\widetilde{x},\widetilde{y},\widetilde{z})$.
\end{lemma}
\begin{proof}
First suppose that we are in the case where the expanded plabic graph has no vertices of degree 2.  Then $\lambda_i(x,y,z)$ is exactly the coefficient of a power of $\zeta$ in one of the boundary measurements.  So $\lambda_i(x,y,z)=\lambda_i(x'y'z')$, as the transformation does not change the boundary measurements.  But in this case, $(x'y'z')=(\widetilde{x},\widetilde{y},\widetilde{z})$, so $\lambda_i(x,y,z)=\lambda_i(\widetilde{x},\widetilde{y},\widetilde{z})$.

Now suppose we are in the general case.  The lack of a third edge on some vertices does not affect the square move in any way.  The transformation from $(x,y,z)$ to $(\widetilde{x},\widetilde{y},\widetilde{z})$ is the same as if all the vertices were degree 3, and thus, we still have that $\lambda_i(x,y,z)=\lambda_i(\widetilde{x},\widetilde{y},\widetilde{z})$.
\end{proof}

Since $\prod_{i=1}^n x_i=\prod_{i=1}^n \widetilde{y_i}$, $\prod_{i=1}^n y_i=\prod_{i=1}^n \widetilde{x_i}$, and $\lambda_i(x,y,z)=\lambda_i(\widetilde{x},\widetilde{y},\widetilde{z})$, $-p_j$ from our original network is equal to $p_j$ from the network with edge weights $(\widetilde{x},\widetilde{y},\widetilde{z})$.  Thus, we can find $(\widetilde{\widetilde{x}},\widetilde{\widetilde{y}},\widetilde{\widetilde{z}})$ by taking our original network, moving the edge $p_j$ around by square moves, and then moving the edge $-p_j$ around by square moves.

\begin{lemma}\label{lemma:p'}
After applying the gauge transformations to the network to get from $(\widetilde{x},\widetilde{y},\widetilde{z})$ to $(x',y',z')$, the edges labeled by $p_j,-p_j$ in the $(\widetilde{x},\widetilde{y},\widetilde{z})$ network are replaced with $-p_j'$ and $p_j'=\displaystyle\frac{\prod_{i=1}^n x_i'-\prod_{i=1}^n y_i'}{\lambda_j(x',y',z')}$, respectively.
\end{lemma}
\begin{proof}
In the network with edges $(\widetilde{x},\widetilde{y},\widetilde{z})$, consider $p_j\lambda_j(\widetilde{x},\widetilde{y},\widetilde{z})=\prod_{i=1}^n x_i-\prod_{i=1}^n y_i$ = a sum of weights of some cycles beginning and ending at the vertex on the bottom right of the face between the two strings with edge $\widetilde{x_j}$.  Since gauge transformations don't affect the weight of cycles, $p_j\lambda_j(\widetilde{x},\widetilde{y},\widetilde{z})=-p_j'\lambda_j(x',y',z')$ where $-p_j'$ is the weight of the edge previously weighted $p_j$ after applying gauge transformations to obtain the network with edge weights $(x',y',z')$ from the one with edge weights $(\widetilde{x},\widetilde{y},\widetilde{z})$.  So, $p'$ must be $\displaystyle\frac{\prod_{i=1}^n y_i-\prod_{i=1}^n x_i}{\lambda_1(x',y',z')}=\displaystyle\frac{\prod_{i=1}^n x_i'-\prod_{i=1}^n y_i'}{\lambda_1(x',y',z')}$.  A similar argument holds for $-p$.
\end{proof}

\begin{lemma}\label{lemma:gauge-square-comm}
Gauge transformations commute with the square move.
\end{lemma}
\begin{proof}
Since our formulas for the square move come from the face weighted square move, and gauge transformations don't affect face weights, gauge transformations commute with the square move.
\end{proof}

\begin{thm}\label{thm:comm-diag}
We have the following commutative diagram.

\begin{tikzpicture}[scale=1.5]
\node(A) at (0,2) {$(x,y,z)$};
\node(B) at (3,2) {$(\widetilde{x},\widetilde{y},\widetilde{z})$};
\node(C) at (6,2) {$(x',y',z')$};
\node(D) at (3,1) {$(\widetilde{\widetilde{x}},\widetilde{\widetilde{y}},\widetilde{\widetilde{z}})$};
\node(E) at (6,1) {$(\widetilde{x'},\widetilde{y'},\widetilde{z'})$};
\node(F) at (4.5,0) {$(\left(\widetilde{x}\right)',\left(\widetilde{y}\right)',\left(\widetilde{z}\right)')=(x'',y'',z'')$};
\path[->,font=\large, >=angle 90]
(A) edge node[above] {square moves} (B)
(B) edge node[above] {gauge trans.} (C)
(B) edge node[left] {square moves} (D)
(C) edge node[right] {square moves} (E)
(D) edge node[left] {gauge trans.} (F)
(E) edge node[right] {gauge trans.} (F)
(D) edge node[above] {gauge trans.} (E);
\end{tikzpicture}
\end{thm}
\begin{proof}
By Lemma~\ref{lemma:p'} we can get from $(x',y',z')$ to $(\widetilde{x'},\widetilde{y'},\widetilde{z'})$ by square moves.  By Lemma~\ref{lemma:lambdas-equal}, we can find $(\widetilde{\widetilde{x}},\widetilde{\widetilde{y}},\widetilde{\widetilde{z}})$ from $(\widetilde{x},\widetilde{y},\widetilde{z})$ by doing square moves.  Lemma~\ref{lemma:gauge-square-comm} tells us we have the bottom arrow of the square in the diagram such that the square commutes.  Both $(\left(\widetilde{x}\right)',\left(\widetilde{y}\right)',\left(\widetilde{z}\right)')$ and $(x'',y'',z'')$ can be obtained from $(\widetilde{\widetilde{x}},\widetilde{\widetilde{y}},\widetilde{\widetilde{z}})$ by gauge transformations.  Since $(\left(\widetilde{x}\right)',\left(\widetilde{y}\right)',\left(\widetilde{z}\right)')$ and $(x'',y'',z'')$ both have the same edge weights set equal to 1 and the number of edge weights not set to be 1 is the same as the dimension of the space of face and trail weights, $(\left(\widetilde{x}\right)',\left(\widetilde{y}\right)',\left(\widetilde{z}\right)')=(x'',y'',z'')$.
\end{proof}

We can now prove the second part of Theorem~\ref{thm:Te}.
\begin{proof}
This is a matter of checking that $(x'',y'',z'')=(x,y,z)$.  We will to this using the left branch of our commutative diagram in Theorem~\ref{thm:comm-diag}.

Notice that rather than moving the edge $p$ all the way around and then moving the edge $-p$ all the way around, we can do one square move with $p$ and one square move with $-p$.  Computations very similar to those in the proof of Theorem~\ref{thm:tilde-vars} show that $(\widetilde{\widetilde{x}},\widetilde{\widetilde{y}},\widetilde{\widetilde{z}})=(x,y,z)$.  This means we don't have to do any gauge transformations to get to $(x'',y'',z'')$ and in fact $(x'',y'',z'')=(\widetilde{\widetilde{x}},\widetilde{\widetilde{y}},\widetilde{\widetilde{z}})=(x,y,z)$.
\end{proof}

Suppose we have a cylindric 2-loop plabic network with no interior vertices.  Consider the expanded graph on its universal cover.  Let $y_i$ be a weight in the expanded plabic graph that has not been set to 1.  Let $j$ be minimal such that $j>i$ and $x_j$ is not set to 1.  Let $k,\ell$ be maximal such that $k\leq i, \ell<i$ and $x_k,y_\ell$ are not set to 1.  Let $a$ be the source below $x_k$, $b$ be the source between $x_j$ and $x_k$, $c$ be the source above $x_j$, $d$ be the sink between $y_i$ and $y_\ell$, and $e$ be the sink below $y_\ell$.  Let $n(w) =$ the number of times the edge $w$ crosses the cut from the right $-$ the number of times the edge $w$ crosses the cut from the left.  Then we have the following relations:
\begin{align*}
M_{bd}&=x_j(-\zeta)^{n(x_j)}M_{cd}+y_\ell (-\zeta)^{n(y_\ell)}M_{be}-x_jy_\ell (-\zeta)^{n(x_j)+n(y_\ell)}M_{ce}\\
M_{ad}&=x_k(-\zeta)^{n(x_k)}M_{bd}+y_\ell (-\zeta)^{n(y_\ell)}M_{ae}-x_ky_\ell (-\zeta)^{n(x_k)+n(y_\ell)}M_{be}
\end{align*}

Thus, we have relations between $y_\ell$ and $x_j$ and between $y_\ell$ and $x_k$.  In particular, we have:
\begin{align*}
x_j&=\frac{M_{bd}-y_\ell (-\zeta)^{n(y_\ell)}M_{be}}{(-\zeta)^{n(x_j)}M_{cd}-y_\ell (-\zeta)^{n(x_j)+n(y_\ell)}M_{ce}}\\
y_\ell&=\frac{M_{ad}-x_k(-\zeta)^{n(x_k)}M_{bd}}{(-\zeta)^{n(y_\ell)}M_{ae}-(-\zeta)^{n(x_k)+n(y_\ell)}x_k M_{be}}
\end{align*}

Combining these, we get a relation between $x_j$ and $x_k$. $$x_j=\frac{M_{bd}M_{ae}-M_{ad}M_{be}}{(-\zeta)^{n(x_j)}(M_{ae}M_{cd}-M_{ad}M_{ce})+x_k(-\zeta)^{n(x_j)+n(x_k)}(M_{bd}M_{ce}-M_{be}M_{cd})}$$

\begin{lemma}\label{lemma:xj-xn}
If we know the boundary measurements of a network, then repeated substitution with the above formula to get $x_j$ as a function of $x_n$ will always yield a linear expression in $x_n$ divided by another linear expression in $x_n$.
\end{lemma}
\begin{proof}
We only need to check one step.  Suppose we have $x_j=\frac{A+Bx_m}{C+Dx_m}$ and $x_m=\frac{E}{F+Gx_n}$, coming from the formula above.  Then $x_j=\frac{(AF+BE)+AGx_n}{(CF+DE)+CGx_n}$.
\end{proof}

Now we prove the third part of Theorem~\ref{thm:Te}.
\begin{proof}
Choose an edge on the left string that is not set to 1.  Call it $x_j$.  Let $x_k$ be the next edge below $x_j$ on the left string that is not set to one.  Suppose we do a transformation that preserves the boundary measurements and replaces $x_j$ with $x_j'$, $x_k$ with $x_k'$, and so on.  Since the transformation does not change the boundary measurements, the above relation between $x_j$ and $x_k$ holds for $x_j'$ and $x_k'$.  Similarly, if $x_\ell'$ is the next edge weight on the left string that is not set to 1, $x_k'$ and $x_\ell'$ have the same relationship.  So, substitution gives us a formula for $x_j'$ in terms of $x_\ell'$.  Repeating this process all the way around the cylinder, we get an expression for $x_j'$ in terms of itself.  By Lemma~\ref{lemma:xj-xn}, clearing the denominator gives a quadratic equation in $x_j'$.  Quadratic equations have at most two solutions, so there are at most 2 possibilities for $x_j'$.  We have identified two solutions in our involution, so we know these are the only two solutions.
\end{proof}

Suppose now that we have a cylindric 3-loop plabic graph.  By Theorem~\ref{thm:int-vert-k-loop}, we can assume without loss of generality that there are no interior vertices between the left and middle strings or the middle and right strings.

We will label the left, middle, and right strings 1, 2, and 3, respectively.  Then $T_e^{(i,i+1)}$ will be applying $T_e$ to the strings $i$ and $i+1$.  Taking the expanded directed plabic network for this graph, we can add the edges $p,-p$ between strings 1 and 2, $q,-q$ between strings 2 and 3, and $r,-r$ between strings 1 and 2 so that moving $p$ around the cylinder by square moves until it is below $-p$ is equivalent to applying $T_e^{(1,2)}$, modulo gauge transformations, then moving $q$ around the cylinder by square moves until it's below $p$ is equivalent to applying $T_e^{(2,3)}$, modulo gauge transformations, and finally moving $r$ around the cylinder by square moves until it's below $q$ is equivalent to applying $T_e^{(1,2)}$ a second time, modulo gauge transformations.  Since gauge transformations commute with square moves (Lemma~\ref{lemma:gauge-square-comm}), we can apply all of our gauge transformations at the end, and get $T_e^{(1,2)}\circ T_e^{(2,3)}\circ T_e^{(1,2)}$.

\begin{figure}[htb]
\begin{center}
\begin{tikzpicture}[scale=0.75]
\fill[gray!40!white] (0,0) rectangle (8,12);
\path[->,font=\large, >=angle 90, line width=0.4mm]
(0,3) edge (2,3)
(6,9) edge (8,9)
(2,0) edge (2,2)
(2,2) edge node[left] {$a$} (2,3)
(2,3) edge (2,4)
(2,4) edge (2,6)
(2,6) edge (2,7)
(2,7) edge (2,9)
(2,9) edge (2,11)
(2,11) edge (2,12)
(4,0) edge (4,1)
(4,1) edge node[left] {$c$} (4,2)
(4,2) edge (4,4)
(4,4) edge (4,5)
(4,5) edge (4,6)
(4,6) edge (4,7)
(4,7) edge (4,8)
(4,8) edge (4,9)
(4,9) edge (4,10)
(4,10) edge node[right] {$b$} (4,11)
(4,11) edge (4,12)
(6,0) edge (6,1)
(6,1) edge (6,5)
(6,5) edge (6,8)
(6,8) edge (6,9)
(6,9) edge node[right] {$d$} (6,10)
(6,10) edge (6,12)
(2,2) edge node[below] {$e$} (4,2)
(4,10) edge node[above] {$h$} (6,10)
(2,11) edge node[above] {$f$} (4,11)
(4,1) edge node[below] {$g$} (6,1);;
\path[<-,font=\large, >=angle 90, line width=0.4mm]
(2,4) edge node[below] {$-p$} (4,4)
(2,6) edge node[below] {$-r$} (4,6)
(2,7) edge node[above] {$r$} (4,7)
(2,9) edge node[above] {$p$} (4,9)
(4,5) edge node[below] {$-q$} (6,5)
(4,8) edge node[above] {$q$} (6,8);
\path[-,font=\large, >=angle 90, line width=0.25mm,red]
(0,6.5) edge (8,6.5);
\draw [line width=0.25mm, fill=white] (2,2) circle (1mm);
\draw [line width=0.25mm, fill=white] (2,11) circle (1mm);
\draw [line width=0.25mm, fill=white] (4,1) circle (1mm);
\draw [line width=0.25mm, fill=white] (4,4) circle (1mm);
\draw [line width=0.25mm, fill=white] (4,6) circle (1mm);
\draw [line width=0.25mm, fill=white] (4,7) circle (1mm);
\draw [line width=0.25mm, fill=white] (4,9) circle (1mm);
\draw [line width=0.25mm, fill=white] (4,10) circle (1mm);
\draw [line width=0.25mm, fill=white] (6,5) circle (1mm);
\draw [line width=0.25mm, fill=white] (6,8) circle (1mm);
\draw [line width=0.25mm, fill=white] (6,9) circle (1mm);
\draw [line width=0.25mm, fill=black] (2,3) circle (1mm);
\draw [line width=0.25mm, fill=black] (2,4) circle (1mm);
\draw [line width=0.25mm, fill=black] (2,6) circle (1mm);
\draw [line width=0.25mm, fill=black] (2,7) circle (1mm);
\draw [line width=0.25mm, fill=black] (2,9) circle (1mm);
\draw [line width=0.25mm, fill=black] (4,2) circle (1mm);
\draw [line width=0.25mm, fill=black] (4,5) circle (1mm);
\draw [line width=0.25mm, fill=black] (4,8) circle (1mm);
\draw [line width=0.25mm, fill=black] (4,11) circle (1mm);
\draw [line width=0.25mm, fill=black] (6,1) circle (1mm);
\draw [line width=0.25mm, fill=black] (6,10) circle (1mm);
\end{tikzpicture}
\end{center}
\caption{A face of a network with the edges $p,-p,q,-q,r$, and $-r$ added as described above.  The area above the red line is the top half to the face, as referenced in Lemma~\ref{lemma:8-cycle}.}
\label{fig:pqr}
\end{figure}
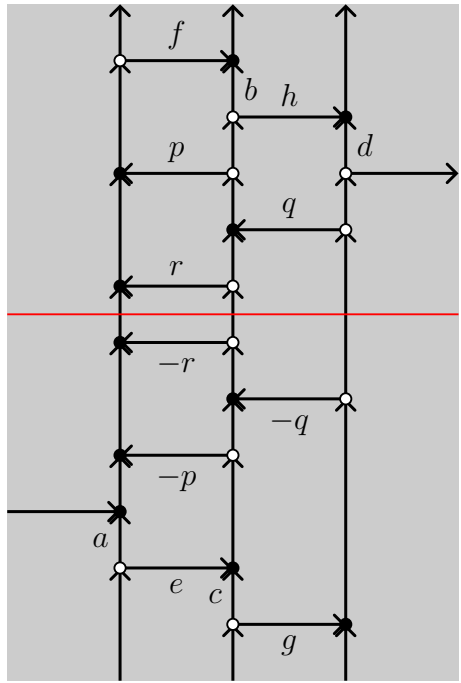

\begin{lemma}\label{lemma:8-cycle}
Starting with a plabic network that looks like the top half of a face with $p,q$, and $r$ added, there is an 8-cycle of square moves.  That is, there is a sequence of 8 square moves where the networks we obtain after each of the first 7 square moves are all different from each other and from the original network, but the network we obtain after the 8th square move is the same as the original network.
\end{lemma}
\begin{proof}
Since this 8-cycle is on a portion of the cylinder isomorphic to the disk, we will use an undirected plabic network with face weights to prove this.  This gives us a more general statement, as we could choose any orientations and sets of edge weights that give the plabic networks in the following figures, and the proof would hold true.  Thus, we will consider that we are starting with a plabic network that looks as follows:
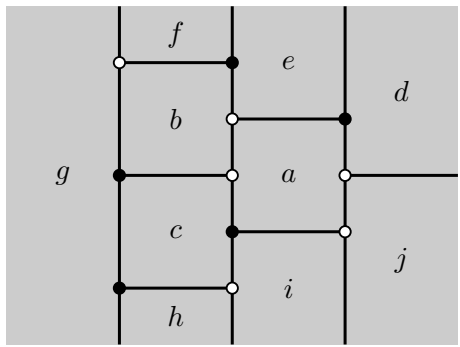
\begin{center}
\begin{tikzpicture}[scale=0.75]
\fill[gray!40!white] (0,6) rectangle (8,12);
\path[-,font=\large, >=angle 90, line width=0.4mm]
(6,9) edge (8,9)
(2,6) edge (2,7)
(2,7) edge (2,9)
(2,9) edge (2,11)
(2,11) edge (2,12)
(4,6) edge (4,7)
(4,7) edge (4,8)
(4,8) edge (4,9)
(4,9) edge (4,10)
(4,10) edge (4,11)
(4,11) edge (4,12)
(6,6) edge (6,8)
(6,8) edge (6,9)
(6,9) edge (6,10)
(6,10) edge (6,12)
(4,10) edge (6,10)
(2,11) edge (4,11)
(2,7) edge (4,7)
(2,9) edge (4,9)
(4,8) edge (6,8);
\draw [line width=0.25mm, fill=white] (2,11) circle (1mm);
\draw [line width=0.25mm, fill=white] (4,7) circle (1mm);
\draw [line width=0.25mm, fill=white] (4,9) circle (1mm);
\draw [line width=0.25mm, fill=white] (4,10) circle (1mm);
\draw [line width=0.25mm, fill=white] (6,8) circle (1mm);
\draw [line width=0.25mm, fill=white] (6,9) circle (1mm);
\draw [line width=0.25mm, fill=black] (2,7) circle (1mm);
\draw [line width=0.25mm, fill=black] (2,9) circle (1mm);
\draw [line width=0.25mm, fill=black] (4,8) circle (1mm);
\draw [line width=0.25mm, fill=black] (4,11) circle (1mm);
\draw [line width=0.25mm, fill=black] (6,10) circle (1mm);
\node at (1,9) {$g$};
\node at (3,6.5) {$h$};
\node at (3,8) {$c$};
\node at (3,10) {$b$};
\node at (3,11.5) {$f$};
\node at (5,7) {$i$};
\node at (5,9) {$a$};
\node at (5,11) {$e$};
\node at (7,7.5) {$j$};
\node at (7,10.5) {$d$};
\end{tikzpicture}
\end{center}

Let $A=1+c+ac$ and $B=1+b+bc+abc=1+bA$.  Applying unicolored edge contractions shows that the above plabic network is the same as the network in the upper left corner of the following 8-cycle:
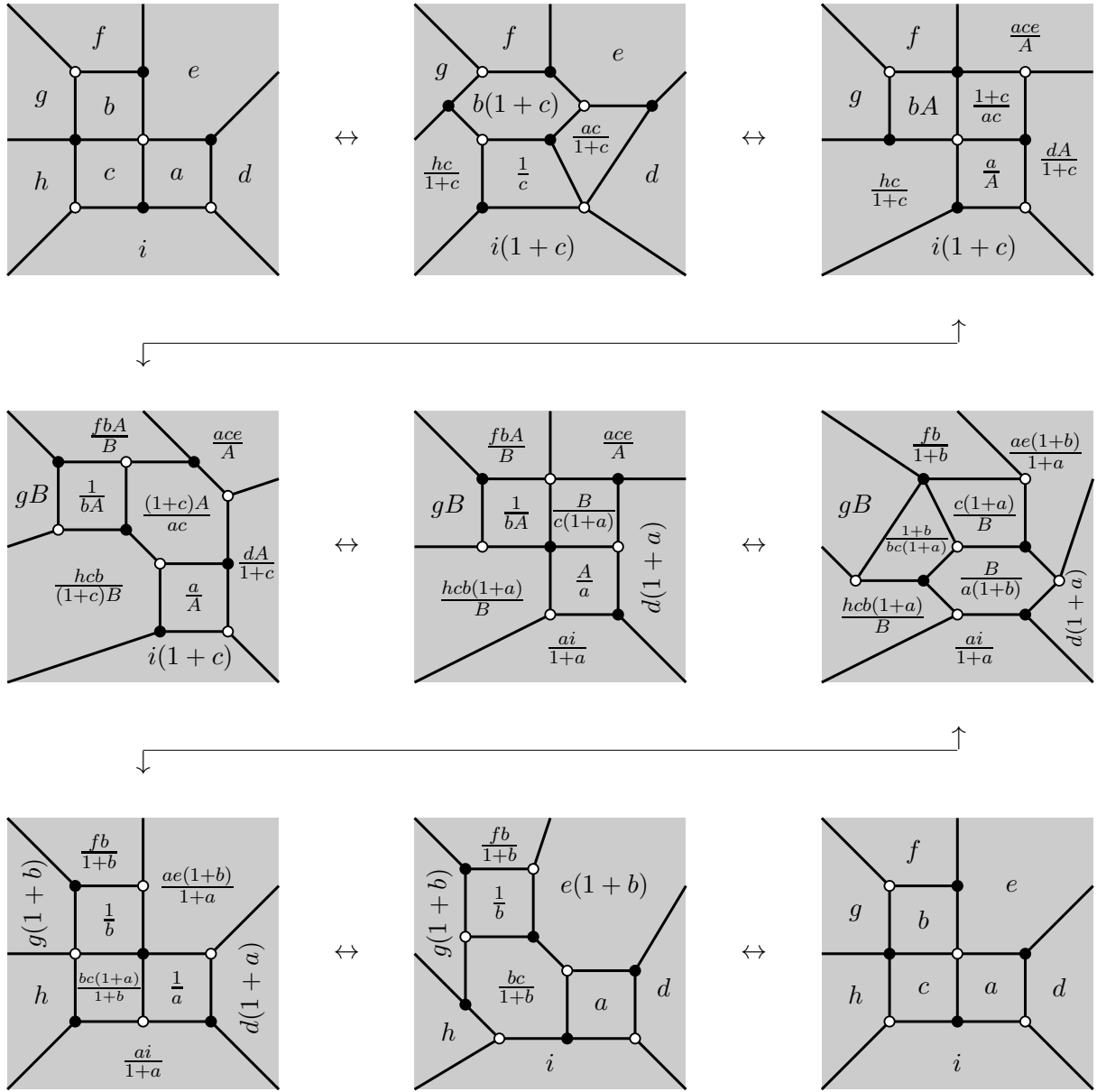
\begin{center}
\begin{tikzpicture}[scale=1]
\fill[gray!40!white] (0,6) rectangle (4,10);
\path[-,font=\large, >=angle 90, line width=0.4mm]
(0,6) edge (1,7)
(0,8) edge (1,8)
(0,10) edge (1,9)
(2,10) edge (2,9)
(4,9) edge (3,8)
(4,6) edge (3,7)
(2,7) edge (1,7)
(1,8) edge (1,7)
(1,8) edge (1,9)
(2,9) edge (1,9)
(2,9) edge (2,8)
(2,8) edge (1,8)
(2,8) edge (2,7)
(3,7) edge (2,7)
(3,7) edge (3,8)
(2,8) edge (3,8);
\draw [line width=0.25mm, fill=white] (1,7) circle (0.75mm);
\draw [line width=0.25mm, fill=white] (1,9) circle (0.75mm);
\draw [line width=0.25mm, fill=white] (2,8) circle (0.75mm);
\draw [line width=0.25mm, fill=white] (3,7) circle (0.75mm);
\draw [line width=0.25mm, fill=black] (1,8) circle (0.75mm);
\draw [line width=0.25mm, fill=black] (2,7) circle (0.75mm);
\draw [line width=0.25mm, fill=black] (2,9) circle (0.75mm);
\draw [line width=0.25mm, fill=black] (3,8) circle (0.75mm);
\node at (2.5,7.5) {$a$};
\node at (1.5,8.5) {$b$};
\node at (1.5,7.5) {$c$};
\node at (3.5,7.5) {$d$};
\node at (2.75,9) {$e$};
\node at (1.35,9.5) {$f$};
\node at (0.5,8.6) {$g$};
\node at (0.5,7.4) {$h$};
\node at (2,6.4) {$i$};
\node at (5,8) {$\leftrightarrow$};
\fill[gray!40!white] (6,6) rectangle (10,10);
\path[-,font=\large, >=angle 90, line width=0.4mm]
(6,6) edge (7,7)
(6,8) edge (6.5,8.5)
(6,10) edge (7,9)
(8,10) edge (8,9)
(10,9) edge (9.5,8.5)
(10,6) edge (8.5,7)
(6.5,8.5) edge (7,9)
(7,9) edge (8,9)
(8,9) edge (8.5,8.5)
(8.5,8.5) edge (8,8)
(8,8) edge (7,8)
(7,8) edge (6.5,8.5)
(7,8) edge (7,7)
(7,7) edge (8.5,7)
(8.5,7) edge (9.5,8.5)
(9.5,8.5) edge (8.5,8.5)
(8.5,7) edge (8,8);
\draw [line width=0.25mm, fill=white] (7,8) circle (0.75mm);
\draw [line width=0.25mm, fill=white] (7,9) circle (0.75mm);
\draw [line width=0.25mm, fill=white] (8.5,7) circle (0.75mm);
\draw [line width=0.25mm, fill=white] (8.5,8.5) circle (0.75mm);
\draw [line width=0.25mm, fill=black] (9.5,8.5) circle (0.75mm);
\draw [line width=0.25mm, fill=black] (8,8) circle (0.75mm);
\draw [line width=0.25mm, fill=black] (8,9) circle (0.75mm);
\draw [line width=0.25mm, fill=black] (7,7) circle (0.75mm);
\draw [line width=0.25mm, fill=black] (6.5,8.5) circle (0.75mm);
\node at (8.6,8) {$\frac{ac}{1+c}$};
\node at (7.5,8.5) {$b(1+c)$};
\node at (7.6,7.5) {$\frac{1}{c}$};
\node at (9.5,7.5) {$d$};
\node at (9,9.25) {$e$};
\node at (7.4,9.5) {$f$};
\node at (6.4,9) {$g$};
\node at (6.4,7.5) {$\frac{hc}{1+c}$};
\node at (7.75,6.4) {$i(1+c)$};
\node at (11,8) {$\leftrightarrow$};
\fill[gray!40!white] (12,6) rectangle (16,10);
\path[-,font=\large, >=angle 90, line width=0.4mm]
(12,6) edge (14,7)
(12,8) edge (13,8)
(12,10) edge (13,9)
(14,10) edge (14,9)
(16,9) edge (15,9)
(16,6) edge (15,7)
(14,9) edge (15,9)
(15,8) edge (15,9)
(13,8) edge (13,9)
(14,9) edge (13,9)
(14,9) edge (14,8)
(14,8) edge (13,8)
(14,8) edge (14,7)
(15,7) edge (14,7)
(15,7) edge (15,8)
(14,8) edge (15,8);
\draw [line width=0.25mm, fill=white] (15,9) circle (0.75mm);
\draw [line width=0.25mm, fill=white] (13,9) circle (0.75mm);
\draw [line width=0.25mm, fill=white] (14,8) circle (0.75mm);
\draw [line width=0.25mm, fill=white] (15,7) circle (0.75mm);
\draw [line width=0.25mm, fill=black] (13,8) circle (0.75mm);
\draw [line width=0.25mm, fill=black] (14,7) circle (0.75mm);
\draw [line width=0.25mm, fill=black] (14,9) circle (0.75mm);
\draw [line width=0.25mm, fill=black] (15,8) circle (0.75mm);
\node at (14.5,7.5) {$\frac{a}{A}$};
\node at (13.5,8.5) {$bA$};
\node at (14.5,8.5) {$\frac{1+c}{ac}$};
\node at (15.5,7.7) {$\frac{dA}{1+c}$};
\node at (15,9.5) {$\frac{ace}{A}$};
\node at (13.35,9.5) {$f$};
\node at (12.5,8.6) {$g$};
\node at (13,7.25) {$\frac{hc}{1+c}$};
\node at (14.25,6.4) {$i(1+c)$};
\draw [line width=0.1mm] plot [smooth, tension=0.8] coordinates { (14,5) (2,5) };
\node at (14.05,5.18) {\rotatebox{90}{$\rightarrow$}};
\node at (2.03,4.83) {\rotatebox{90}{$\leftarrow$}};
\fill[gray!40!white] (0,0) rectangle (4,4);
\path[-,font=\large, >=angle 90, line width=0.4mm]
(0,0) edge (2.25,0.75)
(0,2) edge (0.75,2.25)
(0,4) edge (0.75,3.25)
(2,4) edge (2.75, 3.25)
(4,3) edge (3.25, 2.75)
(4,0) edge (3.25, 0.75)
(0.75, 3.25) edge (1.75, 3.25)
(1.75, 3.25) edge (2.75, 3.25)
(2.75, 3.25) edge (3.25, 2.75)
(3.25, 2.75) edge (3.25, 1.75)
(3.25, 1.75) edge (3.25, 0.75)
(3.25,0.75) edge (2.25, 0.75)
(2.25, 0.75) edge (2.25, 1.75)
(2.25, 1.75) edge (1.75, 2.25)
(1.75, 2.25) edge (0.75, 2.25)
(0.75, 2.25) edge (0.75, 3.25)
(1.75, 2.25) edge (1.75, 3.25)
(2.25, 1.75) edge (3.25, 1.75);
\draw [line width=0.25mm, fill=black] (0.75,3.25) circle (0.75mm);
\draw [line width=0.25mm, fill=black] (2.75,3.25) circle (0.75mm);
\draw [line width=0.25mm, fill=black] (3.25,1.75) circle (0.75mm);
\draw [line width=0.25mm, fill=black] (2.25, 0.75) circle (0.75mm);
\draw [line width=0.25mm, fill=black] (1.75, 2.25) circle (0.75mm);
\draw [line width=0.25mm, fill=white] (1.75,3.25) circle (0.75mm);
\draw [line width=0.25mm, fill=white] (3.25, 2.75) circle (0.75mm);
\draw [line width=0.25mm, fill=white] (3.25,0.75) circle (0.75mm);
\draw [line width=0.25mm, fill=white] (2.25, 1.75) circle (0.75mm);
\draw [line width=0.25mm, fill=white] (0.75, 2.25) circle (0.75mm);
\node at (2.75,1.25) {$\frac{a}{A}$};
\node at (1.25,2.75) {$\frac{1}{bA}$};
\node at (2.5,2.5) {$\frac{(1+c)A}{ac}$};
\node at (3.67,1.75) {$\frac{dA}{1+c}$};
\node at (3.25,3.5) {$\frac{ace}{A}$};
\node at (1.5,3.65) {$\frac{fbA}{B}$};
\node at (0.37,2.75) {$gB$};
\node at (1.2,1.4) {$\frac{hcb}{(1+c)B}$};
\node at (2.7,0.37) {$i(1+c)$};
\node at (5,2) {$\leftrightarrow$};
\fill[gray!40!white] (6,0) rectangle (10,4);
\path[-,font=\large, >=angle 90, line width=0.4mm]
(6,0) edge (8,1)
(6,2) edge (7,2)
(6,4) edge (7,3)
(8,4) edge (8,3)
(10,3) edge (9,3)
(10,0) edge (9,1)
(8,3) edge (9,3)
(9,2) edge (9,3)
(7,2) edge (7,3)
(8,3) edge (7,3)
(8,3) edge (8,2)
(8,2) edge (7,2)
(8,2) edge (8,1)
(9,1) edge (8,1)
(9,1) edge (9,2)
(8,2) edge (9,2);
\draw [line width=0.25mm, fill=black] (7,3) circle (0.75mm);
\draw [line width=0.25mm, fill=black] (8,2) circle (0.75mm);
\draw [line width=0.25mm, fill=black] (9,3) circle (0.75mm);
\draw [line width=0.25mm, fill=black] (9,1) circle (0.75mm);
\draw [line width=0.25mm, fill=white] (8,3) circle (0.75mm);
\draw [line width=0.25mm, fill=white] (7,2) circle (0.75mm);
\draw [line width=0.25mm, fill=white] (9,2) circle (0.75mm);
\draw [line width=0.25mm, fill=white] (8,1) circle (0.75mm);
\node at (8.5,1.5) {$\frac{A}{a}$};
\node at (7.5,2.5) {$\frac{1}{bA}$};
\node at (8.5,2.5) {$\frac{B}{c(1+a)}$};
\node at (9.52,1.7) {\rotatebox{90}{$d(1+a)$}};
\node at (9,3.5) {$\frac{ace}{A}$};
\node at (7.35,3.5) {$\frac{fbA}{B}$};
\node at (6.5,2.6) {$gB$};
\node at (7,1.25) {$\frac{hcb(1+a)}{B}$};
\node at (8.25,0.5) {$\frac{ai}{1+a}$};
\node at (11,2) {$\leftrightarrow$};
\fill[gray!40!white] (12,0) rectangle (16,4);
\path[-,font=\large, >=angle 90, line width=0.4mm]
(12,0) edge (14,1)
(12,2) edge (12.5,1.5)
(12,4) edge (13.5,3)
(14,4) edge (15,3)
(16,3) edge (15.5,1.5)
(16,0) edge (15,1)
(13.5,3) edge (15,3)
(15,3) edge (15,2)
(15,2) edge (15.5,1.5)
(15.5,1.5) edge (15,1)
(15,1) edge (14,1)
(14,1) edge (13.5,1.5)
(13.5,1.5) edge (12.5,1.5)
(12.5,1.5) edge (13.5,3)
(14,2) edge (13.5,1.5)
(14,2) edge (15,2)
(14,2) edge (13.5,3);
\draw [line width=0.25mm, fill=black] (13.5,3) circle (0.75mm);
\draw [line width=0.25mm, fill=black] (15,2) circle (0.75mm);
\draw [line width=0.25mm, fill=black] (15,1) circle (0.75mm);
\draw [line width=0.25mm, fill=black] (13.5,1.5) circle (0.75mm);
\draw [line width=0.25mm, fill=white] (15,3) circle (0.75mm);
\draw [line width=0.25mm, fill=white] (15.5,1.5) circle (0.75mm);
\draw [line width=0.25mm, fill=white] (14,1) circle (0.75mm);
\draw [line width=0.25mm, fill=white] (12.5,1.5) circle (0.75mm);
\draw [line width=0.25mm, fill=white] (14,2) circle (0.75mm);
\node at (14.5,1.5) {$\frac{B}{a(1+b)}$};
\node at (13.4,2.1) {\footnotesize$\frac{1+b}{bc(1+a)}$};
\node at (14.4,2.5) {$\frac{c(1+a)}{B}$};
\node at (15.75,1.1) {\rotatebox{90}{\footnotesize$d(1+a)$}};
\node at (15.3,3.4) {$\frac{ae(1+b)}{1+a}$};
\node at (13.6,3.5) {$\frac{fb}{1+b}$};
\node at (12.5,2.6) {$gB$};
\node at (12.9,1) {$\frac{hcb(1+a)}{B}$};
\node at (14.25,0.5) {$\frac{ai}{1+a}$};
\draw [line width=0.1mm] plot [smooth, tension=0.8] coordinates { (14,-1) (2,-1) };
\node at (14.05,-0.82) {\rotatebox{90}{$\rightarrow$}};
\node at (2.03,-1.17) {\rotatebox{90}{$\leftarrow$}};
\fill[gray!40!white] (0,-6) rectangle (4,-2);
\path[-,font=\large, >=angle 90, line width=0.4mm]
(0,-6) edge (1,-5)
(0,-4) edge (1,-4)
(0,-2) edge (1,-3)
(2,-2) edge (2,-3)
(4,-3) edge (3,-4)
(4,-6) edge (3,-5)
(2,-5) edge (1,-5)
(1,-4) edge (1,-5)
(1,-4) edge (1,-3)
(2,-3) edge (1,-3)
(2,-4) edge (1,-4)
(2,-5) edge (2,-3)
(3,-4) edge (3,-5)
(3,-4) edge (2,-4)
(2,-5) edge (3,-5);
\draw [line width=0.25mm, fill=black] (1,-5) circle (0.75mm);
\draw [line width=0.25mm, fill=black] (1,-3) circle (0.75mm);
\draw [line width=0.25mm, fill=black] (2,-4) circle (0.75mm);
\draw [line width=0.25mm, fill=black] (3,-5) circle (0.75mm);
\draw [line width=0.25mm, fill=white] (1,-4) circle (0.75mm);
\draw [line width=0.25mm, fill=white] (2,-5) circle (0.75mm);
\draw [line width=0.25mm, fill=white] (2,-3) circle (0.75mm);
\draw [line width=0.25mm, fill=white] (3,-4) circle (0.75mm);
\node at (2.5,-4.5) {$\frac{1}{a}$};
\node at (1.5,-3.5) {$\frac{1}{b}$};
\node at (1.5,-4.5) {\footnotesize$\frac{bc(1+a)}{1+b}$};
\node at (3.6,-4.5) {\rotatebox{90}{$d(1+a)$}};
\node at (2.8,-3) {$\frac{ae(1+b)}{1+a}$};
\node at (1.35,-2.5) {$\frac{fb}{1+b}$};
\node at (0.4,-3.25) {\rotatebox{90}{$g(1+b)$}};
\node at (0.5,-4.6) {$h$};
\node at (2,-5.6) {$\frac{ai}{1+a}$};
\node at (5,-4) {$\leftrightarrow$};
\fill[gray!40!white] (6,-6) rectangle (10,-2);
\path[-,font=\large, >=angle 90, line width=0.4mm]
(6,-6) edge (7.25,-5.25)
(6,-4) edge (6.75,-4.75)
(6,-2) edge (6.75,-2.75)
(8,-2) edge (7.75, -2.75)
(10,-3) edge (9.25, -4.25)
(10,-6) edge (9.25, -5.25)
(6.75, -2.75) edge (7.75, -2.75)
(8.25, -5.25) edge (7.25, -5.25)
(7.25, -5.25) edge (6.75, -4.75)
(6.75, -4.75) edge (6.75, -3.75)
(9.25, -4.25) edge (9.25, -5.25)
(9.25,-5.25) edge (8.25, -5.25)
(8.25, -5.25) edge (8.25, -4.25)
(8.25, -4.25) edge (7.75, -3.75)
(7.75, -3.75) edge (6.75, -3.75)
(6.75, -3.75) edge (6.75, -2.75)
(7.75, -3.75) edge (7.75, -2.75)
(8.25, -4.25) edge (9.25, -4.25);
\draw [line width=0.25mm, fill=black] (6.75, -2.75) circle (0.75mm);
\draw [line width=0.25mm, fill=black] (8.25, -5.25) circle (0.75mm);
\draw [line width=0.25mm, fill=black] (6.75, -4.75) circle (0.75mm);
\draw [line width=0.25mm, fill=black] (9.25, -4.25) circle (0.75mm);
\draw [line width=0.25mm, fill=black] (7.75, -3.75) circle (0.75mm);
\draw [line width=0.25mm, fill=white] (7.25, -5.25) circle (0.75mm);
\draw [line width=0.25mm, fill=white] (6.75, -3.75) circle (0.75mm);
\draw [line width=0.25mm, fill=white] (7.75, -2.75) circle (0.75mm);
\draw [line width=0.25mm, fill=white] (8.25, -4.25) circle (0.75mm);
\draw [line width=0.25mm, fill=white] (9.25, -5.25) circle (0.75mm);
\node at (8.75,-4.75) {$a$};
\node at (7.25,-3.25) {$\frac{1}{b}$};
\node at (7.5,-4.5) {$\frac{bc}{1+b}$};
\node at (9.67,-4.5) {$d$};
\node at (8.8,-3) {$e(1+b)$};
\node at (7.25,-2.4) {$\frac{fb}{1+b}$};
\node at (6.37,-3.4) {\rotatebox{90}{$g(1+b)$}};
\node at (6.5,-5.15) {$h$};
\node at (8,-5.6) {$i$};
\node at (11,-4) {$\leftrightarrow$};
\fill[gray!40!white] (12,-6) rectangle (16,-2);
\path[-,font=\large, >=angle 90, line width=0.4mm]
(12,-6) edge (13,-5)
(12,-4) edge (13,-4)
(12,-2) edge (13,-3)
(14,-2) edge (14,-3)
(16,-3) edge (15,-4)
(16,-6) edge (15,-5)
(14,-5) edge (13,-5)
(13,-4) edge (13,-5)
(13,-4) edge (13,-3)
(14,-3) edge (13,-3)
(14,-3) edge (14,-4)
(14,-4) edge (13,-4)
(14,-4) edge (14,-5)
(15,-5) edge (14,-5)
(15,-5) edge (15,-4)
(14,-4) edge (15,-4);
\draw [line width=0.25mm, fill=white] (13,-5) circle (0.75mm);
\draw [line width=0.25mm, fill=white] (13,-3) circle (0.75mm);
\draw [line width=0.25mm, fill=white] (14,-4) circle (0.75mm);
\draw [line width=0.25mm, fill=white] (15,-5) circle (0.75mm);
\draw [line width=0.25mm, fill=black] (13,-4) circle (0.75mm);
\draw [line width=0.25mm, fill=black] (14,-5) circle (0.75mm);
\draw [line width=0.25mm, fill=black] (14,-3) circle (0.75mm);
\draw [line width=0.25mm, fill=black] (15,-4) circle (0.75mm);
\node at (14.5,-4.5) {$a$};
\node at (13.5,-3.5) {$b$};
\node at (13.5,-4.5) {$c$};
\node at (15.5,-4.5) {$d$};
\node at (14.8,-3) {$e$};
\node at (13.35,-2.5) {$f$};
\node at (12.5,-3.4) {$g$};
\node at (12.5,-4.6) {$h$};
\node at (14,-5.6) {$i$};
\end{tikzpicture}
\end{center}

\end{proof}

Finally, we can prove the last part of Theorem~\ref{thm:Te}.
\begin{proof}
Consider one face in our expanded plabic network.  We can either add $p,-p,q,-q,r$, and $-r$ as described above, or we can add $P,-P,Q,-Q,R$, and $-R$ where $P,-P$ are between strings 2 and 3, $Q,-Q$ are between strings 1 and 2, and $R,-R$ are between strings 2 and 3 so that moving $P$ around the cylinder by square moves until it's below $-P$ is equivalent to applying $T_e^{(2,3)}$, modulo gauge transformations, then moving $Q$ around the cylinder by square moves until it's below $P$ is equivalent to applying $T_e^{(1,2)}$, modulo gauge transformations, and finally moving $R$ around the cylinder by square moves until it's below $Q$ is equivalent to applying $T_e^{(2,3)}$ a second time, modulo gauge transformations.  If we add $P,Q,R$ to our network, to a face, the face looks as follows:
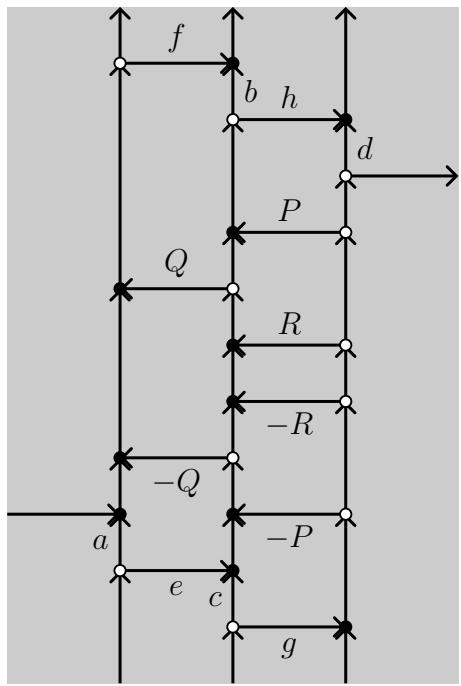
\begin{center}
\begin{tikzpicture}[scale=0.75]
\fill[gray!40!white] (0,0) rectangle (8,12);
\path[->,font=\large, >=angle 90, line width=0.4mm]
(0,3) edge (2,3)
(6,9) edge (8,9)
(2,0) edge (2,2)
(2,2) edge node[left] {$a$} (2,3)
(2,3) edge (2,4)
(2,4) edge (2,7)
(2,7) edge (2,11)
(2,11) edge (2,12)
(4,0) edge (4,1)
(4,1) edge node[left] {$c$} (4,2)
(4,2) edge (4,3)
(4,3) edge (4,4)
(4,4) edge (4,5)
(4,5) edge (4,6)
(4,6) edge (4,7)
(4,7) edge (4,8)
(4,8) edge (4,10)
(4,10) edge node[right] {$b$} (4,11)
(4,11) edge (4,12)
(6,0) edge (6,1)
(6,1) edge (6,3)
(6,3) edge (6,5)
(6,5) edge (6,6)
(6,6) edge (6,8)
(6,8) edge (6,9)
(6,9) edge node[right] {$d$} (6,10)
(6,10) edge (6,12)
(2,2) edge node[below] {$e$} (4,2)
(4,10) edge node[above] {$h$} (6,10)
(2,11) edge node[above] {$f$} (4,11)
(4,1) edge node[below] {$g$} (6,1);
\path[<-,font=\large, >=angle 90, line width=0.4mm]
(4,3) edge node[below] {$-P$} (6,3)
(4,5) edge node[below] {$-R$} (6,5)
(4,6) edge node[above] {$R$} (6,6)
(4,8) edge node[above] {$P$} (6,8)
(2,4) edge node[below] {$-Q$} (4,4)
(2,7) edge node[above] {$Q$} (4,7);
\draw [line width=0.25mm, fill=white] (2,2) circle (1mm);
\draw [line width=0.25mm, fill=white] (2,11) circle (1mm);
\draw [line width=0.25mm, fill=white] (4,1) circle (1mm);
\draw [line width=0.25mm, fill=black] (4,3) circle (1mm);
\draw [line width=0.25mm, fill=black] (4,5) circle (1mm);
\draw [line width=0.25mm, fill=black] (4,6) circle (1mm);
\draw [line width=0.25mm, fill=black] (4,8) circle (1mm);
\draw [line width=0.25mm, fill=white] (4,10) circle (1mm);
\draw [line width=0.25mm, fill=white] (6,5) circle (1mm);
\draw [line width=0.25mm, fill=white] (6,8) circle (1mm);
\draw [line width=0.25mm, fill=white] (6,9) circle (1mm);
\draw [line width=0.25mm, fill=black] (2,3) circle (1mm);
\draw [line width=0.25mm, fill=black] (2,4) circle (1mm);
\draw [line width=0.25mm, fill=white] (6,3) circle (1mm);
\draw [line width=0.25mm, fill=black] (2,7) circle (1mm);
\draw [line width=0.25mm, fill=white] (6,6) circle (1mm);
\draw [line width=0.25mm, fill=black] (4,2) circle (1mm);
\draw [line width=0.25mm, fill=white] (4,4) circle (1mm);
\draw [line width=0.25mm, fill=white] (4,7) circle (1mm);
\draw [line width=0.25mm, fill=black] (4,11) circle (1mm);
\draw [line width=0.25mm, fill=black] (6,1) circle (1mm);
\draw [line width=0.25mm, fill=black] (6,10) circle (1mm);
\end{tikzpicture}
\end{center}

Let's look at just the upper half of this picture.  Notice this is the same as the second step in the 8-cycle in Lemma~\ref{lemma:8-cycle}.  Starting with $p,q$, and $r$ inserted into a face, the 8-cycle tells us that using a square move to change to a network where the upper half of that face looks like the one with $P,Q$, and $R$, although we do not know that the edges have the right weights, and then pushing each of those edges up one face on the cylinder, is the same as pushing $p,q$, and $r$ up one face on the cylinder, and then using a square move to change to a network that looks like one with $P,Q$, and $R$ pushed up one face on the cylinder.

Suppose we move $p,q$, and $r$ all the way around the cylinder (performing $T_e^{(1,2)}\circ T_e^{(2,3)}\circ T_e^{(1,2)}$, modulo gauge transformations) and then do a square move to get to a network that looks like one with $P,Q$, and $R$ pushed around the cylinder.  The previous paragraph tells us that instead, we could push $p,q$, and $r$ around the cylinder except for one face, then do a square move to get to a network that looks like one with $P,Q$, and $R$ pushed around the cylinder except for one face, then push the three edges up the last face.  We can keep doing this square move to get a face that looks like one with $P,Q$, and $R$ begin pushed around the cylinder one face earlier until it's the first thing we do.  Doing this move first gives us values on the edges that are the same as those after we push them around the cylinder, as the values of $p,q$, and $r$ are the same before and after being pushed around the cylinder.  By Lemma~\ref{lemma:unique-p} this means that the values are the $P,Q$, and $R$ defined above, and further, that $T_e^{(1,2)}\circ T_e^{(2,3)}\circ T_e^{(1,2)}=T_e^{(2,3)}\circ T_e^{(1,2)}\circ T_e^{(2,3)}$.
\end{proof}

We now turn our attention to $T_f$ and Theorem~\ref{thm:Tf}.

\begin{thm}\label{thm:face-prime}
Applying $T_e$ to the edge weights defined in Lemma~\ref{lemma:face-to-edge}, then turning the graph back into a face weighted graph gives the following weights:
\begin{align*}
a_i'&=\frac{\widehat{\lambda}_j(a,b,c)}{\widehat{\lambda}_p(a,b,c)\left(\prod_{\substack{b_r\text{ associated to }s\\\text{where }p\leq s<j}} b_r\right)}&&\hspace{-1.3in}\text{where }a_i\text{ is associated to }j,a_{i-1}\text{ is associated to }p\\
b_i'&=\frac{\widehat{\lambda}_q(a,b,c)}{\widehat{\lambda}_j(a,b,c)\left(\prod_{\substack{a_r\text{ associated to }s\\\text{where }j< s\leq q}} a_r\right)}&&\hspace{-1.3in}\text{where }b_i\text{ is associated to }j, b_{i+1}\text{ is associated to }q\\
c_i'&=\frac{c_i\widehat{\lambda}_{i-1}(a,b,c)\left(\prod_{\substack{a_r\text{ associated to }s\\\text{where }i\leq s\leq i+1}} a_r\right)\left(\prod_{\substack{b_r\text{ associated to }s\\\text{where }i-1\leq s\leq i}} b_r\right)}{\widehat{\lambda}_{i+1}(a,b,c)}\\
t'&=t
\end{align*}
\end{thm}
\begin{proof}
The formulas for $a_i', b_i'$, and $c_i'$ follow from Theorem~\ref{thm:tilde-vars}.  Note there is no need to use the formulas for $x'_i,y'_i$, and $z'$, as these arise from the formulas for $\widetilde{x}_i, \widetilde{y}_i$ by gauge transformations, which do not affect the face weights.  Because of the way we chose our trail, none of the gauge transformations will affect the weight of the trail.  Thus the trail weight remains the same.
\end{proof}

Now we prove Theorem~\ref{thm:Tf}.

\begin{proof}
By Theorem~\ref{thm:face-prime}, these properties are inherited from Theorem~\ref{thm:Te}.
\end{proof}

\section{Spider Web Quiver Proofs}\label{sec:ClAlgPfs}

\subsection{Proof of Proposition~\ref{prop:mutation-seq}}\label{sec:proof-mutation-seq}

Let $Q_r=\mu_r\mu_{r-1}...\mu_1(Q)$.  Lemmas~\ref{lemma:edges-middle-middle} through~\ref{lemma:edges-outer-inner} may be proven by induction.

\begin{lemma}\label{lemma:edges-middle-middle}
For $r<n-2$, the edges of $Q_r$ between vertices in the middle circle are:
\begin{enumerate}[(1)]
\item a possibly empty directed path $1\to 2\to...\to r$,
\item a directed path $r+1\to r+2\to...\to n$,
\item an oriented triangle $r\to n\to r+1\to r$.
\end{enumerate}
\end{lemma}

Let $A$ be the set of vertices in the middle circle that have edges to the outer circle.  Let $A'=\{\text{vertices }i\ |\ i+1\in A\}$.

\begin{lemma}\label{lemma:edges-middle-outer-1inA}
Suppose $1\in A$.  Then for $r<n-2$ the edges between the middle and outer circle in $Q_r$ are:
\begin{enumerate}[(1)]
\item (the largest vertex on the outer circle originally connected to a vertex on the middle circle) $\to1\to 1^-$,
\item $r\to$ (the largest vertex on the outer circle that originally had an edge to a vertex between 1 and $r$) $\to$ (the largest vertex strictly between 1 and $r$ in $A'$) $\to$ (the next largest vertex along the outer circle originally connected to a vertex on the middle circle) $\to$ (the next largest vertex strictly between 1 and $r$ in $A'$) $\to...\to$ (the smallest vertex strictly between 1 and $r$ in $A'$) $\to$ (the next largest vertex along the outer circle originally connected to a vertex on the middle circle) $\to 1$,
\item $n\to$ (the largest vertex on the outer circle that originally had an edge to a vertex between 1 and $n-1$) $\to$ (the largest vertex in $A$ less than $n$) $\to$ (the next largest vertex along the outer circle originally connected to a vertex on the middle circle) $\to$ (the next largest vertex in $A$) $\to...\to$ (the largest vertex of $A$ greater than $r+1$) $\to$ (the next largest vertex along the inner circle originally connected to a vertex on the middle circle) $\to r+1$.
\end{enumerate}
\end{lemma}

\begin{lemma}\label{lemma:edges-middle-outer-1notA}
Suppose $1\not\in A$.  For $r<n-2$, if $r$ is smaller that the smallest vertex of $A$, then the edges between the middle and outer circle in $Q_r$ are:
\begin{enumerate}[(1)]
\item the same edges as in the original quiver, that is, (the largest vertex on the outer circle originally connected to a vertex on the middle circle) $\to$ (the largest vertex in $A$) $\to$ (the next largest vertex along the outer circle originally connected to a vertex on the middle circle) $\to$ (the next largest vertex in $A$) $\to$ (the next largest vertex along the outer circle originally connected to a vertex on the middle circle) $\to...\to$ (the smallest vertex in $A$) $\to$ (the largest vertex along the outer circle originally connected to a vertex on the middle circle).
\end{enumerate}
Otherwise the edges between the middle and outer circle in $Q_r$ are:
\begin{enumerate}[(1)]
\item (the largest vertex on the outer circle originally connected to a vertex on the middle circle) $\to$ (the smallest vertex in $A$ along the middle circle),
\item $r\to$ (the largest vertex on the outer circle that originally had an edge to a vertex between 1 and $r$) $\to$ (the largest vertex strictly between 1 and $r$ in $A'$) $\to$ (the next largest vertex along the outer circle originally connected to a vertex on the middle circle) $\to$ (the next largest vertex strictly between 1 and $r$ in $A'$) $\to...\to$ (the smallest vertex marked in $A'$),
\item $n\to$ (the largest vertex on the outer circle that originally had an edge to a vertex between 1 and $n-1$) $\to$ (the largest vertex in $A$ less than $n$) $\to$ (the next largest vertex along the outer circle originally connected to a vertex on the middle circle) $\to$ (the next largest vertex in $A$) $\to...\to$ (the largest vertex of $A$ greater than $r+1$) $\to$ (the next largest vertex along the inner circle originally connected to a vertex on the middle circle) $\to r+1$.
\end{enumerate}
\end{lemma}

Let $B$ be the set of vertices in the middle circle that have edges to the inner circle.  Let $B'=\{\text{vertices }i\ |\ i+1\in B\}$.

\begin{lemma}\label{lemma:edges-middle-inner-1inB}
Suppose $1\in B$.  Then for $r<n-2$ the edges between the middle and inner circle in $Q_r$ are:
\begin{enumerate}[(1)]
\item (the largest vertex on the inner circle originally connected to a vertex on the middle circle) $\to1\to 1^+$,
\item $r\to$ (the largest vertex on the inner circle that originally had an edge to a vertex between 1 and $r$) $\to$ (the largest vertex strictly between 1 and $r$ in $B'$) $\to$ (the next largest vertex along the inner circle originally connected to a vertex on the middle circle) $\to$ (the next largest vertex strictly between 1 and $r$ in $B'$) $\to...\to$ (the smallest vertex strictly between 1 and $r$ in $B'$) $\to$ (the next largest vertex along the inner circle originally connected to a vertex on the middle circle) $\to 1$,
\item $n\to$ (the largest vertex on the inner circle that originally had an edge to a vertex between 1 and $n-1$) $\to$ (the largest vertex in $B$ less than $n$) $\to$ (the next largest vertex along the inner circle originally connected to a vertex on the middle circle) $\to$ (the next largest vertex in $B$) $\to...\to$ (the largest vertex of $B$ greater than $r+1$) $\to$ (the next largest vertex along the inner circle originally connected to a vertex on the middle circle) $\to r+1$.
\end{enumerate}
\end{lemma}

\begin{lemma}\label{lemma:edges-middle-inner-1notB}
Suppose $1\not\in B$.  For $r<n-2$, if $r$ is smaller that the smallest vertex of $B$, then the edges between the middle and inner circle in $Q_r$ are:
\begin{enumerate}[(1)]
\item the same edges as in the original quiver, that is, (the largest vertex on the inner circle originally connected to a vertex on the middle circle) $\to$ (the largest vertex in $B$) $\to$ (the next largest vertex along the inner circle originally connected to a vertex on the middle circle) $\to$ (the next largest vertex in $B$) $\to$ (the next largest vertex along the inner circle originally connected to a vertex on the middle circle) $\to...\to$ (the smallest vertex in $B$) $\to$ (the largest vertex along the inner circle originally connected to a vertex on the middle circle).
\end{enumerate}
Otherwise the edges between the middle and inner circle in $Q_r$ are:
\begin{enumerate}[(1)]
\item (the largest vertex on the inner circle originally connected to a vertex on the middle circle) $\to$ (the smallest vertex in $B$ along the middle circle),
\item $r\to$ (the largest vertex on the inner circle that originally had an edge to a vertex between 1 and $r$) $\to$ (the largest vertex strictly between 1 and $r$ in $B'$) $\to$ (the next largest vertex along the inner circle originally connected to a vertex on the middle circle) $\to$ (the next largest vertex strictly between 1 and $r$ in $B'$) $\to...\to$ (the smallest vertex marked in $B'$),
\item $n\to$ (the largest vertex on the inner circle that originally had an edge to a vertex between 1 and $n-1$) $\to$ (the largest vertex in $B$ less than $n$) $\to$ (the next largest vertex along the inner circle originally connected to a vertex on the middle circle) $\to$ (the next largest vertex in $B$) $\to...\to$ (the largest vertex of $B$ greater than $r+1$) $\to$ (the next largest vertex along the inner circle originally connected to a vertex on the middle circle) $\to r+1$.
\end{enumerate}
\end{lemma}

\begin{lemma}\label{lemma:edges-outer-outer}
For $r<n-2$, the edges of $Q_r$ between vertices in the outer circle are:
\begin{enumerate}[(1)]
\item The directed circle $1^-\to2^- \to...\to n^- \to 1^-$,
\item The edge $1^- \to$ (the largest vertex on the outer circle originally connected to a vertex on the middle circle), if $r\geq$ (the smallest vertex in $A$).
\end{enumerate}
\end{lemma}

\begin{lemma}\label{lemma:edges-inner-inner}
For $r<n-2$, the edges of $Q_r$ between vertices in the inner circle are:
\begin{enumerate}[(1)]
\item The directed circle $1^+ \to 2^+ \to...\to n^+ \to 1^+$.
\item The edge $1^+ \to$ (the largest vertex on the inner circle originally connected to a vertex on the middle circle), if $r\geq$ (the smallest vertex in $B$).
\end{enumerate}
\end{lemma}

\begin{lemma}\label{lemma:edges-outer-inner}
For $r<n-2$, the edges of $Q_r$ between vertices between the outer and inner circle are:
\begin{enumerate}[(1)]
\item $1^-\to$ (the largest vertex on the inner circle originally connected to a vertex on the middle circle) and $1^+\to$ (the largest vertex on the outer circle originally connected to a vertex on the middle circle), if the smallest vertex in $B$ is the smallest vertex in $A$,
\item $1^-\to$ (the largest vertex on the inner circle originally connected to a vertex on the middle circle), if the smallest vertex in $A$ is less than the smallest vertex in $B$,
\item $1^+\to$ (the largest vertex on the outer circle originally connected to a vertex on the middle circle), if the smallest vertex in $B$ is less than the smallest vertex in $A$.
\end{enumerate}
\end{lemma}

\begin{lemma}\label{lemma:Qn-2}
If $s_{n-1,n}$ is the transposes vertices $n-1$ and $n$, then $s_{n-1,n}\mu_n\mu_{n-1}(Q_{n-2})=Q_{n-2}$.
\end{lemma}
\begin{proof}
We can check all the cases.
\end{proof}

The proof of Proposition~\ref{prop:mutation-seq} follows easily.

\begin{proof}
\begin{align*}
\mu_1\mu_2...\mu_{n-2}s_{n-1,n}\mu_n\mu_{n-1}...\mu_1(Q)&=\mu_1\mu_2...\mu_{n-2}s_{n-1,n}\mu_n\mu_{n-1}(Q_{n-2})\\
&=\mu_1\mu_2...\mu_{n-2}(Q_{n-2})\text{ by Lemma~\ref{lemma:Qn-2}}\\
&=Q\text{ as }\mu_1\mu_2...\mu_{n-2}\text{ is }(\mu_{n-2}\mu_{n-3}...\mu_1)^{-1}
\end{align*}
\end{proof}

\subsection{Proof of Theorem~\ref{thm:x-vars}}\label{sec:proof-x-vars}

Recall the definitions of $A$ and $B$ from~\ref{sec:proof-mutation-seq}.

Let $x_{i,A}:=\begin{cases} x_\alpha & \text{if }i\text{ is the smallest element of }A, \\ 1 & \text{otherwise}.\end{cases}$

Let $x_{i,B}:=\begin{cases} x_\beta & \text{if }i\text{ is the smallest element of }B, \\ 1 & \text{otherwise}.\end{cases}$

We will define $x_0:=x_n$ for ease of notation.

\begin{prop}\label{prop:x-vars-half}
For a quiver $Q$ as described at the beginning of the section, $$\mu_{n-2}...\mu_1(x_i)=\frac{\sum_{j=0}^i x_{[j]^-}x_{[j]^+}\left(\prod_{k=0}^{j-1} x_k\right)\left(\prod_{k=j+1}^i x_{k+1}x_{k,A}x_{k,B}\right)}{\prod_{k=1}^i x_k},$$ for $1\leq i\leq n-2$.
\end{prop}
\begin{proof} Notice that $\mu_{n-2}...\mu_1(x_i)=\mu_i...\mu_1(x_i)$.  Using this and Lemmas~\ref{lemma:edges-middle-middle} through~\ref{lemma:edges-middle-inner-1notB}, we see $$\mu_{n-2}...\mu_1(x_1)=\frac{x_nx_{[1]^-}x_{[1]^+}+x_2x_{1,A}x_{1,B}}{x_1},$$ $$\mu_{n-2}...\mu_1(x_i)=\frac{x_nx_{[i]^-}x_{[i]^+}+\mu_{n-2}...\mu_1(x_{i-1})x_{i+1}x_{i,A}x_{i,B}}{x_i}.$$
Then the result follows by induction.
\end{proof}

We now have the proof of Theorem~\ref{thm:x-vars}.

\begin{proof}
First consider $i=n,n-1$.  In this case, $\mu_1\mu_2...\mu_{n-2}s_{n-1,n}\mu_n\mu_{n-1}...\mu_1(x_i)=s_{n-1,n}\mu_n\mu_{n-1}...\mu_1(x_i)$.  We can compute this using Lemmas~\ref{lemma:edges-middle-middle} through~\ref{lemma:edges-middle-inner-1notB} and Proposition~\ref{prop:x-vars-half}.  We have the following:
$$\tau(x_{n-1})=\frac{x_{[n-1]^-}x_{[n-1]^+}+\mu_{n-2}...\mu_1(x_{n-2})x_{n-1,A}x_{n-1,B}}{x_n}$$
$$\tau(x_n)=\frac{x_{[n-1]^-}x_{[n-1]^+}+\mu_{n-2}...\mu_1(x_{n-2})x_{n-1,A}x_{n-1,B}}{x_{n-1}}$$
Notice that for any $i$, $x_{[i]^-_*}=x_{[i]^-}\left(\prod_{k=i+1}^j x_{k,A}\right),$ as long as there is an element of $A$ in $\{1,...,j\}$.  Similarly, $x_{[i]^+_*}=x_{[i]^+}\left(\prod_{k=i+1}^j x_{k,B}\right),$ as long as there is an element of $B$ in $\{1,...,j\}$.  Substituting for $\mu_{n-2}...\mu_1(x_{n-2})$ in each of the above cases, we have what we want.

Now consider $i<n-2$.  Because mutation is an involution, we know $$\mu_i...\mu_1\mu_1\mu_2...\mu_{n-2}s_{n-1,n}\mu_n\mu_{n-1}...\mu_1(x_i)=\mu_{i+1}...\mu_{n-2}s_{n-1,n}\mu_n\mu_{n-1}...\mu_1(x_i).$$  Since mutating at a vertex only affects the variable at that vertex, this is the same as $\mu_i\mu_{i-1}...\mu_1(x_i)$.  From Proposition~\ref{prop:x-vars-half}, we know the formula for this expression.  Thus, all we need to do is show, starting with the cluster variables $\widetilde{x}_i=x_i\overline{x}$, mutating at 1 through $i$ gives us back the formulas in Proposition~\ref{prop:x-vars-half}.
\begin{align*}
\mu_i...\mu_1(\widetilde{x}_i)&=\frac{\sum_{j=0}^i x_{[j]^-}x_{[j]^+}\left(\prod_{k=0}^{j-1} x_k\overline{x}\right)\left(\prod_{k=j+1}^i x_{k+1}\overline{x}x_{k,A}x_{k,B}\right)}{\prod_{k=1}^i x_k\overline{x}}\\
&=\frac{\sum_{j=0}^i x_{[j]^-}x_{[j]^+}\left(\prod_{k=0}^{j-1} x_k\right)\left(\overline{x}\right)^j\left(\prod_{k=j+1}^i x_{k+1}x_{k,A}x_{k,B}\right)\left(\overline{x}\right)^{i-j}}{\left(\prod_{k=1}^i x_k\right)\left(\overline{x}\right)^i}\\
&=\frac{\sum_{j=0}^i x_{[j]^-}x_{[j]^+}\left(\prod_{k=0}^{j-1} x_k\right)\left(\prod_{k=j+1}^i x_{k+1}x_{k,A}x_{k,B}\right)}{\left(\prod_{k=1}^i x_k\right)}\\
\end{align*}
\end{proof}

\subsection{Proof of Theorem~\ref{thm:y-vars}}\label{sec:proof-y-vars}

\begin{lemma}
For $1< r\leq n-2$, the $y$-variables for the vertices in the middle circle of the quiver after $r$ mutations are:
\begin{enumerate}[(1)]
\item If $i<r$, then $\displaystyle y_i'=\frac{y_{i+1}\sum_{j=0}^{i-1}\prod_{k=0}^j y_k}{\sum_{j=0}^{i+1}\prod_{k=0}^j y_k}$
\item $\displaystyle y_r'=\frac{\sum_{j=0}^{r-1}\prod_{k=0}^j y_k}{\prod_{k=0}^r y_k}$
\item $\displaystyle y_{r+1}'=\frac{\prod_{k=0}^{r+1} y_k}{\sum_{j=0}^r \prod_{k=0}^j y_k}$
\item If $r+1<m<n$, then $y_m'=y_m$
\item $\displaystyle y_n'=y_n\sum_{j=0}^r\prod_{k=0}^j y_k$
\end{enumerate}
\end{lemma}
\begin{proof}
This can be shown by induction.
\end{proof}

\begin{lemma}
After mutating all $n$ vertices and performing the transposition $s_{n-1,n}$, the $y$-variables for the vertices in the middle circle of the quiver are:
\begin{enumerate}[(1)]
\item If $i<n-2$, then $\displaystyle y_i'=\frac{y_{i+1}\sum_{j=0}^{i-1}\prod_{k=0}^j y_k}{\sum_{j=0}^{i+1}\prod_{k=0}^j y_k}$
\item $\displaystyle y_{n-2}'=\frac{y_{n-1}\left(\sum_{j=0}^{n-3}\prod_{k=0}^j y_k\right)\left(1+y_n\sum_{j=0}^{n-2}\prod_{k=0}^j y_k\right)}{\sum_{j=0}^{n-1}\prod_{k=0}^j y_k}$
\item $\displaystyle y_{n-1}'=\frac{1}{y_n\sum_{j=0}^{n-2}\prod_{k=0}^j y_k}$
\item $\displaystyle y_n'=\frac{\sum_{j=0}^{n-2}\prod_{k=0}^j y_k}{\prod_{k=0}^{n-1} y_k}$
\end{enumerate}
\end{lemma}
\begin{proof}
These values can be easily computed.
\end{proof}

\begin{lemma}\label{lemma:y-center}
If $1<r\leq n-2$, then after mutating vertex $r$ for the second time, the $y$-variables in the middle circle of the quiver are:
\begin{enumerate}[(1)]
\item If $\ell<i-1$, then $$y_{\ell}=\frac{y_{\ell+1}\sum_{j=0}^{\ell-1}\prod_{k=0}^j y_k}{\sum_{j=0}^{\ell+1}\prod_{k=0}^j y_k}$$
\item $\displaystyle y_{r-1}=\frac{\left(y_r\sum_{j=0}^{r-2}\prod_{k=0}^j y_k\right)\left(\left(\sum_{j=r+1}^n \prod_{k=r+1}^{j-1} y_k\right)+\left(\prod_{k=r+1}^n y_k\right)\left(\sum_{j=0}^{r-1}\prod_{k=0}^j y_k\right)\right)}{\sum_{j=0}^{n-1}\prod_{k=0}^j y_k}$
\item $\displaystyle y_r=\frac{\sum_{j=0}^{n-1}\prod_{k=0}^j y_k}{\left(y_{r+1}\sum_{j=0}^{r-1}\prod_{k=0}^j y_k\right)\left(\left(\sum_{j=r+2}^n \prod_{k=r+2}^{j-1} y_k\right)+\left(\prod_{k=r+2}^n y_k\right)\left(\sum_{j=0}^r\prod_{k=0}^j y_k\right)\right)}$
\item If $i<m<n$, then $$y_m=\frac{\sum_{j=m}^n\prod_{k=m}^{j-1} y_k+\left(\prod_{k=m}^n y_k\right)\left(\sum_{j=0}^{m-2}\prod_{k=0}^j y_k\right)}{y_{m+1}\left(\sum_{j=m+2}^n\prod_{k=m+2}^{j-1} y_k+\left(\prod_{k=m+2}^n y_k\right)\left(\sum_{j=0}^m\prod_{k=0}^j y_k\right)\right)}$$
\item $\displaystyle y_n=\frac{\left(\sum_{j=0}^{r-1}\prod_{k=0}^j y_k\right)\left(1+y_n\sum_{j=0}^{n-2}\prod_{k=0}^j y_k\right)}{\left(\prod_{k=0}^r y_k\right)\left(\sum_{j=r+1}^n\prod_{k=r+1}^{j-1} y_k+\left(\prod_{k=r+1}^n y_k\right)\left(\sum_{j=0}^{r-1}\prod_{k=0}^j y_k\right)\right)}$
\end{enumerate}
\end{lemma}
\begin{proof}
This can be shown by induction.
\end{proof}

For $1\leq i\leq n-2$, let $$\widetilde{y_i}=\frac{\prod_{k=0}^i y_k}{\sum_{j=0}^{i-1}\prod_{k=0}^j y_k},$$ $$\widehat{y_i}=\frac{y_{i+1}\left(\sum_{j=0}^{i-1}\prod_{k=0}^j y_k\right)\left(\left(\sum_{j=i+2}^n\prod_{k=i+2}^{j-1} y_k\right)+\left(\prod_{i+2}^n y_k\right)\left(\sum_{j=0}^i\prod_{k=0}^j y_k\right)\right)}{\sum_{j=0}^{n-1}\prod_{k=0}^j y_k}.$$  That is, $\widetilde{y_i}$ is the value of $y_i'$ before mutating at $i$ for the first time, and $\widehat{y_i}$ is the value of $y_i'$ before mutating at $i$ for the second time.  Let $$y_{n-1}^*=\frac{\prod_{k=0}^{n-1} y_k}{\sum_{j=0}^{n-2}\prod_{k=0}^j y_k},\ y_n^*=y_n\sum_{j=0}^{n-2}\prod_{k=0}^j y_k,$$ so $y_{n-1}^*$ is the value of $y_{n-1}'$ before mutating at $n-1$ and $y_n^*$ is the value of $y_n'$ before mutating at $n$.

\begin{lemma}\label{lemma:y-A}
Suppose $n\in A$.  Then after completing the mutation sequence, we have the following $y$-variables for the outer circle.
\begin{enumerate}[(1)]
\item If there are no edges between $i^-$ and the middle circle, then $y_{i^-}'=y_{i^-}$.

\item If $i$ is maximal so that there are edges between $i^-$ and the middle circle, then $$y_{i^-}=\frac{y_{i^-}\left(1+\widehat{y_{d_1}}\right)}{\left(1+\widetilde{y_{d_1}}^{-1}\right)}.$$

\item If $i$ is second largest so that there are edges between $i^-$ and the middle circle, then $$y_{i^-}=\frac{y_{i^-}(1+y_{n-1}^*)\prod_{k=d_i}^{n-2}\left(1+\widetilde{y_k}\right)}{(1+y_n^{*-1})\prod_{k=d_i}^{n-2}\left(1+\widehat{y_k}^{-1}\right)}.$$

\item Otherwise, $$y_{i^-}=\frac{y_{i^-}\prod_{k=d_i}^{e_i-1}\left(1+\widetilde{y_k}\right)}{\prod_{k=d_i}^{e_i-1}\left(1+\widehat{y_k}^{-1}\right)}.$$
\end{enumerate}
\end{lemma}
\begin{proof}
We can prove this by using Lemmas~\ref{lemma:edges-middle-outer-1inA} and~\ref{lemma:edges-middle-outer-1notA} to know the arrows in $Q_r$ and applying the $y$-mutation rules.
\end{proof}

\begin{lemma}\label{lemma:y-not-A}
Suppose $n\not\in A$.  Then after completing the mutation sequence, we have the following $y$-variables for the outer circle.
\begin{enumerate}[(1)]
\item If there are no edges between $i^-$ and the middle circle, then $y_{i^-}'=y_{i^-}$.

\item If $i$ is maximal so that there are edges between $i^-$ and the middle circle, then $$y_{i^-}=\frac{y_{i^-}(1+y_{n-1}^*)\left(1+\widehat{y_{d_1}}\right)\prod_{k=e_i}^{n-2}\left(1+\widetilde{y_k}\right)}{(1+y_n^{*-1})\left(1+\widetilde{y_{d_1}}^{-1}\right)\prod_{k=e_i}^{n-2}\left(1+\widehat{y_k}^{-1}\right)}.$$

\item Otherwise, $$y_{i^-}=\frac{y_{i^-}\prod_{k=d_i}^{e_i-1}\left(1+\widetilde{y_k}\right)}{\prod_{k=d_i}^{e_i-1}\left(1+\widehat{y_k}^{-1}\right)}.$$
\end{enumerate}
\end{lemma}
\begin{proof}
We can prove this by using Lemmas~\ref{lemma:edges-middle-outer-1inA} and~\ref{lemma:edges-middle-outer-1notA} to know the arrows in $Q_r$ and applying the $y$-mutation rules.
\end{proof}

\begin{lemma}\label{lemma:y-B}
Suppose $n\in B$.  Then after completing the mutation sequence, we have the following $y$-variables for the inner circle.
\begin{enumerate}[(1)]
\item If there are no edges between $i^+$ and the middle circle, then $y_{i^+}'=y_{i^+}$.

\item If $i$ is maximal so that there are edges between $i^+$ and the middle circle, then $$y_{i^+}=\frac{y_{i^+}\left(1+\widehat{y_{f_1}}\right)}{\left(1+\widetilde{y_{f_1}}^{-1}\right)}.$$

\item If $i$ is second largest so that there are edges between $i^+$ and the middle circle, then $$y_{i^+}=\frac{y_{i^+}(1+y_{n-1}^*)\prod_{k=f_i}^{n-2}\left(1+\widetilde{y_k}\right)}{(1+y_n^{*-1})\prod_{k=f_i}^{n-2}\left(1+\widehat{y_k}^{-1}\right)}.$$

\item Otherwise, $$y_{i^+}=\frac{y_{i^+}\prod_{k=f_i}^{g_i-1}\left(1+\widetilde{y_k}\right)}{\prod_{k=f_i}^{g_i-1}\left(1+\widehat{y_k}^{-1}\right)}.$$
\end{enumerate}
\end{lemma}
\begin{proof}
We can prove this by using Lemmas~\ref{lemma:edges-middle-inner-1inB} and~\ref{lemma:edges-middle-inner-1notB} to know the arrows in $Q_r$ and applying the $y$-mutation rules.
\end{proof}

\begin{lemma}\label{lemma:y-not-B}
Suppose $n\not\in B$.  Then after completing the mutation sequence, we have the following $y$-variables for the inner circle.
\begin{enumerate}[(1)]
\item If there are no edges between $i^+$ and the middle circle, then $y_{i^+}'=y_{i^+}$.

\item If $i$ is maximal so that there are edges between $i^+$ and the middle circle, then $$y_{i^+}=\frac{y_{i^+}(1+y_{n-1}^*)\left(1+\widehat{y_{f_1}}\right)\prod_{k=g_i}^{n-2}\left(1+\widetilde{y_k}\right)}{(1+y_n^{*-1})\left(1+\widetilde{y_{f_1}}^{-1}\right)\prod_{k=g_i}^{n-2}\left(1+\widehat{y_k}^{-1}\right)}.$$

\item Otherwise, $$y_{i^+}=\frac{y_{i^+}\prod_{k=f_i}^{g_i-1}\left(1+\widetilde{y_k}\right)}{\prod_{k=f_i}^{g_i-1}\left(1+\widehat{y_k}^{-1}\right)}.$$
\end{enumerate}
\end{lemma}
\begin{proof}
We can prove this by using Lemmas~\ref{lemma:edges-middle-inner-1inB} and~\ref{lemma:edges-middle-inner-1notB} to know the arrows in $Q_r$ and applying the $y$-mutation rules.
\end{proof}

Now we have the proof of Theorem~\ref{thm:y-vars}.
\begin{proof}
This follows from Lemmas~\ref{lemma:y-center} through~\ref{lemma:y-not-B}.
\end{proof}

\bibliographystyle{amsplain}

\begin{thebibliography}{XXX}
\baselineskip=15pt
\bibitem{ABCGPT} N. Arkani-Hamed, J. Bourjaily, F. Cachazo, A. Goncharov, A. Postnikov, and J. Trnka, \emph{Scattering Amplitudes and the Positive Grassmannian}, arXiv:1212.5605.
\bibitem{ABCPT} N. Arkani-Hamed, J. Bourjaily, F. Cachazo, A. Postnikov, and J. Trnka, \emph{On-shell structures of MHV amplitudes beyond the planar limit}, J. High Enery Phys., {\bf2015} (2015) 179.
\bibitem{ABCTPbook} N. Arkani-Hamed, J. Bourjaily, F. Cachazo, J. Trnka, and A. Postnikov, \emph{Grassmannian Geometry of Scattering Amplitudes}, Cambridge University Press, Cambridge, UK, 2016.
\bibitem{ALT} J. Alman, C. Lian, and B. Tran, \emph{Circular planar electrical networks: Posets and positivity}, J. Combin. Theory Ser. A, {\bf132} (May 2015) 58-101.
\bibitem{BK} A. Berenstein and D. Kazhdan, \emph{Geometric and unipotent crystals}, In: Alon N., Bourgain J., Connes A., Gromov M., Milman V. (eds) Visions in Mathematics. Modern Birkh{\"a}user Classics. Birkh{\"a}user Basel, 2010.
\bibitem{CM} E.B. Curtis, D. Ingerman, and J.A. Morrow, \emph{Circular planar graphs and resistor networks}, Linear Algebra Appl., {\bf283} (1998) 115-150.
\bibitem{E} P. Etingof, \emph{Geometric crystals and set-theoretical solutions to the quantum Yang-Baxter equation}, Comm. Algebra, {\bf31} (2003), no. 4, 1961-1973.
\bibitem{FZ} S. Fomin and A. Zelevinsky, \emph{Total positivity: tests and parametrizations}, Math. Intelligencer, {\bf22} (2000), no. 1, 23Ð33.
\bibitem{GSVbook} M. Gekhtman, M. Shapiro, and A. Vainshtein, \emph{Cluster Algebras and Poisson Geometry}, Amer. Math. Soc., Providence, RI, 2010.
\bibitem{GSV2} M. Gekhtman, M. Shapiro, and A. Vainshtein, \emph{Generalized B{\"a}cklund-Darboux transformations for Coxeter-Toda flows from a cluster algebra perspective}, Acta Mathematica, {\bf206} (2011), no. 2, 245-310.
\bibitem{GSV1} M. Gekhtman, M. Shapiro, and A. Vainshtein, \emph{Poisson geometry of directed networks in an annulus}, J. Europ. Math. Soc., {\bf14} (2012) 541Ð570.
\bibitem{GS} A. Goncharov and L. Shen, \emph{Donaldson-Thomas transformations of moduli spaces of G-local systems}, Adv. Math., (2017), doi:10.1016/j.aim.2017.06.017.
\bibitem{ILP} R. Inoue, T. Lam, and P. Pylyavskyy, \emph{On the cluster nature and quantization of geometric $R$-matrices}, arXiv:1607.00722.
\bibitem{KNO} M. Kashiwara, T. Nakashima, and M. Okado, \emph{Tropical R maps and affine geometric crystals}, Represent. Theory, {\bf14} (2010) 446-509.
\bibitem{KNY} K. Kajiwara, M. Noumi, and Y. Yamada, \emph{Discrete Dynamical Systems with $W(A_{m-1}^{(1)}\times A_{n-1}^{(1)})$ Symmetry}, Lett. Math. Phys., {\bf60} (2002), no. 3, 211-219.
\bibitem{KW1} Y. Kodama and L. Williams, \emph{KP solitons, total positivity, and cluster algebras}, PNAS, {\bf108} (2011), no. 22, 8984-8989.
\bibitem{KW2} Y. Kodama and L. Williams, \emph{KP solitons and total positivity for the Grassmannian}, Inventiones mathematicae, {\bf198} (2014), no. 3, 637-699.
\bibitem{LP2} T. Lam and P. Pylyavskyy, \emph{Total positivity in loop groups, I: Whirls and curls}, Adv. Math., {\bf230} (2012), no. 3, 1222-1271.
\bibitem{LP1} T. Lam and P. Pylyavskyy, \emph{Crystals and total positivity on orientable surfaces}, Selecta Math., {\bf19} (2013), 173-235.
\bibitem{N} T. Nakanishi, \emph{Periodicities in cluster algebras and dilogarithm identities}, In: A. Skowronski and K. Yamagata (eds) Representations of Algebras and Related Topics. EMS Series of Congress Reports. European Mathematical Society, Zurich, 2011.
\bibitem{P} A. Postnikov, \emph{Total Positivity, Grassmannians, and Networks}, arXiv:math/0609764.
\end{thebibliography}

\end{document}